\documentclass[11pt]{article}
\usepackage[margin=1in, centering]{geometry}

\usepackage{mathtools}
\usepackage{amssymb}
\usepackage{amsthm}
\usepackage{xcolor}
\usepackage{bbm}
\usepackage[normalem]{ulem}
\usepackage[parfill]{parskip}
\usepackage{verbatim}

\usepackage[colorlinks=true, citecolor=blue]{hyperref}
\usepackage[capitalise]{cleveref}

\newtheorem{thm}{Theorem}[section]
\newtheorem{prop}[thm]{Proposition}
\newtheorem{lem}[thm]{Lemma}
\newtheorem{cor}[thm]{Corollary}
\newtheorem{claim}[thm]{Claim}
\newtheorem{problem}[thm]{Problem}
\newtheorem{example}[thm]{Example}
\newtheorem{remark}[thm]{Remark}
\newtheorem{observation}[thm]{Observation}
\newtheorem{question}[thm]{Question}
\newtheorem{tech-question}[thm]{Technical Question}

\newtheorem{definition}[thm]{Definition}

\newcommand{\F}{\mathbb{F}}  
  
\newcommand{\R}{\mathbb{R}}  
\newcommand{\Z}{\mathbb{Z}}
\newcommand{\N}{\mathbb{N}}  
\newcommand{\C}{\mathbb{C}}     
\DeclareMathOperator*{\E}{\mathbb{E}}  
\renewcommand{\P}{\mathbb{P}} 
\newcommand{\1}{ \mathbbm{1}}
\newcommand{\eps}{ \varepsilon}  

\newcommand{\set}[2]{\left\{#1 \: : \: #2\right\}}
\newcommand{\ip}[2]{\left\langle#1 ,#2 \right\rangle}
\newcommand{\abs}[1]{\left|#1 \right|}
\newcommand{\norm}[1]{\left\|#1 \right\|}
\newcommand{\paren}[1]{\left(#1 \right)}

\newcommand{\ind}[1]{\1 \! \! \paren{#1}}
\newcommand{\prob}[1]{\P \!  \paren{#1}}

\newcommand{\Bohr}{\textnormal{Bohr}}
\renewcommand{\hat}{\widehat}

\providecommand{\Comments}{1}
\ifnum\Comments=1
\usepackage[colorinlistoftodos,prependcaption,textsize=scriptsize]{todonotes}
\else
\usepackage[disable]{todonotes}
\fi
\newcommand{\mytodo}[1]{\ifnum\Comments=1{#1}\fi}

\newcommand{\tableoftodos}{\ifnum\Comments=1 \listoftodos[Comments/To Do's] \fi}

\newcommand{\ignore}[1]{}

\title{Strong Bounds for 3-Progressions}
\author{Zander Kelley\thanks{
Department of Computer Science, University of Illinois at Urbana-Champaign. Supported by NSF grant CCF-1814788, and CAREER award 2047310. Email: \texttt{awk2@illinois.edu}
}, 
Raghu Meka\thanks{
Department of Computer Science, University of California, Los Angeles. Supported by NSF AF 2007682. Email: \texttt{raghum@cs.ucla.edu.}}}
\date{}
\begin{document}

\maketitle

\begin{abstract}
    We show that for some constant $\beta > 0$, any subset $A$ of integers $\{1,\ldots,N\}$ of size at least $2^{-O((\log N)^\beta)} \cdot N$ contains a non-trivial three-term arithmetic progression. Previously, three-term arithmetic progressions were known to exist only for sets of size at least $N/(\log N)^{1 + c}$ for a constant $c > 0$. 

    Our approach is first to develop new analytic techniques for addressing some related questions in the finite-field setting and then to apply some analogous variants of these same techniques, suitably adapted for the more complicated setting of integers.
\end{abstract}

\section{Introduction} \label{introduction}

A 3-progression is a triple of integers of the form $(a, a + b, a + 2b)$, and it is said to be \textit{nontrivial} if $b \neq 0$. 
It is straightforward to check that a triple $(x,z,y)$ is of this form if and only if $x + y = 2z$, and that, in this case, the progression is trivial only if $x = y$.
We consider the problem of finding a nontrivial 3-progression within a set 
$A \subseteq [N] \subseteq \Z$, 
where we assume only that the set $A$ has somewhat large density inside $[N]$. 
Regarding this problem, we prove the following.

\begin{thm}\label{3-progs}
The following holds for some absolute constant exponent $\beta > 0$.
Suppose $A \subseteq [N]$ has density $\delta = |A|/N.$ Then, either $A$ contains a nontrivial 3-progression, or else
$$ \delta \leq 2^{-\Omega( \log(N)^{\beta} )} .
$$
\end{thm}
In \cite{roth53}, Roth proved a statement of the same form: that dense enough sets $A \subseteq [N]$ must contain a nontrivial 3-progression, for density threshold $\delta \approx 1 / \log \log N.$ This was improved by Heath-Brown and Szemer{\'e}di to $\delta \approx 1 / \log(N)^c$ for some small $c > 0$ \cite{heath87,szemeredi90}. This bound was further refined in the works of Bourgain, and Sanders \cite{bourgain99, bourgain08, sanders12certain} where it is shown that one can take $c = 1/2$, $c = 2/3$, and then $c = 3/4$. Sanders then obtained a density-threshold of the form $\delta \approx (\log \log N)^{6} / \log N$ \cite{sanders11}.\footnote{See also the work of Bloom and Sisask \cite{bs19}, which obtains a comparable result by some different methods which are closely related to the approach taken here.} This was further sharpened by a factor $(\log \log N)^2$ by Bloom and then again by another factor $\log \log N$ by Schoen \cite{bloom16, schoen21}. The best bound previously available is due to Bloom and Sisask \cite{bs20}, who show that a set $A \subseteq [N]$ with no 3-progressions must have density
$$\delta \leq  O \left( \frac{1}{\log(N)^{1+c}} \right)$$
for some small $c > 0$. We also refer to \cref{comparison} for a more detailed discussion of some previous approaches and how our methods relate to them.

In the other direction, it was shown by Behrend\footnote{\cite{behrend46}.  See also the works of Elkin \cite{elkin10}, Green and Wolf \cite{gw10}, and O'Bryant \cite{obryant11} for some refinements.}  that for infinitely many values $N$ there are indeed subsets $A \subseteq [N]$ of density roughly
$ \delta \approx 2^{-\log(N)^{1/2}} $
which have no nontrivial 3-progressions.

More specifically, we establish the following.
\begin{thm}\label{many-3-progs}
Suppose $A \subseteq [N]$ has density at least $2^{-d}$. Then the number of triples $(x,y,z) \in A^3$ with $x + y = 2z$ is at least
$$ 2^{-O(d^{12})} N^2 .$$
\end{thm}

Since there are only $|A| \leq N$ trivial 3-progressions, we do indeed obtain a nontrivial 3-progression, unless $\log N \leq {O(d^{12})}$. 

We also consider a similar problem where $A \subseteq \F_q^n$ is a subset of a vector space over some finite field. In this setting, we prove the following.

\begin{thm}\label{3-progs-in-finite-fields}
Suppose $A \subseteq \F_q^n$ has density at least $2^{-d}$.
Then the number of triples $(x,y,z) \in A^3$ with $x + y = 2z$ is at least
$$ q^{-O(d^{9})}  |\F_q^n|^2 .$$
\end{thm}

As a standalone result, we point out that this theorem is strictly worse than the state-of-the-art bound, which can be obtained from algebraic techniques pioneered by \cite{clp17} and \cite{eg17}, which gave a strong resolution to the cap-set problem.
Indeed, Fox and Lov\'asz obtain strong bounds for Green's ``arithmetic removal lemma'' by combining the algebraic results of \cite{kss18} with some additional combinatorial arguments. In particular, their results imply the following.

\begin{thm}[Special case of Theorem 3 in \cite{fl17}]
Suppose $A \subseteq \F_q^n$ has density at least $2^{-d}$. 
Then the number of triples $(x,y,z) \in A^3$ with $x + y = 2z$ is at least
$ q^{-O(d)}  |\F_q^n|^2.$
\end{thm}

Instead, our interest in \cref{3-progs-in-finite-fields} is that we can obtain it using purely analytic techniques and that these same techniques can then be slightly modified and extended to prove \cref{many-3-progs}. 
Partially for expository reasons, we devote a substantial portion of this paper to first addressing the finite field case. This will allow us to present most of the key ideas needed to prove our main theorem while ignoring additional complications arising in the setting $A \subseteq \Z$. We also remark that techniques based on additive combinatorics have found several applications within computer science (going all the way back to the paper of \cite{CFL} who used Behrend's construction for a communication protocol to more recent applications such as \cite{GHZadd}). We believe the new analytic techniques will have other applications. 

\subsection{Structural Results in the Finite Field Setting}

This section highlights a structural result in the finite field setting that follows from our analytic techniques but is not known to follow from algebraic methods. To describe it, we recall the notation for sumsets:
$$A + B := \set{a + b}{a \in A,  b \in B}.$$
For a set $A \subseteq \F_q^n$, we also use the notation $\textnormal{Span}^*(A)$ to denote the minimal affine subspace containing $A$ -- that is, $\textnormal{Span}^*(A)$ refers to the common intersection of all affine subspaces containing $A$, which is itself an affine subspace.

The following structural result says that for any dense set $A \subseteq \F_q^n$, there is a reasonably large subset $A' \subseteq A$, which is tightly contained in its own span (in a specific technical sense).

\begin{lem} [A large subset tightly contained in its own span] \label{sunflower}
Fix a parameter $\tau \in \left[0,\tfrac{1}{2}\right]$.
Suppose that $A \subseteq \F_q^n$ has size $|A| \geq 2^{-d} |\F_q^n|$. 
Then, there is a subset $A' \subseteq A$ with
\begin{enumerate}
\item[i.] $\frac{|A'| }{ |\textnormal{Span}^*(A')|} \geq \frac{ |A| }{ |\F_q^n|}$ and
\item[ii.] $\textnormal{Codim}( \textnormal{Span}^*(A')) \leq O(d^5 \log(1/\tau)^{4})$
\end{enumerate}
such that $ |A' + A'| \geq (1 -\tau) |\textnormal{Span}^*(A')|$.
\end{lem}

Taken together, conditions (i) and (ii) roughly say simply that $A'$ is large: for example, they imply that $|A'| \geq q^{-O(d^5 \log(1/\tau)^{4})} |\F_q^n|$. However, we are additionally guaranteed (by (i)) that the density of $A'$ inside its span is no worse than the original density of $A$ inside $\F_q^{n}$. To unpack the conclusion, let us write $\textnormal{Span}^*(A') =: V + \theta$ for some linear subspace $V$ and some shift $\theta \in \F_q^n$. Also let $B := A' - \theta$. Clearly, we have $\textnormal{Span}(B) = V$. That is, $B$ is contained in $V$, and it also ``eventually" generates $V$: for instance we can surely say that $ \set{c_1 b_1 + c_2 b_2 \cdots + c_n b_n}{c_i \in \F_q, b_i \in B} = V$.
On the other hand, we can compare this to our conclusion, which says that merely taking
$$ B + B = \set{b_1 + b_2}{b_1, b_2 \in B} $$
is already enough to generate all but a tiny fraction of $V$: we have $|B + B| \geq (1- \tau) |V|$ and so $|V \setminus (B + B)| \leq \tau |V|$. 

We proceed to compare this result with some related results in the literature: 
Specifically, Sanders' quasipolynomial Bogolyubov–Ruzsa lemma (specialized to the finite field setting) and the critical technical lemma underlying its proof.

Let us first consider what we'll call the \textit{Bogolyubov–Ruzsa Problem}. We recall the notation for sumsets: $2A$ denotes $A + A$ and $tA$ denotes the set of all sums $a_1 + a_2 + \ldots + a_t$ with $a_i \in A$. 

\begin{problem}[Bogolyubov-Ruzsa Problem over $\F_p^n$]
Let $A \subseteq \F_p^n$ be a set of points in a vector space over a prime field. 
Fix $t \in \N$ and assume that $|A + A| \leq K |A|$.
Find (or rather, prove the existence of) an affine subspace $V$ which is as large as possible and contained in $t A$. 
\end{problem}

Note that any sufficiently dense set (i.e., one of size $|A| \geq K^{-1} |\F_p^n|$) trivially satisfies $|A + A| \leq |\F_q^n| \leq K |A|$, so to solve the Bogolyubov-Ruzsa problem one must at least handle the special case of finding a large subspace in $tA$ for all dense sets $A$. Furthermore, solving the problem in this special case is sufficient to obtain very similar parameters for the general case. The reduction can be summarized as follows: for any $t \leq O(1)$, a set $A \subseteq \F_p^n$ with $|A + A| \leq K |A|$ can be embedded by some linear map $\phi$ from $\F_p^n$ into $\F_p^m$ in such way that
\begin{itemize}
\item $\phi$ is injective on $A$ (and indeed, even on $tA$), 
\item $|\phi(A)| \geq K^{-O(1)} |\F_p^m|$, and
\item The set  $t \phi(A) = \phi(t A)$ contains an affine subspace $V \subseteq \F_p^m$ only if the preimage $\phi^{-1}(V)$ is an affine subspace in $\F_p^n$.
\end{itemize}
For a more detailed explanation, see, e.g., \ the nice description given in \cite{lovett15}.
In light of this reduction to the dense case, we restrict our attention to the following.\footnote{
We note that the reduction from the more general Bogolyubov-Ruzsa problem for vector spaces over prime fields $\F_p$ to the more specific Bogolyubov problem over $\F_p$ uses the fact that subspaces in $\F_p^n$ can be precisely characterized as those subsets which are closed under addition -- i.e., \ additive subgroups. However, it is sensible to consider the Bogolyubov problem itself over general finite fields $\F_q$. 
}

\begin{problem}[Bogolyubov Problem over $\F_q^n$]
Let $A \subseteq \F_q^n$.
Fix $t \in \N$ and assume that $|A| \geq 2^{-d} |\F_q^n|$.
Find an affine subspace $V$ which is as large as possible and contained in $t A$. 
\end{problem}

Regarding this problem, Sanders proves the following.

\begin{thm}[\cite{sanders12}]
Suppose $A \subseteq \F_q^n$ has size $|A| \geq 2^{-d} |\F_q^n|$. 
Then there is an affine subspace $V \subseteq 4A$ of codimension at most $O(d^4)$ in $\F_q^n$. 
\end{thm}

Sanders' approach to the Bogolyubov Problem for $4A$ is first to address what we'll call the \textit{Approximate Bogolyubov Problem} for $2A$:

\begin{problem}[Approximate Bogolyubov Problem over $\F_q^n$]
Let $A \subseteq \F_q^n$.
Fix $t \in \N$ and assume that $|A| \geq 2^{-d} |\F_q^n|$.
Find an affine subspace $V$ which is as large as possible and is ``mostly" contained in $t A$, in the sense that
$$ |V \cap (t A)| \geq (1-\tau)|V|$$
for some $\tau$ as small as possible. 
\end{problem}

The connection between the Approximate Bogolyubov Problem for $tA$ and the Exact Bogolyubov Problem for $2tA$ is due to the following simple claim.

\begin{claim}
Suppose $B \subseteq \F_q^n$ is a set with a large intersection with some linear subspace $V$: $ |B \cap V| > \frac{1}{2} |V|.$
Then $B + B$ contains $V$.
\end{claim}
\begin{proof}
Let $x \in V$. We may trivially write $x = (x + y) - y$ for any element $y \in V$. Consider doing so for a uniformly random choice of $y \in V.$ By a union bound, the chance that the events $x + y \in B$ and $-y \in B$ occur simultaneously is nonzero. 
\end{proof}

Sanders provides the following solution to the Approximate Bogolyubov Problem for sumsets $2A$ over $\F_q^n$.

\begin{thm}[\cite{sanders12}] \label{sanders-approx-bogolyubov}
Suppose $A \subseteq \F_q^n$ has size $|A| \geq 2^{-d} |\F_q^n|$. 
Then for any $\tau \geq 2^{-d}$, there is an affine subspace $V$ of codimension at most $O(d^4 / \tau^2)$ with
$$ |V \cap (A + A)| \geq (1 -\tau) |V|.$$
\end{thm}

It is useful to compare this with our structural result above, \cref{sunflower}, which accomplishes something similar in spirit. Indeed, given $A$ we may find $A' \subseteq A$ and consider its container $V := \textnormal{Span}^*(A')$, which in our case has codimension $O(d^5 \log(1/\tau)^4)$. We note that $2V$ is again an affine subspace of the same size as $V$ ($2V$ is just a translate of $V$) which contains $A' + A'$, and that
$$ |2V \cap (A + A)| \geq |2V \cap (A' + A')| = |A' + A'| \geq (1 -\tau) |2V|.$$
Thus, our result can also be understood as a solution to the Approximate Bogolyubov Problem for $2A$; compared with Sanders' solution, we obtain codimension which is moderately worse in terms of its dependence of the density parameter $d$ in exchange for substantially improved dependence on the error parameter $\tau$. This answers a question posed by Hatami, Hosseini, and Lovett in \cite{hhl18} -- they ask whether it is possible to obtain codimension $d^{O(1)} \log(1/\tau)^{O(1)}$ for the Approximate Bogolyubov Problem for $2A$, and we confirm that indeed it is. 

We can also answer a related question posed by Schoen and Sisask in Section 9 of \cite{schoen-sisask}: they ask whether it is possible to obtain codimension $d^{O(1)}$ for the Exact Bogolyubov Problem for $3A$, as Sanders does for $4A$. The following result confirms this.

\begin{cor}[Direct corollary of \cref{sunflower}]
Suppose $A \subseteq \F_q^n$ has size $|A| \geq 2^{-d} |\F_q^n|$. 
Then, there is a subset $A' \subseteq A$ with size
$$ |A'| \geq q^{-O(d^9)} |\F_q^n| $$
such that
$ A' + A' + A' $
is an affine subspace. 
\end{cor}
\begin{proof}
We apply \cref{sunflower} with parameter $\tau$ set slightly smaller than $2^{-d}$. 
We consider the resulting subset $A'$ and affine subspace $\textnormal{Span}^*(A')$, which we write as $V + \theta$ for some linear subspace $V$ and some shift $\theta$. Consider the set $B := A' - \theta \subseteq V$. We argue that $3B = V,$ which proves that $3A' = V + 3\theta$. 
Indeed, if $3B \neq V$, then there is some $x \in V$ such that the set $x - B$ does not intersect with $B + B$. However, $x - B$ is a set in $V$ of size $|x - B| = |B| \geq 2^{-d} |V|$, and we also have $|B + B| \geq (1 -\tau) |V|$, so this is not possible. 
\end{proof}

\subsection{A robust structural lemma}

The structural result given in the previous section is in fact a special case of a more analytically robust variant which we proceed to describe. We recall the notation
$$ R_A(x) := \left|\set{(a_1, a_2) \in A^2}{a_1 + a_2 = x} \right|,$$
which counts the number of ``representations" of $x$ as a sum of two elements of $A$.
We also consider the alternative normalization
$$ r_A(x) := \frac{R_A(x)}{|A|^2}$$
which is normalized as a distribution function on $\F_q^n$ -- that is, $\sum_{x \in \F_q^n} r_A(x) = 1$.
Recall that \cref{sunflower} says that for any dense set $A \subseteq \F_q^n,$ there is a reasonably large subset $A ' \subseteq A$ which is tightly contained in its own span, in the sense that 
$$ |A' + A'| \geq (1 - \tau) |\textnormal{Span}^*(A')|. $$
For the sake of clarity let us write $\textnormal{Span}^*(A') = V + \theta$ where $V$ is a linear subspace. We have of course that $|V| = |\textnormal{Span}^*(A')|$.
Note that the distribution $r_{A'}(x)$ is supported on the affine subspace $2\textnormal{Span}^*(A') = V + 2 \theta.$
Our robust variant of \cref{sunflower} states that in fact the distribution $r_{A'}(x)$ is very close to the uniform distribution on $V + 2\theta$. The fact that the support of $r_{A'}$, i.e.\ $A' + A'$, is large follows as a corollary. 

To state this result we need the following notion for measuring  closeness of distributions. In what follows, a ``distribution" $\pi : \Omega \rightarrow \R$ is a nonnegative function on $\Omega$ with $\sum_{x \in \Omega} \pi(x) = 1$.

\begin{definition}

Given $k \geq 1$ and two distributions $\pi, \pi' : \Omega \rightarrow \R$, define the $k$-norm divergence of $\pi$ from $\pi'$ as the quantity
$$  \left(  \frac{ \sum_{x \in \Omega} |\pi(x) - \pi'(x)|^k  }{  \sum_{x \in \Omega} \pi'(x)^k } \right)^{1/k}.
$$  
In the case that $\pi'$ corresponds to some flat distribution $ \pi'(x) = \frac{\1_S(x)}{|S|}$,
uniform over some subset $S \subseteq \Omega$, and $\pi$ is supported on $S$, this can be equivalently expressed as
$$ \left( \E_{x \in S} \left| \frac{\pi(x)}{\pi'(x)} - 1 \right|^k \right)^{1/k}.$$
\end{definition}

\begin{lem}[A robust variant of \cref{sunflower}] \label{robust-sunflower}
Fix a constant $\eps \in (0,1)$ and a parameter $k \geq 1$.
Suppose that $A \subseteq \F_q^n$ has size $|A| \geq 2^{-d} |\F_q^n|$. 
Then, there is a subset $A' \subseteq A$ with
\begin{enumerate}
\item[i.] $\frac{|A'| }{ |\textnormal{Span}^*(A')|} \geq \frac{ |A| }{ |\F_q^n|}$ and
\item[ii.] $\textnormal{Codim}( \textnormal{Span}^*(A')) \leq O_{\eps}(d^5 k^4)$
\end{enumerate}
such that the distribution $r_{A'}(x)$ has $k$-norm divergence at most $\eps$ from the uniform distribution on $2 \textnormal{Span}^*(A')$. 
\end{lem}

We note that \cref{sunflower} does indeed follow as a special case by invoking \cref{robust-sunflower} e.g.\ with $\eps = 1/2$. This is because, if $\pi$ is a distribution supported on some set $B \subseteq B'$, and $|B' \setminus B| \geq \tau |B'|$, then the $k$-norm divergence from $\pi$ to the uniform distribution on $B'$ is at least $\tau^{1/k}$.

\section{Technical introduction} \label{technical-introduction}

The techniques and applications considered in this work are centered around the following question. Let $A, B$, and $C$ be subsets of a finite abelian group $G$ of size $N$, each with size at least $2^{-d} N$. We ask: what sort of generic ``pseudorandom" conditions on the sets $A, B, C$ are sufficient to ensure that the number of solutions to the equation
$ a + b = c, $
with $(a,b,c)\in A \times B \times C$,
is off by only a small multiplicative factor $(1 \pm \eps)$ from the the ``expected" number, $ |A||B||C| / N $?

We are particularly concerned with the following three conditions on a set $A \subseteq G$, which each attempts to quantify the ``additive pseudorandomness" of $A$ in some capacity. For concreteness, we will first state the three conditions -- spreadness, regularity, and self-regularity -- in the counting measure. However, we will shortly return to restate each of them using alternative normalization conventions which will greatly increase their interpretability. We recall some standard notations for the number of ``representations" of $x$ as a sum or difference of elements of some sets $A$ and $B$: $R_{A,B}(x) := \left | \set{(a,b) \in (A \times B)}{a + b = x}\right|$,
$R_{A,B}^{-}(x) := \left | \set{(a,b) \in (A \times B)}{a - b = x}\right|$,
$R_{A}(x) := R_{A,A}(x)$, and 
$R_{A}^{-}(x) := R_{A,A}^{-}(x)$.

\begin{definition}[$(\gamma, r)$-spread] \label{spread}
Let $A$ be a subset of $\F_q^n$. For $\gamma, r \geq 1$, we say $A$ is $(\gamma,r)$-spread if, 
$$  \frac{|A \cap V|}{|V|} \leq  \frac{ \gamma |A|}{|\F_q^n|}  .
$$
for all affine subspaces $V \subseteq \F_q^n$ of codimension at most $r$.
\end{definition}

Our first condition -- spreadness -- applies only to the case $G = \F_q^n$; however, we will eventually return to consider possible substitutes for when we are working in a different group.\footnote{Specifically, we will consider some alternatives where a Bohr set, or a generalized arithmetic progression, instead plays the role of the subspace.} Spreadness is important to us because it is \textit{easy to obtain}, for instance, by a density-increment argument. 

\begin{definition}[$(\gamma,k)$-regular] \label{regular}
 Let $A$ be a subset of a finite abelian group $G$. For $\gamma, k \geq 1$, 
 We say that $A$ is $(\gamma,k)$-regular if
 $$  \sum_{a \in A} R_{B,C}(a) \leq \frac{\gamma |A| |B| |C|}{|G|}
 $$
 for all sets $B,C \subseteq G$ with size at least $|B|,|C| \geq 2^{-k} |G|$.
\end{definition}

\begin{definition}[$(\gamma,k)$-self-regular] \label{self-regular}
 We say that $A \subseteq G$ is $(\gamma,k)$-self-regular if
 $$ \sum_{x \in G} R_{A}^{-}(x)^k \leq \frac{\gamma |A|^{2k} }{|G|^{k-1}} .
 $$
\end{definition}

Our latter two conditions -- regularity and self-regularity -- apply to general finite abelian groups. In contrast to spreadness, regularity is important to us because it is \textit{useful to have}, although it is not clear that it is easy to obtain.

We can summarize the overall argument for our structural result, \cref{robust-sunflower}, into three steps:
\begin{enumerate}
    \item One can use density increments to obtain spreadness. This follows from a   simple greedy argument and is formalized in \cref{greedy}.
    \item One can (informally speaking) obtain regularity from spreadness. This is a core ingredient of our approach and is formalized in \cref{II}.
    \item One can obtain strong \emph{two-sided} bounds on the number of solutions to $a + b = c$ from regularity. This is formalized in \cref{I}. 
\end{enumerate}

Points 2 and 3 above can be considered the two main technical contributions of our work. Taken together, they give the following answer to 
the central question described above (for the setting $G = \F_q^n$). To control the number of solutions to $a + b + c = 0$ with $a,b,c \in A,B,C$, to within a factor $(1 \pm O(\eps))$ of the expected number $|A||B||C|/|G|$, it suffices that any two of the three sets are $(1+\eps,\textnormal{poly}(d,1/\eps))$-spread.

\subsection{Density formulation} \label{density-formulation}

We take a moment to switch to a more ``analytic" language for expressing the three conditions introduced above which both eases their interpretability and clarifies their importance for controlling the number of solutions to $a + b = c$. Here in the introduction, we avoid an extended discussion of details regarding, e.g., \ normalization conventions  (a more detailed description can be found in \cref{prelims}). However, we highlight the following key points.

\begin{itemize}
\item We use the notation $\|f\|_{k} :=  \left( \E_{x \in G} |f(x)|^k \right)^{1/k}$ for $k \geq 1$ and functions $f : G \rightarrow \R$, as well as $\|f\|_{\infty} := \max_{x \in G} |f(x)| $.
\item We consider \textit{density functions} on $G$, which are simply nonnegative functions $F : G \rightarrow \R$ with $\| F \|_1 = 1$. 
\item Given a set $A \subseteq G$, we abuse notation and let $A(x)$ denote the \textit{density function} which corresponds to the uniform density on $A$. Under this normalization, a set of size $|A| = 2^{-d} |G|$ has corresponding density function $A(x)$ with $\|A\|_{\infty} =  2^{d}.$
\item Given density functions $A(x)$ and $B(x)$, we consider both the \textit{convolution} $(A * B)(x)$ and the \textit{cross-correlation} $(A \star B)(x)$, both of which are again density functions. Indeed, they are the density functions that are proportional to the representation-counting functions $R_{A, B}(x)$ and $R_{B, A}^{-}(x)$. 
\end{itemize}

Finally, let us introduce the following notations relevant to capturing the spreadness and regularity conditions. For nonnegative $f : \F_q^n \rightarrow \R$, we define
\begin{equation} \label{perp-norm}
\|f\|_{\perp, r} := \max_{\substack{V \subseteq \F_q^n \\ \textnormal{Codim}(V) \leq r}} \ip{V}{f},
\end{equation}
where the maximization is quantified over affine subspaces $V \subseteq \F_q^n$ of codimension at most  $r$. For nonnegative $f : G \rightarrow \R$, we define
\begin{equation} \label{star-norm}
\|f\|_{*,k} := \max_{\substack{B,C \subseteq G \\ \|B\|_{\infty}, \|C\|_{\infty} \leq 2^{k}}} \ip{B * C}{f},
\end{equation}

where the maximization is quantified over subsets $B,C \subseteq G$ of sizes $|B|,|C| \geq 2^{-k} |G|$. 
We can restate each of our three conditions above in the density formulation.

\begin{definition}[Spreadness -- density formulation]
We say that $A \subseteq \F_q^n$ is $(\gamma,r)$-spread if
$$ \|A\|_{\perp, r} \leq \gamma.$$
\end{definition}

\begin{definition}[Regularity -- density formulation]  
 We say that $A$ is $(\gamma,k)$-regular if
 $$ \|A\|_{*, k} \leq \gamma.
 $$
\end{definition}

\begin{definition}[Self-regularity -- density formulation] 
 We say that $A$ is $(\gamma,k)$-self-regular if
 $$ \|A \star A \|_{k} \leq \gamma. $$
\end{definition}

Generally, for all three conditions, we are most interested in the regime where $\gamma = 1 + \eps$ and $\eps$ is a small constant.

\subsection{Regularity from spreadness; Upper-bounds from regularity} \label{regularity-from-spreadness}

We can also restate our central question in the density formulation as follows.

\begin{question} \label{central-question}
    Let $A, B, C \subseteq G$. Suppose that the corresponding density functions satisfy the (min-)entropy deficit condition $\|A\|_{\infty}, \|B\|_{\infty}, \|C\|_{\infty} \leq 2^{d}$. 
    What sort of generic pseudorandom conditions on $A(x),B(x),C(x)$ are sufficient to ensure that
    $$ \ip{A * B}{C} = 1 \pm O(\eps),$$
    where $\eps$ is some small constant?
\end{question}

We note that (basically by definition) imposing the condition of $(1+\eps,d)$-regularity on the set $A$ is sufficient to obtain good \textit{upper bounds} for this question (and also necessary if we wish to avoid imposing any further conditions on the sets $B$ and $C$ beyond the given constraints on their sizes) as we have $\ip{A * B}{C} = \ip{A}{B \star C}$. 
Thus, if this quantity is at least $1 + \eps$ then the fact that $\|A\|_{*,d} \geq 1 + \eps$ is witnessed by the convolution of $C$ with $-B$.  

Our first main technical result says that spread sets are regular.
\begin{thm}[Structure vs.\ Pseudorandomness in $\F_q^n$ -- simplified combination of \cref{II-part-2} and \cref{II-part-1}] \label{II}
Suppose $A \subseteq \F_q^n$ has size $|A| \geq 2^{-d} |\F_q^n|.$
If $ \|A\|_{*,k} \geq 1 + \eps$, 
then
$ \|A\|_{\perp,r} \geq 1 + \frac{\eps}{8}$ for some $r \leq O_{\eps}(k^7 d) .$
Similarly, if
$ \|A \star A\|_{k} \geq 1 + \eps,$
then
$ \|A\|_{\perp,r} \geq 1 + \frac{\eps}{4}.$
for some $r \leq O_{\eps}(k^4 d^4)$. 
\end{thm}
We can offer a couple of different interpretations of this result. 
Firstly we can interpret it analytically as follows. Suppose the density function $A(x)$ has a deviation upwards from the uniform density function, which can be detected, or ``witnessed", by some convolution of two large sets (i.e., \ by a density corresponding to a rather weak additive structure). Then, the above theorem says that this deviation can be witnessed almost as well by the uniform density over a reasonably large affine subspace (i.e., \ a set with a very strict additive structure). Alternatively, we can interpret the contrapositive in the well-known structure vs.\ pseudorandomness paradigm. Specifically, we have the following dichotomy: a given set $A$ is either \textit{pseudorandom}, in the sense that it is $(1+\eps,k)$-regular, or it has some non-negligible amount of additive structure which can be detected by some large affine subspace $V$,
$$ \ip{V}{A} \geq 1 + \frac{\eps}{8} ,$$
and we obtain the resulting  \textit{density increment} of $A$ onto $V$:
$$ \frac{|A \cap V|}{|V|} \geq \left(1 + \frac{\eps}{8}\right) \frac{|A|}{|\F_q^n|}. $$

We give the following quick summary of the proof of \cref{II} which is intended for experts.
Alternatively, a more detailed overview can be found in \cref{extended-proof-overview}, which in particular gives some additional discussion of the tools involved, as well as an account of how the proof is arrived at ``naturally" as the result of a sequence of observations.

For brevity, we discuss only the most important claim made in \cref{II}: that $\|A \star A\|_{k} \geq 1 + \Omega(1)$ implies $\|A\|_{\perp,r} \geq 1 + \Omega(1)$ for some $r \leq O(k^4 d^4)$.
We will establish the stronger claim
$ \|A \star A\|_{\perp, r} \geq 1 + \Omega(1), $
which is sufficient.
Our starting point is the following consequence of Sanders' invariance lemma.\footnote{
By this we mean the result one obtains by combining the Croot-Sisask lemma with Chang's inequality -- an idea which first appeared in the work of Sanders \cite{sanders12}. See \cref{sanders-invar-appendix} for the specific form of the result used here.}

\begin{prop}
Let $k \geq 1$ and $\eps \in [2^{-k},1]$. 
For any bounded function $f : \F_q^n \rightarrow [0,1]$, we have
$$ \|f\|_{\perp,k^4/\eps^2} \geq \|f\|_{*,k} - O(\eps). $$
\end{prop}

Suppose $\|A \star A\|_{k} \geq 1 + 2\eps$ for some small constant $\eps$, and consider the super-level set indicator
$$ f(x) := \ind{(A \star A)(x) \geq 1 + \eps}. $$
Our plan is to argue that $\|f\|_{*,k'} \geq 1 - \frac{\eps}{8}$ for some $k' \leq O(kd)$. 
From this it follows that
$$ \|A \star A\|_{\perp, O(k^4 d^4)} \geq (1 + \eps) \|f\|_{\perp, O(k^4 d^4)} \geq (1 + \eps)(1 - \tfrac{\eps}{4}) \geq 1 + \Omega(1),$$
as desired. To argue that we indeed have $\|f\|_{*,k'} \approx 1$ for some value $k'$ which is not too large, we develop the following tool (which is applicable more generally to any finite abelian group $G$).

\begin{lem}[Sifting lemma]
Suppose $A \subseteq G$ has size $|A| = \delta |G|$.
Fix a nonnegative function $h : G \rightarrow \R_{\geq 0}$ and an integer $k \geq 2$.
Let $D(x)$ 
denote the unique density function which is proportional to $(A \star A)(x)^k$, and suppose that
$ \ip{D}{h} = \eta.$

Then there is a subset $A' \subseteq A$ with
$ \ip{A' \star A'}{h} \leq 2 \cdot \eta $
and
$ \frac{|A'|}{|G|} \geq \tfrac{1}{2} \cdot \delta^{k}. $
\end{lem}

Applying this to $h := 1-f$ shows that $\|f\|_{*,O(kd)} \approx 1$. This completes the argument; it remains only to prove the sifting lemma and work out the quantitative details -- this is done in \cref{sec:sifting-lemma}.
We remark briefly that the sifting lemma can be seen as an extension of Schoen's ``Pre-BSG" lemma -- this is discussed in more detail in \cref{the-pre-BSG-lemma-and-sifting}.

\subsection{Strong two-sided bounds from self-regularity} \label{2-sided}

We have seen that spreadness is sufficient to give good \textit{upper bounds} for our central question,  \cref{central-question}. However, for applications, typically, one is more interested in \textit{lower bounds} -- in particular, in settings where we would like to establish the existence of solutions.
For this we turn to the following: our second main technical contribution, which says that if $A$ and $B$ are both self-regular, then $A * B$ is near-uniform. 

\begin{thm}[Strong two-sided bounds from self-regularity] \label{I}
Let $A, B$ be subsets of a finite abelian group $G$.
Suppose that $A$ and $B$ are both $(1 + \eps, \lceil k/\eps \rceil )$-self-regular for some $\eps \in [0,\frac{1}{4}]$ and some even integer $k \geq 2$. Then
$$ \|A * B - 1\|_{k} \leq 2 \eps. $$
In particular,
$ \ip{A * B}{C} = 1 \pm O(\eps) $
for any set $C$ of size $|C| \geq 2^{-k}|G|$.
\end{thm}

\begin{proof}

Our first step is to apply the following claim, which can be proved without much trouble by a slightly nontrivial application of Cauchy–Schwarz.
\begin{prop}[Decoupling inequality] \label{cauchy-schwarz}
Let $A, B$ be density functions on a finite abelian group $G$. 
For even integers $k \in \N$ we have
$$ \|A * B - 1\|_k \leq \|A \star A - 1\|_k^{1/2} \|B \star B- 1\|_k^{1/2} .$$
\end{prop}
So, it suffices to show that both $\|A \star A - 1\|_k \leq 2\eps$ and $\|B \star B - 1\|_k \leq 2 \eps$; without loss of generality let us consider $A$.
We argue that the near-uniformity of $A \star A$ follows from an upper bound on $\|A \star A\|_{k'}$ (with $k' := \lceil k/\eps \rceil $).
This follows by combining two simple claims. 

\begin{prop} [Positive correlation for spectrally positive functions -- restatement of \cref{positive-correlation}] \label{positive-correlation-section-2} Let $G$ be a finite abelian group, and let $f_1, f_2, \ldots, f_t$ be some real-valued functions on $G$ which are ``spectrally positive." That is, each of the Fourier coefficients $\hat{f_i}(\alpha)$ is real and non-negative. 
Then, for uniformly random $x \in G$,
$$ \E [ f_1(x) f_2(x) \cdots f_t(x) ] \geq \E [ f_1(x) ]  \E [ f_2(x) ] \cdots \E [f_t(x)] \geq 0. $$
\end{prop}

\begin{prop}[Odd moments -- reformulation of \cref{odd-moments}] 
Let $Z$ be a real-valued random variable. We use the notation $\|Z\|_{k} := \E\left[|Z|^{k}\right]^{1/k}$. 
Suppose $Z$ has non-negative odd moments: that is, $\E\left[Z^t\right] \geq 0$ for all odd integers $t \in \N$. If 
$ \|1+Z\|_{k'} \leq 1 + \eps$
for some integer $k'$ and some $\eps \in [0,\frac{1}{4}]$, then
$ \|Z\|_{k} \leq 2 \eps$
for any even integer $k \leq \eps k'$.
\end{prop}

Our density function $(A \star A)(x)$ is indeed spectrally positive, and so is its centering, $F := A \star A - 1$. From this we see that $A \star A$ has non-negative odd central moments:
$$ \E \big((A \star A)(x) - 1\big)^{t} = F(x)^t \geq 0$$
for all odd $t \in \N$. Since we have assumed the upper bound
$ \|A \star A\|_{k'} \leq 1 + \eps,$
we obtain the desired two-sided bound 
$ \|A \star A - 1\|_k \leq 2 \eps. $
 This proves the first claim.
For the second claim, we can write
\begin{align*}
\ip{A * B}{C} &= \ip{1}{C} + \ip{A * B - 1}{C} = 1 + \ip{A * B - 1}{C}  ,
\end{align*}
and use a Hölder inequality to estimate
\begin{multline*}
\left| \ip{A * B - 1}{C} \right| \leq \|A * B - 1\|_{k} \|C\|_{1 + \frac{1}{k-1}} 
\leq \|A * B - 1\|_{k} \|C\|_{\infty}^{1/k} \|C\|_{1}^{1-1/k} \\
=   \|A * B - 1\|_{k}  \|C\|_{\infty}^{1/k} 
\leq 2 \|A * B - 1\|_{k}.  \qedhere
\end{multline*}

\end{proof}

We remark that it is always possible, for generic density functions $D$, to infer some kind of bound on $\|D-1\|_{k}$ from an upper bound $\|D\|_{k} \leq 1 + \eps$. However, this generic bound degrades rapidly as $k$ increases: consider, for example, a density function which is uniform over some subset of $G$ of size $(1-\eps) |G|$. In this case we have
$$ \|D\|_{k'} \leq \|D\|_{\infty} = \frac{1}{1-\eps} \approx 1 + \eps$$
for any $k'$, and yet
$$ \|D-1\|_{k} \geq \eps^{1/k}. $$
So, for our arguments above, it was quite important that there are no self-convolutions which are shaped in such a way. More specifically: self-convolutions $A \star A$ which have deviations downwards from $1$ must also have deviations upwards from $1$ of comparable strength, as measured by the $k$-norm, for (roughly speaking) any choice of $k$. This is what allows us to efficiently convert our upper bound from \cref{II} into a lower bound. 
It is interesting to contrast this situation with the discussion surrounding \cref{central-question-spread} in the appendix. Specifically, we discuss a formal setting where one can obtain strong upper-control on the quantity $\ip{A * B}{C}$ from the density-increment method, but it is impossible to obtain any nontrivial lower bound $\ip{A * B}{C} > 0$.

\textbf{A review of all the steps.} \label{review}
We recall that, as a result of the combination of all of our observations, one can infer from the starting assumption $\ip{A * B}{C} \leq 1 -\Omega(1)$ that either $A$ or $B$ must have a density increment onto some large affine subspace $V$. Now that we have seen everything needed for this, we offer the following summary listing all the steps in one place.
\begin{itemize}
\item Infer by Hölder that $\|A * B - 1\|_{k} \geq \Omega(1)$, for $k \geq d$, as witnessed by $C$.
\item Argue that $k$-norm distance to uniform is maximized by considering self-convolutions:\\ $\|A * B - 1\|_{k} \leq \|A \star A - 1\|_{k}^{1/2} \|B \star B - 1\|_{k}^{1/2}$. Conclude that either $\|A \star A - 1\|_{k} \geq \Omega(1)$ or  $\|B \star B - 1\|_{k} \geq \Omega(1)$; Suppose $\|A \star A - 1\|_{k} \geq \Omega(1)$.
\item Argue by spectral positivity that the presence of deviations of $A \star A$ downwards from $1$ entails the presence of deviations upwards from $1$ of comparable strength. More concretely: argue that (upon increasing $k$ slightly) we have $\|A \star A\|_{k} \geq 1 + \Omega(1)$.
\item Use sifting to find a convolution $A' \star A'$ witnessing $\|f\|_{*,O(dk)} \approx 1$, where 
$$f := \ind{A \star A \geq 1 + \Omega(1)}.$$
\item Use Sanders' invariance lemma (which is itself powered by the Croot-Sisask lemma and Chang's inequality) to deduce that $\|f\|_{\perp, O(d^4 k^4)} \approx 1$.
\item Conclude that $\ip{V}{A \star A} \geq 1 + \Omega(1)$ for some $V$.
\item Conclude that $\ip{V'}{A} \geq 1 + \Omega(1)$ for some $V ' = V + a$.
\end{itemize}

Let us now outline our plan for the remainder of the paper. 

In \cref{finitefield-mainproofs}, 
 we discuss the well-known density-increment framework and describe how it can be used to find a large subset $A' \subseteq A$ which is spread relative to its span.  We use this to complete the proof our main result for finite field setting: the structural lemma, \cref{robust-sunflower}.

In the middle part of the paper, we develop our main tools used to prove  \cref{II} and \cref{I}: ``sifting" and ``spectral positivity," respectively. Both tools apply to general finite abelian groups $G$. We also establish some ``local" variants of these techniques in preparation for later when we will consider subsets $A \subseteq \Z$, as well as subsets of cyclic groups.

\begin{itemize}
\item In \cref{prelims}, we go over some definitions and conventions related to normalization, convolutions, and Fourier analysis. 
\item In \cref{sec:sifting-and-proof-of-II}, we prove the ``sifting lemma" (\cref{sifting-lem-density-formulation}) and use is to finish the proof of \cref{II}.
\item In \cref{spectral positivity}, we develop the simple tools needed to prove \cref{I}. These include various $k$-norm inequalities for convolutions and self-convolutions, the latter of which relies on the notion of spectral positivity. 
\end{itemize}

In the final part of the paper, we consider the setting $A \subseteq [N] \subseteq \Z$. 

\begin{itemize}
\item In \cref{integers-overview}, we give an overview of our plan to establish the existence of many 3-progressions in the integer case, discussing what changes must be made to our approach which handles the finite field case. Most notably, we consider some potential replacements for the notion of ``spreadness", where some kind of approximate subgroup, such as a generalized arithmetic progression of bounded rank or a Bohr set of bounded rank, instead plays the role of the subspace of bounded codimension. 
\item In \cref{prelims-integers}, we go over some definitions and basic properties related to generalized progressions, Bohr sets, and Freiman homomorphisms. We also introduce a potentially new device related to the notion of a Freiman homomorphism which we call a ``safe" set. 
\item In \cref{proof-of-3-progs-in-the-integers}, we complete our proof that dense sets $A \subset [N]$ have many 3-progressions, which relies on tools including sifting, spectral positivity, safe sets, as well as a translation-invariance lemma which is due to Schoen and Sisask.
\end{itemize}
\subsection{The density-increment framework: completing the proof of \cref{robust-sunflower}}\label{finitefield-mainproofs}

In this section we prove our main result for the finite field setting by combining \cref{II} and \cref{I} with the well-known density-increment approach.

Let us argue that spreadness is a property that is easy to obtain.
Before we begin, we discuss the notion of relative spreadness.
For a set $A \subseteq \F_q^n$ contained within some linear subspace $V$ of codimension $s$, let us say that $A$ is spread relative to $V$ if, upon embedding
$$ A \subseteq V \cong \F_q^{n-s},$$
the set is spread in the sense of \cref{spread}. 
Additionally, suppose that $A$ is contained in some affine subspace $V'$, which can be uniquely described as $V' = V + \theta$ for some linear subspace $V$ and some $\theta \in V^{\perp}.$ Let us say that $A$ is spread relative to $V'$ if $A - \theta$ is spread relative to $V$.

Now, given a set $A \subseteq \F_q^n$ of density at least $2^{-d}$, we can find a large set $A'$ which is spread (relative to the ambient space $\textnormal{Span}^*(A')$) by the following simple greedy algorithm. Let $A_0 := A$. We proceed to describe a nested sequence of subsets $A_{i} \subseteq A_{i-1}$; let $V_i = \textnormal{Span}^*(A_i)$. If the current set $A_i$ is not $(1+\eps,r)$ spread (relative to its container $V_i$), then there must be an affine subspace $V' \subseteq V_i$, which has codimension at most $r$ inside $V_i$, such that
$$ \frac{|A_i \cap V'|}{|V'|} > \left(1 + \eps\right) \frac{|A_i|}{|V_i|}. $$
We then pass to the subset $A_{i+1} := A_i \cap V'$. Since we begin with a set with density $\delta_0 \geq 2^{-d}$, and the density of $A_i$ in $V_i$ increases by at least a factor $(1+\eps)$ on every iteration, we must have density
$$ \delta_{i} \geq (1+\eps)^{i} \cdot 2^{-d}.$$
after $i$ iterations. However, density cannot exceed $1$ at any point. Thus, in the case of $\eps \in [0,1]$, this process must terminate within some number $t \leq d/\eps$ iterations; by definition, this means that the final set $A_t$ is $(1+\eps,r)$-spread inside its own span. To summarize, by this simple argument, we have proved the following.

\begin{prop} \label{greedy}
Let $\eps \in [0,1]$.
Suppose $A \subseteq \F_q^n$ has size $|A| \geq 2^{-d} |\F_q^n|$.
Then, for some linear subspace $V \subseteq \F_q^n$ of codimension $s \leq rd/\eps$, 
and some shift $\theta \in V^{\perp}$,
the subset $A' := A \cap (V + \theta)$ satisfies the following:
\begin{enumerate}
    \item $\frac{|A'|}{|V|} \geq \frac{|A|}{|\F_q^n|}$, and
    \item The set $A' - \theta \subseteq V$ is $(1+\eps,r)$-spread inside $V \cong \F_{q}^{n-s}$ in the sense of \cref{spread}.
\end{enumerate}
\end{prop}

We remark that it is somewhat more common to frame the structure vs.\ pseudorandomness approach as a sort of ``branching process", where at each step, our current set is either pseudorandom and therefore (possibly by a lengthy argument) satisfactory, or we find a density increment and then repeat this argument recursively. Inspired by \cite{bk12} and \cite{alwz20}, we find it clarifying to insist that we do all of our density-incrementing up front and then argue that the resulting (spread) set must be satisfactory. Ultimately, though, there does not seem to be a tangible advantage arising from either viewpoint.

Next, we argue that spreadness of $A$ and $B$ implies near-uniformity of $A * B$ (simply by combining \cref{II} with \cref{I}).

\begin{prop} [Near-uniformity from spreadness] \label{near-uniformaity-from-spreadness}
We have the following for some absolute constant $c \geq 1$.
Suppose $A,B \subseteq \F_q^n$ are two sets each of size at least $2^{-d}|\F_q^n |$, where $d \geq 1$.
Let $r \in \N$, $\eps \in [0,\frac{1}{4}]$, and $k \geq 1$ be such that
$$ r \geq d^4 k^4 / \eps^c .$$
If $\|A\|_{\perp, r} \leq 1 + \eps$ and $\|B\|_{\perp,r} \leq 1 + \eps$, then
$$ \|A * B - 1\|_{k} \leq O(\eps).$$
\end{prop}

\begin{proof}
We argue by the contrapositive.
Let $\eta \in [0,1]$, and suppose that
$$ \|A * B - 1\|_{k} \geq  \eta.$$
By \cref{I}, this means that
either
$ \|A \star A\|_{k'} \geq  1 + \Omega(\eta)$
or
$ \|B \star B\|_{k'} \geq  1 + \Omega(\eta)$
for some $k' \leq O(k/\eta)$;
without loss of generality suppose $ \|A \star A\|_{k'} \geq 1 + \Omega(\eta)$.
We use \cref{II} (or rather, the more specific \cref{II-part-2}) to infer that
$$ \|A\|_{\perp, r} \geq 1 + \Omega(\eta) $$
for some $r \leq M(\eta) \cdot d^4 k^4$, where
$M(\eta) \leq \textnormal{poly}(1/\eta)$. 
Since $\eta \in [0,1]$ was arbitrary, the result follows. \qedhere
\end{proof}

By combining the two claims above we easily obtain our structural lemma stated in the introduction, \cref{robust-sunflower}.
\begin{proof}[Proof of \cref{robust-sunflower}]
We are given a constant $\eps$, a parameter $k \geq 1$, and a set $|A| \geq 2^{-d} |\F_q^n|$. Changing $\eps$ and $d$ only slightly, we may assume $d \geq 1$ and $\eps \leq 1/4$. 

We set a parameter $r \in \N$ to be some constant factor larger than $k^4 d^4 / \eps^c$. We invoke \cref{greedy} and consider the subset $A'  = A \cap (V + \theta)$ which is $(1+\eps,r)$-spread relative to its span, $V + \theta$, an affine subspace with codimension at most $rd/\eps = O(d^5 k^4 / \eps^{c+1})$ in $\F_q^n$.
By our choice of $r$, we may apply \cref{near-uniformaity-from-spreadness}.
So, the distribution
$$ r_{A'-\theta}(x) := \frac{R_{A-\theta}(x)}{|A|^2} $$
has $k$-norm divergence from the uniform distribution on $V$ bounded by $O(\eps)$. 
Equivalently, $r_{A'}(x)$ has $O(\eps)$ divergence from the uniform distribution on $V + 2\theta$. 
Since $\eps \in [0,\frac{1}{2}]$ was arbitrary the result follows. \qedhere
\end{proof}

Furthermore, our lower bound for 3-progressions in the finite field setting also follows as a direct consequence of \cref{robust-sunflower}.

\begin{proof}[Proof of \cref{3-progs-in-finite-fields} from \cref{robust-sunflower}]
Suppose $A \subseteq \F_q^n$ has size $|A| \geq 2^{-d} |\F_q^n|$.
If $q$ is even, we clearly have $|A|^2$ solutions to $x + y = 2z$, since $2z = 0$ for all $z$. So, let us assume that $q$ is odd.
We seek to lower bound the number of 3-progressions in $A$, which is to say the quantity
$$ \sum_{z \in A} R_{A}(2z).$$
Our plan is to invoke \cref{robust-sunflower}, with $\eps = 1/4$ and $k = d + 1$, to obtain a nice subset $A'$, and then we will count only the solutions to $x + y = 2z$ with $x,y,z \in A'$, and ignore the rest.
For the sake of clarity let us write $\textnormal{Span}^*(A) = V + \theta$ where $V$ is a linear subspace, and consider instead the translated set $B := A' - \theta \subseteq V$, noting that translating an entire set does not change the number of solutions to $x + y = 2z$ with $x,y,z$ in the set. 

\cref{robust-sunflower} tells us that the distribution $r_B(x)$ has $k$-norm divergence at most $1/4$ from the uniform distribution on $V$. This in particular means that the number of points $x \in V$ where
$$ r_B(x) \leq \frac{1}{2|V|}  $$
is at most $2^{-k} |V|$. Now using the fact that $q$ is odd, we note that the dilation map $z \mapsto 2z$ is a permutation of $V$, and so
$$ \sum_{z \in B} r_B(2z) \geq  \left( |B| - 2^{-k} |V| \right) \frac{1}{2 |V|} \geq \frac{2^{-d}}{4},$$
where we have used that fact that $|B| \geq 2^{-d} |V|$ and that $k = d + 1$. Overall, we see that
\begin{align*}
\sum_{z \in A} R_{A}(2z) &\geq \sum_{z \in B} R_{B}(2z) \\
&\geq 2^{-O(d)} |B|^2 \\
&\geq q^{-O(d^5 k^4)} |\F_q^n|^2 \\
&= q^{-O(d^9)} |\F_q^n|^2 . \qedhere
\end{align*}
\end{proof}

\section{Preliminaries} \label{prelims}

\subsection{Densities and normalization}

\begin{definition}
 For an arbitrary finite set $\Omega$, a \textit{density function} on $\Omega$ is simply a non-negative function $D : \Omega \rightarrow \R_{\geq 0}$ normalized so that
 $$\E_{x \in \Omega} D(x) = 1$$.
\end{definition}
\begin{itemize}
\item Let $A \subseteq \Omega$ be a subset of size $|A| = \delta|\Omega|$. Abusing notation, we also write 
\begin{equation*}
A(x) := \frac{\1_A(x)}{\delta} = 
\begin{cases} 
    1/\delta &\text{ for } x \in A\\
    0 &\text{ otherwise}\\
\end{cases}
\end{equation*}
to denote the (re-normalized) indicator function of $A$.
Under this normalization, $A(x)$ is a density function.

\item For any two functions $f, g : \Omega \rightarrow \R$ we define $\ip{f}{g} := \E_{x \in \Omega} f(x)g(x)$. For any set $A$ we have the probabilistic interpretation
\begin{equation*}
\ip{A}{g} = \E_{x \in \Omega} A(x) g(x) = \E_{a \in A} g(a).
\end{equation*}
We can also apply this more generally to an arbitrary density function $B$. Let $b \sim B$ denote a random variable in $\Omega$ whose probability distribution is proportional to $B$. Then $\langle B, g \rangle$ has the interpretation
\begin{equation*}
     \ip{B}{g} = \E_{x \in \Omega} B(x) g(x) = \E_{b \sim B} g(b).
\end{equation*}
Occasionally it will be useful to allow more generally for complex-valued functions $f,g$ (which can appear e.g.\ during some intermediate calculations involving Fourier expansions of real-valued functions, even if the resulting quantity must be also real). In this case we insist on the convention
$$ \ip{f}{g} = \E_{x \in \Omega} \overline{f(x)} g(x)$$
so that at least when $B(x)$ is real-valued and $g(x)$ is complex-valued we still have
$$ \ip{B}{g} = \E_{x} B(x) g(x) = \E_{b \sim B} g(b).$$

\item For $k\geq 1$, we use the notation
\begin{equation*}
    \norm{f}_k = \left( \E_{x \in \Omega} \abs{f(x)}^k \right)^{1/k},
\end{equation*}
and
\begin{equation*}
    \norm{f}_\infty = \max_{x \in \Omega} \abs{f(x)}.
\end{equation*}
\item It is a consequence of Jensen's inequality that for $1 \leq k \leq k'$ and any function $f$ we have
$$ \|f\|_{k} \leq \|f\|_{k'}.$$
\end{itemize}

\subsection{Convolutions}

\begin{itemize}
\item We specialize the finite set $\Omega$ from above to be a finite abelian group $G$. Given two functions $f, g : G \rightarrow \R$, we define their convolution with the following normalization:
\begin{equation*}
    (f * g)(x) := \E_{y \in G} f(y)g(x - y)  = \E_{y \in G} f(x - y)g(y) = 
    \frac{1}{|G|}\sum_{
    \substack{
    y,z \in G \\
    y + z = x
    }}
    f(y)g(z).
\end{equation*}
\item For a density function $B$, we can interpret $B * g$ as
$$ (B * g)(x) := \E_{y \in G} B(y) g(x - y) = \E_{b \sim B} g(x - b) .$$

\item For two real-valued function $f,g$, we define the \textit{cross-correlation} $f \star g$  as the convolution of $f(-x)$ with $g$:
\begin{equation*}
    (f \star g)(x) := \E_{y \in G} f(-y) g(x - y) = \E_{y \in G} f(y) g(x + y),
\end{equation*}
and we have
$$ (B \star g)(x) = \E_{b \sim B}g(x + b).$$
\item If $A,B$ are densities corresponding to (independent) random variables $a, b$, then
as a consequence of our normalization conventions, $A * B$ is again a density. Indeed, it is the density corresponding to the random variable $a + b \in G$. Similarly, $A \star B$ is the density corresponding to the random variable $b-a \in G$.
\item We note the identity  $\ip{f * g}{h} = \ip{f}{g \star h}$.
\end{itemize}

\subsection{Fourier analysis on finite abelian groups}
\begin{itemize}
    \item Let $\set{e_{\alpha}(\cdot)}{\alpha \in G}$ denote the set of characters of the finite abelian group $G$. Each character is a function from $G$ to the set of complex numbers of modulus 1. 
    The product of two characters is again a character:
    For $\alpha, \beta \in G$, $e_{\alpha} \cdot e_{\beta} = e_{\alpha + \beta} $. Beyond this, the main important properties of the characters (for the development of the Fourier expansion) are
    \begin{enumerate}
        \item[(i)] orthogonality: $\ip{e_{\alpha}}{e_{\beta}} = 
        \E_{x \in G} \overline{e_{\alpha}(x)} e_{\beta} (x) = 
        \E_{x \in G}  e_{\alpha}(-x)  e_{\beta}(x) =  \ind{\alpha = \beta}$, and
        \item[(ii)] symmetry: $e_{\alpha}(x) = e_{x}(\alpha). $
    \end{enumerate}
    In the case that the group $G$ is presented explicitly as $G = \Z_{N_1} \times \Z_{N_2} \times \cdots \times \Z_{N_r}$
    we can consider the following concrete description:
    $$ e_{\alpha}(x) := \textnormal{exp} \left( 2 \pi i \sum_{j=1}^r \frac{\alpha_j x_j}{N_j} \right).$$
    \item For a set $A \subseteq G$, we define the fourier coefficients
    $$ \widehat{A}(\alpha) := \ip{A}{e_\alpha} = 
    \E_{a \in A} e_{\alpha}(a).$$
    \item Note that $\widehat{A}(-\alpha) = \ip{A}{e_{-\alpha}} = \overline{\ip{A}{e_{\alpha}}} = \overline{\widehat{A}(\alpha)}$.
    \item For a function $f : G \rightarrow \R$, we define the fourier coefficients
    $$ \widehat{f}(\alpha) := \ip{f}{e_{\alpha}} = 
    \E_{y \in G} f(y) e_{\alpha}(y) .$$
    \item We can express any function by its fourier expansion
    $$ f(x) = \sum_{\alpha \in G} \widehat{f}(\alpha) e_{\alpha}(-x) .$$
    \begin{proof}
    \begin{align*}
    \sum_{\alpha} \left( \E_y f(y) e_{\alpha}(y) \right) e_{\alpha}(-x) &=
    \sum_{y} \E_\alpha  f(y)   e_{\alpha}(y) e_{\alpha}(-x)  \\
    &= \sum_{y}  f(y)  \E_\alpha e_{y}(\alpha) e_{x}(-\alpha) \\
    &= \sum_{y}  f(y)  \ind{y = x} \\
    &= f(x). \qedhere
    \end{align*}
    \end{proof}
    \item We have the following fourier-analytic identities for real-valued functions $A(x),B(x),f(x)$.
    In general they can all be verified, with no creativity required, by the following process: express any function appearing by its Fourier expansion, expand any products of sums appearing into a sum of products, and then use the identity $\E_{x \in G}  e_{\alpha}(x)  = \ind{\alpha = 0}$ to evaluate any expectations. 
    \begin{enumerate}
        \item $\ip{A}{f} = 
        \sum_{\alpha}\widehat{A}(-\alpha)   \widehat{f}(\alpha) = \sum_{\alpha} \hat{A} (\alpha) \hat{f}(-\alpha)  $
        \item $(A * B)(x) = \sum_{\alpha} \widehat{A}(\alpha) \widehat{B}(\alpha) e_{\alpha}(-x)$
        \item $(A \star B)(x) = \sum_{\alpha} \widehat{A}(-\alpha) \widehat{B}(\alpha) e_{\alpha}(-x)$
        \item $(A \star A)(x) = \sum_{\alpha} | \widehat{A}(\alpha) |^2 e_{\alpha}(-x)$
        \item $\|A\|_{2}^2 = \ip{A}{A} = \sum_{\alpha} |\widehat{A}(\alpha) |^2 $
        \item $\|A \star A\|_2^2 = \sum_{\alpha} |\widehat{A}(\alpha) |^4$
    \end{enumerate}
\end{itemize}

\section{Proof of \cref{II}, and Sifting} \label{sec:sifting-and-proof-of-II}


\subsection{Sifting lemma} \label{sec:sifting-lemma}

It will be convenient to set aside the following generic argument.

\begin{prop}[Weighted pigeonhole principle, or ``first-moment method"] \label{weighted-php}
  Fix some nonnegative numbers $g_1, g_2, \ldots g_m$ and $h_1, h_2,\ldots, h_m$, with $\sum_i h_i > 0$.
  Let 
  $$\eta := \frac{\sum_{i=1}^m g_i}{\sum_{i=1}^m h_i} $$
  and
  $$\mu(h) := \frac{1}{m} \sum_{i=1}^m h_i. $$
  \begin{enumerate}
      \item[(i)] There exists a choice of $j \in [m]$ with $\frac{g_j}{h_j} \leq \eta$ (and not of the form $\frac{0}{0}$).
      \item[(ii)] Furthermore, there exists a choice of $j$ with both
      $$ \frac{g_j}{h_j} \leq 2 \eta \text{ and } h_{j} \geq \tfrac{1}{2} \mu(h)$$
      \item[(iii)] More generally, suppose that $H, \tau > 0$ are such that
      $$ \sum_{i \: : \: h_{i} \geq H} h_i \geq \tau \sum_{i} h_i.$$
      Then there is a choice of $j$ with both
      $$ \frac{g_j}{h_j} \leq \frac{\eta}{\tau} \text{ and } h_{j} \geq H.$$
  \end{enumerate}
\end{prop}
\begin{proof}
The degenerate case $\eta = 0$ is easy to handle: we simply pick the largest value $h_j$. So suppose $\eta \neq 0$ and consider the first claim.
We discard any indices $i$ where both $g_i = h_i = 0$, noting that this does not change the value of $\sum_i g_i$ or $\sum_i h_i$. 
Consider the equality 
$ \sum_{i} g_i = \eta \sum_i h_i. $
If the desired conclusion does not hold, we have $g_i > \eta h_i$ for all $i$, and so
$ \sum_i (g_i - \eta h_i) > 0,$
a contradiction. For the second claim, we apply the first claim to the modified sequence 
$$h_i' := h_i \cdot \ind{h_i \geq \mu / 2} .$$
For the final claim, we apply the first claim to the modified sequence
\begin{equation*}
 h_i'' := h_i\cdot  \ind{h_i \geq H}.  \qedhere
\end{equation*}
\end{proof}

In preparation to make some combinatorial arguments we briefly switch to the counting measure. 
This also allows us to also handle finite subsets $A$ of infinite groups $G$, noting that summations such as
$\sum_{x \in G} R_{A}(x) $
are sensible because $R_A$ is finitely supported. 

\begin{lem}[Sifting lemma -- counting formulation] Consider
\begin{itemize}
\item a finite subset $A$ of an abelian group $G$,
\item a function $f : G \rightarrow \R_{\geq 0}$, and
\item an integer $k \geq 2$. 
\end{itemize}

There is a subset $A' \subseteq A$ with
$$ \frac{1}{|A'|^2} \sum_{a,b \in A'}f(a - b) \leq 2 \cdot \frac{\sum_{x} R_{A}^{-}(x)^k f(x)}{\sum_{x} R_{A}^-(x)^k}$$
and
$$ |A'| \geq \tfrac{1}{2} \cdot \frac{\sum_{x} R_{A}^{-}(x)^k}{|A|^k}. $$
Specifically, $A'$ is of the form
$$ A'(s) := A \cap (A + s_1) \cap \cdots \cap (A + s_{k-1})$$
for some $s = (s_1, s_2 ,\cdots, s_{k-1}) \in G^{k-1}$. 
\end{lem}

\begin{proof}

For $s \in G^{k-1}$, let
$$ A'(s) :=  A \cap (A + s_1) \cap (A + s_2) \cap \cdots \cap (A + s_{k-1}). $$
We wish to understand the quantity
$$ \frac{\sum_{a,b \in A'(s)} f(a-b)}{|A'(s)|^2}  $$
for various choices of $s$.
Define
$$ g(s) := \sum_{a,b \in A'(s)} f(a - b) $$
and 
$$ h(s) := |A'(s)|^{2}.$$ 

We would like to apply the weighted pigeonhole principle to find a suitable choice of $s$, and so we need to compute $\sum_{s} g(s)$ and $\sum_{s} h(s)$. 
Before we start the computation we point out ahead-of-time the following combinatorial identities:
$$ R_{A}^{-}(a-b) = |A \cap (A + a-b)| = |(A - a) \cap (A - b)| = \sum_{t \in G} \1_{A-a}(t) \1_{A-b}(t) = \sum_{t \in G} \1_{A+t}(a) \1_{A+t}(b).$$
We have
\begin{align*}
    \sum_{x} R_{A}^-(x)^k f(x) &=
    \sum_{a,b \in A} R_{A}^-(a-b)^{k-1} f(a-b) \\
    &= \sum_{a,b \in A} \left(\sum_{t} \1_{A+t}(a) \1_{A+t}(b) \right)^{k-1} f(a - b) \\
    &= \sum_{a,b \in A} \sum_{s_1, \ldots s_{k-1}} 
    \left(  \prod_{j=1}^{k-1} \1_{A+s_j}(a) \right)  \cdot
   \left(  \prod_{j=1}^{k-1}  \1_{A+s_j}(b) \right)  \cdot
    f(a - b) \\
    &= \sum_{s_1, \ldots s_{k-1}}  \sum_{a,b \in A} 
    \left(  \1_{\bigcap_{j=1}^{k-1} (A+s_j)}(a) \right)  \cdot
   \left(  \1_{\bigcap_{j=1}^{k-1} (A+s_j)}(b) \right)  \cdot
    f(a - b) \\
    &= \sum_{s_1, \ldots, s_{k-1}}  \sum_{a,b \in A'(s)} f(a-b) \\
    &= \sum_{s} g(s).
\end{align*}
A special case of this same calculation (with $f \equiv 1$) also gives
$$ \sum_{x} R_{A}^-(x)^k = \sum_{s} h(s)  = \sum_s |A'(s)|^2 .$$
Thus we have
$$ \eta := \frac{\sum_s g(s)}{\sum_s h(s)} = \frac{\sum_x R(x)^k f(x)}{\sum_x R(x)^k}. $$
We remark that in the case that $G$ is finite and $A$ is a dense subset of size $|A| = \delta |G|$, we can obtain a quite satisfactory conclusion already by applying part (ii) of the weighted pigeonhole principle. Indeed, for uniformly random $s \in G^{k-1}$, we have average size
$$ \E_{s} |A'(s)|^2 = \E_{s} h(s) = \frac{1}{|G|^{k-1}} \sum_x R_{A}^{-}(x)^k = \frac{1}{|G|^{k-2}} \E_x R_{A}^-(x)^k = \delta^{2k} |G|^2  \E_x (A \star A)(x)^k \geq \delta^{2k} |G|^2 .$$
However, we consider the following argument which is better in general. 
We wish to determine a value $M$ which is as large as possible and also satisfies
$$ \sum_{s \: : \: |A'(s)| \geq M} |A'(s)|^2 \geq \tfrac{1}{2}  \cdot \sum_{s} |A'(s)|^2. $$
For any choice of $M$ we have
\begin{align*}
 \sum_{s \: : \: |A'(s)| \leq M} |A'(s)|^2 &\leq M \cdot  \sum_{s} |A'(s)| \\
 &= M \cdot \sum_{s} \sum_{a \in A} \1_{A+s_1}(a) \cdot \1_{A + s_2}(a) \cdots \1_{A + s_{k-1}} (a) \\
 &= M \cdot \sum_{a \in A} \sum_s  \1_{A-a}(s_1) \cdot \1_{A-a}(s_2) \cdots \1_{A-a}(s_{k-1}) \\
 &= M \cdot |A|^k.
\end{align*}
This shows that we may take 
$$ M = \tfrac{1}{2} \cdot \frac{\sum_{s} |A'(s)|^2}{|A|^k}  = \tfrac{1}{2} \cdot \frac{\sum_x R_{A}^-(x)^k}{|A|^k} .$$
We conclude by part (iii) of the weighted pigeonhole principle. 
\end{proof}

\begin{remark}
We note that the combinatorial argument above which shows that
$$ \sum_{x} R_{A}^-(x)^k f(x) = \sum_{s} \sum_{a,b \in A'(s)} f(a-b) = \sum_{s}  \sum_{x} R_{A'(s)}^-(x) f(x) $$
is in fact valid for any function $f$. This more plainly means that we have the identity
$$ {R_{A}^{-}}(x)^k \equiv \sum_{s} R_{A'(s)}^-(x)$$
for the $k$-th power of $R_A^-$. 
Given this, we can summarize the remaining points of the argument as follows.
\begin{itemize}
\item For any $f$, we can interpret the sum $\sum_{s,x} R_{A'(s)}^-(x) f(x)$ in two ways: firstly as $$\sum_{s} \left( \sum_{x} R_{A'(s)}^-(x) f(x) \right)$$ but also as $ \sum_{x} \left( \sum_{s}R_{A'(s)}^-(x) \right) f(x) = \sum_x R_{A}^-(x)^k f(x)$.
\item For the particular case $f \equiv 1$, we may further interpret $\sum_{x} R_{A'(s)}^-(x) = |A'(s)|^2$.
\item If we have a reasonable bound on the size of $\set{s}{|A(s')| \neq 0}$, we can already infer a reasonable bound on the median value of $|A'(s)|^2$. 
\item We also have the identity $\sum_{s} |A'(s)| = |A|^k$, which is relevant for obtaining a better estimate for the median value of $|A'(s)|^2$. 
\end{itemize}
\end{remark}

\begin{proof}[Proof of the extended Pre-BSG Lemma (\cref{extended-pre-bsg})]
We apply the sifting lemma straightforwardly to the sub-level set indicator
$$ f(x) = \ind{R_{A}^-(x) \leq c \cdot \kappa^{\frac{1}{k-1}} \cdot |A|},$$
where
$$ \kappa = \frac{\sum_{x} R_{A}^-(x)^k}{|A|^{k+1}}. $$
We check that
$$ \sum_x R_{A}^-(x)^{k} f(x) \leq c^{k-1} \cdot \kappa \cdot |A|^{k-1} \sum_{x} R_{A}^-(x) = c^{k-1} \cdot \kappa \cdot |A|^{k+1},$$
so indeed
$$ \frac{\sum_x R_{A}^-(x)^k f(x)}{\sum_x R_{A}^-(x)^k} =  \frac{\sum_x R_{A}^-(x)^k f(x)}{\kappa |A|^{k+1}} \leq c^{k-1}. $$
We also have
\begin{equation*}
|A'| \geq \tfrac{1}{2} \cdot \frac{\sum_x R_{A}^-(x)^k}{|A|^k} = \tfrac{1}{2} \cdot \frac{\kappa |A|^{k+1}}{|A|^k}   = \tfrac{1}{2} \cdot \kappa \cdot |A|. \qedhere
\end{equation*}
\end{proof}

We proceed to state a density formulation of the sifting lemma for subsets $A \subseteq G$ of a finite groups $G$ with of $|A| = \delta |G|$.
We use the notation
$$ D^{\wedge k}(x) := \left( \frac{D(x)}{\|D\|_k}\right)^k $$
to denote the unique density function proportional to $D^k$.
We also point out the translation
$$ \frac{\sum_{x \in G} R_{A}^-(x)^k}{|A|^k} = \frac{|A|^k}{|G|^k} \sum_{x \in G} (A \star A)(x)^k = \frac{|A|^k}{|G|^{k-1}} \E_x (A \star A)(x)^k = \delta^k \cdot |G| \cdot \|A \star A\|_{k}^k. 
$$ 

\begin{lem}[Sifting lemma -- density formulation] \label{sifting-lem-density-formulation} Consider
\begin{itemize}
\item a set  $A \subseteq G$ of size $|A| = \delta |G|$,
\item a function $f : G \rightarrow \R_{\geq 0}$, and
\item an integer $k \geq 2$. 
\end{itemize}

There is a subset $A' \subseteq A$ with
$$ \ip{A' \star A'}{f} \leq 2 \cdot \ip{(A \star A)^{\wedge k}}{f}$$
and
$$ \frac{|A'|}{|G|} \geq \tfrac{1}{2} \cdot \delta^k \cdot \|A \star A\|_{k}^k \geq \tfrac{1}{2} \cdot \delta^k. $$
\end{lem}

One can easily derive a formulation of the extended Pre-BSG Lemma for the dense setting which corresponds e.g.\ to either part of \cref{compressions-escape}. We consider the following formulation which is tailored specifically to our intended application. 

\begin{cor}[Sifting a robust witness] \label{sift-robust-witness}
Suppose that $A \subseteq G$ has size $|A| \geq 2^{-d} |G|$, and
$$ \|A \star A\|_k \geq 1 + \eps$$
for some $k \geq 1$ and $\eps > 0.$ 
Consider 
$$ S := \set{x \in G}{(A \star A) \leq 1 + \eps /2}.$$

Let $\overline{\eps} = \min \{ 1, \eps \}$. There is a subset $A' \subseteq A$ with
$$ \ip{A' \star A'}{\1_S} \leq \frac{\overline{\eps}}{2^4} $$
and
$$ \frac{|A'|}{|G|} \geq 
\begin{cases}
2^{-O(d k)} \cdot \eps^{O(d / \eps)}  &\textnormal{ when } \eps < 1/2  \\
2^{-O(d k)} &\textnormal{ when } \eps \geq 1/2.
\end{cases}
$$
\end{cor}
\begin{proof}
Apply the sifting lemma with $f := \1_S$ and
\begin{equation*} k' = \begin{cases}
\lceil k +  2 \lg(32/\eps) / \eps \rceil  &\textnormal{ when } \eps < 1/2  \\
\lceil k + 20 \rceil &\textnormal{ when } \eps \geq 1/2. \qedhere
\end{cases} 
\end{equation*}
\end{proof}
The point of the quantity $\overline{\eps} = \min \{1, \eps \}$ is that for all $\eps \geq 0$, $\lambda \geq 1$ we have
$$ \left(1 + \frac{\eps}{\lambda} \right) \left(1 - \frac{\overline{\eps}}{4 \lambda} \right) \geq 1 + \frac{\eps}{2 \lambda}. $$
So $A' \star A'$ above is a robust witness to $A \star A \gg 1 + \eps/4$ with some room to spare: for 
$$f := \ind{A \star A \geq 1 + \eps/2}$$ 
we have
$$ \ip{A' \star A'}{A \star A} \geq \left(1 + \frac{\eps}{2} \right) \ip{A' \star A'}{f} \geq \left(1 + \frac{\eps}{2} \right)  \left(1 - \frac{\overline{\eps}}{2^3} \right)  \geq 1 + \frac{\eps}{2^2}. $$

\begin{thm} \label{II-part-2}
Let $A \subseteq \F_q^n$ be a set of size $|A| \geq 2^{-d} |\F_q^n|$, where $d \geq 1$.
Suppose that
$$\|A \star A\|_{k} \geq 1 + \eps$$
for some $k \geq 1$. We have
$$ \|A\|_{\perp, r} \geq 1 + \frac{\eps}{4} $$
for some \begin{equation*} r = \begin{cases}
 O(d^4 k^4)  &\textnormal{ when } \eps \geq 1/2  \\
 O ( d^4 k^4 / \eps^2 + d^4  \lg(1/\eps)^4 / \eps^6) &\textnormal{ when } \eps < 1/2. \qedhere
\end{cases} 
\end{equation*}
In either case we also have the related conclusion
$$ \max_{\substack{W \subseteq \F_q^n \\ \textnormal{dim}(W) \leq r}} \|P_W A - 1 \|_{2}^2 \geq \frac{\eps}{4} $$
where $P_W A := W^{\perp} * A$.
\end{thm}

We note that the latter conclusion here is (qualitatively) stronger for small $\eps$. 
Indeed, we have $\|P_W A \|_2^2 = 1 + \|P_W A - 1 \|_{2}^2$, and in light of the alternative characterization
$$ \|A\|_{\perp,r} = \max_{\substack{W \subseteq \F_q^n \\ \textnormal{dim}(W) \leq r}} \|P_W A\|_{\infty},$$
we have the lower bound
$$ \|A\|_{\perp,r} \geq \|P_W A\|_{\infty} \geq \|P_W A\|_2 \geq \sqrt{1 + \frac{\eps}{4}} \approx 1 + \frac{\eps}{8}  $$
for some small subspace $W$.

\begin{proof}
In the case $\eps \leq \frac{1}{2}$,
we replace $k$  by  $\max \{k, \lg(1/\eps)/\eps \}$ for convenience, noting that $\|A \star A\|_{k}$ does not decrease as a result.
We consider
$$ f := \ind{A \star A \geq 1 + \tfrac{\eps}{2}} $$
and (using \cref{sift-robust-witness}) we sift a robust witness $A' \star A'$ with
$$ \ip{A' \star A'}{f} \geq 1 - \frac{\overline{\eps}}{2^4}.$$
and density $|A'|/|G| \geq 2^{-O(d k)}$.
We apply Sanders' invariance lemma (\cref{sanders-invar-appendix}) to obtain a linear subspace $V$ with
$$ \ip{V * A' \star A'}{f} \geq \ip{A' \star A'}{f} - \frac{\overline{\eps}}{2^4} \geq 1 - \frac{\overline{\eps}}{2^3} $$
and codimension $r \leq O(d^4 k^4 / \overline{\eps}^2). $
We conclude that
$$\ip{V * A' \star A'}{A \star A} \geq \left(1 + \frac{\eps}{2} \right) \ip{V * A' \star A'}{f} \geq 
\left(1 + \frac{\eps}{2} \right) \left(1 - \frac{\overline{\eps}}{2^3} \right) \geq 1 + \frac{\eps}{4}.
$$
This gives
$$ \ip{V'}{A \star A} \geq 1 + \frac{\eps}{4} $$
for some affine subspace $V'$ and 
$$ \ip{V'}{A} \geq 1 + \frac{\eps}{4} $$
for some $V'' = V' + a$. This proves the first claim.

To prove the latter claim, we depart from the argument above when we reach
$$ \ip{V * A' \star A'}{A \star A} \geq 1 + \frac{\eps}{4}. $$
It is pleasant, although not ultimately crucial, to interpret this quantity Fourier-analytically. Letting $W = V^{\perp}$, we can express
$$ \ip{V * A' \star A'}{A \star A} = \sum_{\alpha \in W} |\hat{A'}(\alpha)|^2 |\hat{A}(\alpha)|^2 = 1 + \sum_{\substack{\alpha \in W \\ \alpha \neq 0}} |\hat{A'}(\alpha)|^2 |\hat{A}(\alpha)|^2 ,
$$
and so we have
$$ \sum_{\substack{\alpha \in W \\ \alpha \neq 0}} |\hat{A'}(\alpha)|^2 |\hat{A}(\alpha)|^2 \geq \frac{\eps}{4}. $$
Using the trivial bound $|\hat{A'}(\alpha)| \leq 1$ we infer
$$ \sum_{\substack{\alpha \in W \\ \alpha \neq 0}} |\hat{A}(\alpha)|^2 \geq \frac{\eps}{4}, $$
and indeed this quantity is the same as $\|P_W A - 1\|_{2}^2$.  \qedhere
\end{proof}

\subsection{Sifting for general convolutions}

We now address the problem of obtaining
$ \|A\|_{\perp,r} \geq 1 + \frac{\eps}{8} $
from an assumption
$ \|A\|_{*,k} \geq 1 + \eps.$
To begin with we have by definition that 
$ \ip{A}{B * C} \geq 1 + \eps$
for some large sets $|B|,|C| \geq  2^{-k} |G|$.
We argue that
$$ \ip{A}{B * C} = \ip{A \star B}{C} \leq 2^{k/k'} \|A \star B\|_{k'}$$
for any choice of $k'$. Choosing $k' = 4 \lg(e) k/\overline{\eps}$ gives
$$ \|A \star B\|_{k'} \geq 2^{-k/k'}(1 + \eps) \geq \left( 1 - \frac{\overline{\eps}}{4}\right)(1 + \eps) \geq 1 + \frac{\eps}{2}. $$

At this point we argue roughly as before -- we just need a variant of the sifting lemma for general convolutions $A \star B$.

\begin{lem}[Sifting general convolutions] \label{sifting-general} Consider
\begin{itemize}
\item an abelian group $G$,
\item two finite sets $A,B \subseteq G$,
\item a function $f : G \rightarrow \R_{\geq 0}$, and
\item an integer $k \geq 2$. 
\end{itemize}

There are subsets $A' \subseteq A$ and $B' \subseteq B$ with
$$ \frac{1}{|A'||B'|} \sum_{a \in A'} \sum_{b \in B'} f(a - b) \leq 2 \cdot \frac{\sum_{x} R_{A,B}^{-}(x)^k f(x)}{\sum_{x} R_{A,B}^-(x)^k}$$
and
$$ |A'| \geq \tfrac{1}{4} \cdot \frac{\sum_{x} R_{A,B}^{-}(x)^k}{|B|^k}, $$
$$ |B'| \geq \tfrac{1}{4} \cdot \frac{\sum_{x} R_{A,B}^{-}(x)^k}{|A|^k},$$
Specifically, $A'$ and $B'$ are of the form
$$ A'(s) := A \cap (A + s_1) \cap \cdots \cap (A + s_{k-1}),$$
$$ B'(s) := B \cap (B + s_1) \cap \cdots \cap (B + s_{k-1})$$
for some $s = (s_1, s_2 ,\cdots s_{k-1}) \in G^{k-1}$. 
\end{lem}

 To prove this we will need one more form of the weighted pigeonhole principle. 

\begin{prop}[Continuation of \cref{weighted-php}]   {\ } \\
Under the hypotheses of \cref{weighted-php}, 
\begin{enumerate}
\item[(iv)] Suppose that $I \subseteq [m]$ and $\tau > 0$ are such that
$$ \sum_{i \in I} h_i \geq \tau \sum_{i} h_i. $$
Then there is a choice of $j$ with both
$$ \frac{g_j}{h_j} \leq \frac{\eta}{\tau} \text{ and } j \in I. $$
\end{enumerate}
\end{prop}
\begin{proof}
Consider the modified sequence $h_i' := h_i \cdot \ind{i \in I}$. \qedhere
\end{proof}

\begin{proof}[Proof of \cref{sifting-general}]

For $s \in G^{k-1}$, let
$$ A'(s) :=  A \cap (A + s_1)  \cap \cdots \cap (A + s_{k-1}), $$
$$ B'(s) := B \cap (B + s_1) \cap \cdots \cap (B + s_{k-1}) $$
Note that
$$ R_{A,B}^{-}(a-b) = |A \cap (B + a-b)| = |(A - a) \cap (B - b)| = \sum_{t \in G} \1_{A-a}(t) \1_{B-b}(t) = \sum_{t \in G} \1_{A+t}(a) \1_{B+t}(b).$$
We express
\begin{align*}
    \sum_{x} R_{A,B}^-(x)^k f(x) &= \sum_{a \in A} \sum_{b \in B}  \left(\sum_{t} \1_{A+t}(a) \1_{B+t}(b) \right)^{k-1} f(a - b) \\
    &= \sum_{s} \sum_{a \in A} \sum_{b \in B} 
    \left(  \1_{\bigcap_{j=1}^{k-1} (A+s_j)}(a) \right)  \cdot
   \left(  \1_{\bigcap_{j=1}^{k-1} (B+s_j)}(b) \right)  \cdot
    f(a - b) \\
    &=: \sum_{s} \Phi_f(s).
\end{align*}
We apply the weighted pigeonhole principle to the ratio
$$ \frac{\Phi_f(s)}{ \Phi_{1}(s)} .$$
to find a suitable choice of $s$.
Specifically, we apply variant (iv) with some index-set of the form
$$ S = \set{s}{ |A'(s)| \geq M_{A} \textnormal{ and } |B'(s)| \geq  M_{B}}. $$
For any choice of $M_A,M_B$ we can estimate
$$ \sum_{s \: : \: s \not\in S} \Phi_{1}(s) =  \sum_{s \: : \: s \not\in S} |A'(s)||B'(s)| \leq
M_A \sum_{s} |B'(s)| + M_B \sum_{s} |A'(s)| =  M_A |B|^k + M_B |A|^k.
$$
Choosing 
$$ M_A := \tfrac{1}{4} \cdot \frac{\sum_{s} \Phi_s(s)}{|B|^k} = \tfrac{1}{4} \cdot \frac{\sum_x R_{A,B}^-(x)^k}{|B|^k}, $$
$$ M_B := \tfrac{1}{4} \cdot \frac{\sum_{s} \Phi_s(s)}{|A|^k} = \tfrac{1}{4} \cdot \frac{\sum_x R_{A,B}^-(x)^k}{|A|^k} $$
is sufficient to ensure that
\begin{equation*}
\sum_{s \in S} \Phi_{1}(s) \geq \frac{1}{2} \cdot \sum_{s} \Phi_{1}(s). \qedhere
\end{equation*}
\end{proof}

The size guarantees here for $A'$ and $B'$ have a pleasant interpretation in the density formulation. 
Let $G$ be a finite group and let $A \subseteq G$ be a set of size $|A| = \delta |G|$.
We have the translation
$$ \frac{\sum_x R_{A,B}^-(x)^k}{|B|^k} = \frac{|A|^k}{|G|^k} \sum_{x} (A \star B)(x)^k = \delta^{k} \cdot \|A \star B\|_{k}^k \cdot |G| \geq \delta^{k} \cdot |G|. 
$$

\begin{lem}[Sifting general convolutions -- density formulation] Consider
\begin{itemize}
\item a finite abelian group $G$,
\item two sets $A,B \subseteq G$ of sizes $|A| = \delta_A |G|$, $|B| = \delta_B |G|$.
\item a function $f : G \rightarrow \R_{\geq 0}$, and
\item an integer $k \geq 2$. 
\end{itemize}

There are subsets $A' \subseteq A$ and $B' \subseteq B$ with
$$ \ip{A' \star B'}{f}\leq 2 \cdot \ip{(A \star B)^{\wedge k}}{f}$$
and
$$ \frac{|A'|}{|G|} \geq \tfrac{1}{4} \cdot \delta_A^k \cdot \|A \star B\|_{k}^k \geq \tfrac{1}{4} \cdot \delta_A^k ,$$
$$ \frac{|B'|}{|G|} \geq  \tfrac{1}{4} \cdot \delta_B^k \cdot \|A \star B\|_{k}^k \geq \tfrac{1}{4} \cdot \delta_B^k .$$
\end{lem}

\begin{thm} \label{II-part-1}
Let $A \subseteq \F_q^n$ be a set of size $|A| \geq 2^{-d} |\F_q^n|$, where $d \geq 1$.
Suppose that
$$\|A\|_{*,k} \geq 1 + \eps $$
for some $k \geq 1$. We have
$$ \|A\|_{\perp, r} \geq 1 + \frac{\eps}{8} $$
for some \begin{equation*} r = \begin{cases}
O(d k^7)  &\textnormal{ when } \eps \geq 1  \\
 O( d k^7 / \eps^6 + d k^3 \lg(2/\eps)^4 / \eps^6) &\textnormal{ when } \eps < 1. \qedhere
\end{cases} 
\end{equation*}
\end{thm}

\begin{proof}
It follows from
$ \|A\|_{*,k} \geq 1 + \eps $
that for $p = 4 \lg(e) k/\overline{\eps} + \lg(2/\overline{\eps}) / \overline{\eps}$ we have
$ \|A \star B\|_{p} \geq 1 + \frac{\eps}{2} $
for some set of size $|B| \geq 2^{-k} |\F_q^n|$. 
From here we argue similarly as in the proof of \cref{II-part-2}.
This time we suppress some quantitative details.
Let 
$ f := \ind{A \star B \geq 1 + \eps/4}.$
We sift $A \star B$ to find a convolution $A' \star B'$ witnessing
$\ip{A' \star B'}{f} \approx 1$
with $|A'| \geq 2^{-O(d p)} |\F_q^n|$ and $|B'| \geq 2^{-O(k p)} |\F_q^n|. $
We apply Sanders' invariance lemma to find a linear subspace $V$ with
$ \ip{V * A' \star B'}{f} \approx 1 $
and codimension $r \leq O(dp \cdot k^3p^3 / \overline{\eps}^2) = O(d k^3 p^4 / \overline{\eps}^2)$.
We conclude that 
$ \ip{V'}{A \star B} \geq 1 + \frac{\eps}{8} $
for some translate $V'$ of $V$ and
$ \ip{V''}{A } \geq 1 + \frac{\eps}{8} $
for some translate $V'' = V' + b$. 
\end{proof}

\subsection{Local variants}

We proceed to give some local variants of the statements above which will be needed for setting $A \subseteq [N] \subseteq \Z$. 

\begin{lem}[Sifting lemma -- a local variant] \label{sifting-local}
Consider
\begin{itemize}
    \item a finite abelian group $G$,
    \item a set $A \subseteq G$.
    \item some additional subsets $B,C \subseteq G$, 
    \item a function $f : G \rightarrow \R_{\geq 0}$, and 
    \item an integer $k \geq 1$. 
\end{itemize}
There is some choice of $s = (s_1, s_2, \ldots, s_k) \in G^k$ giving rise to subsets $B' \subseteq B$ and $C' \subseteq C$ of the form
\begin{itemize}
    \item $B' = B \cap (A - s_1) \cap (A - s_2) \cap  \cdots \cap (A - s_{k})$ and
    \item $C' = C \cap (s_1 - A) \cap (s_2 - A) \cap \cdots \cap (s_k - A)$
\end{itemize}
which satisfy 
$$    \frac{1}{|B'| |C'|} \sum_{b \in B'} \sum_{c \in C'} f(b + c) \leq   2 \cdot \frac{\sum_x R_{B,C}(x) R_A^{-}(x)^k f(x)}{\sum_{x} R_{B,C}(x) R_A^{-}(x)^k }$$
and
$$|B'||C'| \geq \tfrac{1}{2} \cdot \frac{\sum_{x} R_{B,C}(x) R_A^-(x)^k }{|G|^k} . $$
\end{lem}
\begin{proof}
We express
\begin{align*}
    \sum_{x \in G} R_{B,C}(x) R_A^-(x)^k \cdot f(x)  &= 
    \sum_{b \in B} \sum_{c \in C} R_A^-(b + c)^k  \cdot f(b+c)  \\
    &= \sum_{b \in B} \sum_{c \in C} | A \cap (A + b + c) |^k  \cdot f(b+c)  \\
    &= \sum_{b \in B} \sum_{c \in C} | (A - b) \cap (A + c) |^k \cdot f(b+c)  \\
    &= \sum_{b \in B} \sum_{c \in C} \left( \sum_{t} \1_{A-b}(t) \1_{A+c}(t) \right)^k \cdot f(b+c)  \\
    &= \sum_{b \in B} \sum_{c \in C} \left( \sum_{t} \1_{A-t}(b) \1_{t-A}(c)\right)^k  \cdot f(b+c)  \\
    &= \sum_{s} \sum_{b \in B} \sum_{c \in C} 
    \left(  \prod_{j=1}^{k}  \1_{A-s_j}(b)  \right)  \cdot
   \left(  \prod_{j=1}^{k}  \1_{s_j-A}(c) \right)   \cdot
    f(b+c) \\
    &=: \sum_{s} \Phi_{f}(s) .
\end{align*}
We then apply variant (ii) of the weighted pigeonhole principle to the ratio 
$$ \frac{\Phi_f(s)}{\Phi_1(s)} $$
to find an appropriate choice of $s$. \qedhere
\end{proof}

We need the following notation.

\begin{definition}[Weighted $k$-(semi)norm] \label{weighted-k-norm}
For a density $D$ on an arbitrary finite set $\Omega$, and $k \geq 1$, we use the notation
$$ \| f \|_{k, D} := \ip{D}{|f|^k}^{1/k} =  \left( \E_{x \sim D} |f(x)|^k \right)^{1/k} .$$
\end{definition}

\begin{cor}[Sifting a local robust witness] \label{local-witness}
Consider
\begin{itemize}
    \item a finite abelian group $G$ of size $N$,
    \item a subset $A \subseteq G$ of size $|A| = 2^{-d} N$,
    \item some additional subsets $B,C \subseteq G$, and
    \item $k \geq 1,$ $\eps \in [0,1]$.
\end{itemize}
Suppose that
$$ \|A \star A\|_{k, B * C} \geq 1 + 2 \eps, $$
and define the sublevel-set indicator function
$$f(x) := \ind{\big(A \star A\big)(x) \leq 1 + \eps}.$$
For any integer $k' \geq k$,  there are subsets $B' \subseteq B$ and $C' \subseteq C$
with
$$ \ip{B' * C'}{f} \leq 2 \cdot \left(\frac{1 + \eps}{1 + 2 \eps}\right)^{k'} \leq 2 \cdot 2^{-\eps k'/2} $$
and
$$\frac{|B'||C'|}{|B||C|} \geq  \tfrac{1}{2} \cdot 2^{-2dk'} .$$
In particular, 
$$\frac{|B'|}{|B|}, \frac{|C'|}{|C|}  \geq \tfrac{1}{2}  \cdot 2^{-2dk'} .$$
\end{cor}
\begin{proof}
We  straightforwardly  apply \cref{sifting-local} -- it remains to interpret the result. We have
\begin{align*}
    \ip{B' * C'}{f} &\leq 2 \cdot \frac{\sum_x R_{B,C}(x) R_A^{-}(x)^k f(x)}{\sum_{x} R_{B,C}(x) R_A^{-}(x)^k } \\
    &=  2 \cdot \frac{\E_x \big(B * C\big)(x) \big(A \star A\big)(x)^{k'} f(x)}{\ip{B * C}{(A \star A)^{k'}}} \\
    &\leq 2 \cdot \frac{\E_x \big(B * C\big)(x) (1 + \eps)^{k'} }{(1+2\eps)^{k'}} \\
    &= 2 \cdot \left( \frac{1 + \eps}{1 + 2 \eps} \right)^{k'}
\end{align*}
and
\begin{align*}
|B'||C'| &\geq  \frac{1}{2 |G|^{k'}} \sum_{x} R_{B,C}(x) R_A^-(x)^{k'} \\
&= \frac{1}{2 |G|^{k'}} \cdot |G| \cdot \frac{|B||C|}{|G|} \cdot \frac{|A|^{2k'}}{|G|^{k'}}
\cdot \ip{B * C}{(A \star A)^{k'}} \\
&\geq  \frac{1}{2 |G|^{k'}} \cdot |G| \cdot \frac{|B||C|}{|G|} \cdot \frac{|A|^{2k'}}{|G|^{k'}} \\
&= \tfrac{1}{2} \cdot 2^{-dk'} |B||C| , 
\end{align*}
where we have used that 
$$\ip{B * C}{(A \star A)^{k'}} = \|A \star A\|_{k', B * C}^{k'} \geq \|A \star A\|_{k, B * C}^k
\geq 1
$$
by assumption.
\end{proof}

\section{Spectral positivity and $k$-norm inequalities for convolutions} \label{spectral positivity}

\subsection{A decoupling inequality for convolutions}

\begin{prop}[Fourier interpretation of $k$-norms]
For any function $f : G \rightarrow \R$ and any even $k \in \N$,
$$ \|f\|_{k}^k = \E_{x \in G} f(x)^k = \sum_{\substack{
(\alpha_1, \alpha_2, \ldots, \alpha_k) \in G^k \\
\alpha_1 + \alpha_2 + \cdots + \alpha_k = 0
}}
\prod_{j=1}^k \widehat{f}(\alpha_j) .
$$
\end{prop}

\begin{proof}
Expand the expression $\left(\sum_{\alpha} \widehat{f}(\alpha) e_\alpha(-x)\right)^k$ and take the expectation.
\end{proof}

\begin{lem}[Decoupling inequality] \label{decoupling-inequality}
Let $A,B$ be density functions on a finite abelian group $G$. 
For even integers $k \in \N$ we have
$$ \|A * B - 1\|_k \leq \|A \star A - 1\|_k^{1/2} \|B \star B - 1\|_k^{1/2} .$$
In particular,
$$ \|A * A - 1\|_k \leq \|A \star A - 1\|_k . $$
\end{lem}

\begin{proof}
We note that for any density $A$, $A * 1 = 1 * A \equiv 1$. So, the claimed inequality is the same as
$$ \norm{(A - 1) * (B - 1)}_k \leq \|(A - 1) \star (A - 1)\|_k^{1/2} \|(B - 1) \star (B - 1)\|_k^{1/2}.
$$
We prove more generally that for any functions $f,g$,
$$ \norm{f * g}_k \leq \|f \star f\|_k^{1/2} \|g \star g\|_k^{1/2}.
$$
By re-scaling, it suffices\footnote{
One can check the degenerate case  $\|f \star f\|_{k} = 0$ separately.
This case occurs only when $f \equiv 0$, which can be seen by the calculation $\|f \star f\|_{k}^2 \geq \|f \star f\|_{2}^2 = \sum_{\alpha} |\hat{f}(\alpha)|^4.$
} 
to prove this for $\|f \star f\|_k = \|g \star g\|_k = 1$.
We have
\begin{align*}
    \E_{x} \left( f * g \right) (x)^k &=  \E_x\left( \sum_{\alpha} \widehat{f}(\alpha) \widehat{g}(\alpha) e_\alpha(-x) \right)^k\\
    &\leq \sum_{\substack{
\alpha_1, \alpha_2, \ldots, \alpha_k \\
\alpha_1 + \alpha_2 + \cdots + \alpha_k = 0
}}  \left| \prod_{i=1}^k \widehat{f}(\alpha_i) \cdot \prod_{i=1}^k   \widehat{g}(\alpha_i) \right| \\
&\leq \sum_{\alpha}  \left(  \frac{1}{2} \prod_{i=1}^k |\widehat{f}(\alpha_i)|^2
+  \frac{1}{2} \prod_{i=1}^k | \widehat{g}(\alpha_i) |^2 \right) \\
&= \frac{1}{2} \norm{f \star f}_k^k  + \frac{1}{2} \norm{g \star g}_k^k \\
&= 1. \qedhere
\end{align*}

\end{proof}

We remark that this inequality has appeared before in other works. For example, it is a special case of Lemma 13 in \cite{shkredov17}.

\subsection{Spectral Positivity}

\begin{definition}
Suppose $f : G \rightarrow \R$ is a function such that
$ \widehat{f}(\alpha) $ is real and nonnegative for all $\alpha \in G$. 
We say that such functions are ``spectrally positive", and we use the notation
$$ f \succeq 0 $$
to denote the fact that $f$ is such a function.
\end{definition}
We note that any self-convolution $f \star f$ is spectrally positive:
$$(f \star f)(x) =  \sum_{\alpha \in G} |\widehat{f}(\alpha)|^2 e_{\alpha}(-x) \succeq 0,$$
and conversely that any spectrally positive function $f(x) = \sum_{\alpha} \widehat{f}(\alpha) e_{\alpha}(-x)$ can be expressed as a self-convolution $g \star g$, where 
$$g(x) := \sum_{\alpha}\sqrt{ \widehat{f}(\alpha)} e_{\alpha}(-x) .$$

\begin{prop}
The set $\set{f}{f \succeq 0}$ enjoys the following two closure properties.
\begin{itemize}
    \item (Closure under multiplication.) Suppose $f,g \succeq 0$. Then 
    $$ h(x) := f(x) \cdot g(x) \succeq 0.$$
    \item (Closure under centering.) Suppose $f \succeq 0$, and that $D$ is a symmetric\footnote{By this we mean that $D(-x) \equiv D(x)$.} density function. Then
    $$ h(x) := f(x) - (D * f)(x) \succeq 0.$$
    In particular, by letting $D \equiv 1$,
    $$ f - \E[f] \succeq 0.$$
 \end{itemize}
\end{prop}

\begin{proof}
For the first claim, simply express $f$ and $g$ by their Fourier expansions and expand the product:
$$ \left( \sum_{\beta} \widehat{f}(\beta) e_\beta(-x) \right) \left( \sum_{\beta'} \widehat{g}(\beta') e_{\beta'}(-x) \right) = \sum_{\alpha} \left( \sum_{ \beta + \beta' = \alpha} \widehat{f}(\beta) \widehat{g}(\beta') \right) e_\alpha(-x) \succeq 0.$$
For the second claim we first note that since $D$ is a density, for any $\alpha$ we have
$$ |\widehat{D}(\alpha)| = \left| \E_{x \sim D} e_{\alpha}(x) \right| \leq \E_{x \sim D} \left| e_{\alpha}(x)) \right| = 1; $$
in fact, this applies more generally to any function $D$ with $\|D\|_{1} \leq 1$.
Then we compute
\begin{equation*}
f(x) - (D * f)(x) = \sum_{\alpha} \widehat{f}(\alpha) (1 - \widehat{D}(\alpha)) e_\alpha(-x) \succeq 0,
\end{equation*}
where we have used that $D(-x) = D(x)$ and so $D$ has Fourier coefficients $\widehat{D}(\alpha) \in \R.$
\end{proof}

\begin{cor} [Positive correlation for spectrally positive functions] \label{positive-correlation}
Suppose 
$$f_1,f_2,\ldots,f_k \succeq 0$$ are some spectrally positive functions on $G$. Then
$$ \E [ f_1 f_2 \cdots f_k ] \geq \E [ f_1 ]  \E [ f_2 ] \cdots \E [f_k]. $$
\end{cor}
\begin{proof}
In light of the closure property for multiplication, it suffices to verify this for just two functions. Let $f,g \succeq 0$. 
Write $g \equiv (g  - \E[g]) + \E[g]$, and express
$$ \ip{f}{g} = \ip{f}{g  - \E[g]} + \E[f] \cdot \E[g]. $$
We note that the function $F(x) := f(x) \cdot (g(x) - \E[g])$ is spectrally positive. 
In particular,
\begin{equation*}
 \ip{f}{g  - \E[g]} = \E[F] = \hat{F}(0) \geq 0  . \qedhere
\end{equation*}
\end{proof}
 
\begin{cor}[Odd central moments] \label{odd-central-moments}
Consider a spectrally positive density $D$ on $G$. The odd central moments of $D$ are non-negative. That is,
$$ \E_{x \in G} (D(x) - 1)^k \geq 0$$
for all odd integers $k \in \N$ (and hence for all $k \in \N$).
\end{cor}
\begin{proof}
$F(x) := D(x) - 1 \succeq 0$.
Also, $F(x)^k = F(x)  F(x) \cdots  F(x) \succeq 0;$
in particular $\E[F] \geq 0$.
Alternatively, consider this as a special case of \cref{positive-correlation}.
\end{proof}

The utility of this is due to the following. 

\begin{prop} \label{odd-moments}
For a real-valued random variable $X$ and $k \geq 1$, we use the notation
$$ \|X\|_{k} := \left( \E |X|^k \right)^{1/k}. $$
Suppose $X$ is such that 
\begin{itemize}
    \item $\E (X - 1)^{k} \geq 0$ for all odd $k \in \N$, and
    \item $\|X - 1\|_{k_0} \geq \eps$ for some even $k_0 \geq 2$ and some $\eps \in [0,\frac{1}{2}]$.
\end{itemize}
Then, for any integer $k' \geq 2 k_0/\eps$,
$$ \|  X  \|_{k'} \geq 4^{-1/k'} (1 + \eps) \geq 1 + \frac{\eps}{2}.$$
\end{prop}

We note that one can infer a statement of this sort already from the following observation: for any odd $k \geq k_0$ we have
$$ \E(X - 1)^k = \E |(X - 1)_{+}|^k - \E |(X - 1)_{-}|^k \geq 0$$
and
$$ \E |X - 1|^{k} = \E |(X - 1)_{+}|^k + \E |(X - 1)_{-}|^k \geq \eps^k,  $$
so it follows that
$$ \E |(X - 1)_{+}|^k \geq \frac{\eps^k}{2}. $$
From here one can get a reasonable lower bound on some $\|X\|_{k'}$ e.g.\ by a basic pruning argument analogous to the proof of the Paley–Zygmund inequality.
However, to obtain nicer constants we give a somewhat different proof which can be found in the appendix (see \cref{sec:odd-moments}). 

\subsection{Local variants}

We proceed to give some local variants of the statements above which will be needed for setting $A \subseteq [N] \subseteq \Z$. 
We recall the notation from \cref{weighted-k-norm},
$$ \| f \|_{k, D} := \ip{D}{|f|^k}^{1/k} =  \left( \E_{x \sim D} |f(x)|^k \right)^{1/k}.$$

\begin{lem}[Local decoupling inequality] \label{local-decoupling}
Let $G$ be a finite abelian group, and suppose $D$ is a spectrally-positive density function on $G$. Let $(T_{\theta} D)(x) := D(x - \theta)$ be some translate of $D$. For arbitrary real-valued functions $f,g$ on $G$, and for all even $k \in \N$, we have
$$ \|f * g\|_{k, T_{\theta} D}^2 \leq  \|f \star f\|_{k, D} \|g \star g\|_{k, D} .$$
In particular, if $A$ is a density function then
$$ \|A * A - 1\|_{k, T_{\theta} D} \leq \|A \star A - 1\|_{k, D} .$$
\end{lem}
\begin{proof}
We have
\begin{align*}
     \ip{T_\theta D}{(f * g)^k} &= \sum_{\gamma} \widehat{D}(\gamma) e_{\gamma}(\theta) \sum_{\beta_1 + \beta_2 + \cdots + \beta_k = \gamma} \prod_{i=1}^k \widehat{f}(\beta_i) \widehat{g}(\beta_i) \\
     &\leq \sum_{\gamma}  \sum_{\beta_1 + \beta_2 + \cdots + \beta_k = \gamma}  \widehat{D}(\gamma)\prod_{i=1}^k |\widehat{f}(\beta_i)| \
    |\widehat{g}(\beta_i)|\\
    &\leq \sum_{\gamma}  \sum_{\beta_1 + \beta_2 + \cdots + \beta_k = \gamma}  \widehat{D}(\gamma)  \left(
    \frac{\eta}{2}\prod_{i=1}^k \left|\widehat{f}(\beta_i)\right|^2 + \frac{1}{2\eta}
    \prod_{i=1}^k\left|\widehat{g}(\beta_i)\right|^2 \right) \\
    &= \frac{\eta}{2} \ip{D}{(f \star f)^k} + \frac{1}{2 \eta} \ip{D}{(g \star g)^k}
\end{align*}
for any choice of $\eta > 0$. Optimizing over $\eta$, we conclude that
\begin{equation*} \ip{T_\theta D}{(f * g)^k} \leq \sqrt{\ip{D}{(f \star f)^k} } \sqrt{\ip{D}{(g \star g)^k}}. \qedhere
\end{equation*}

\end{proof}

\begin{lem}[Lower bounds from upper bounds -- local variant] \label{local-upper-to-lower}
If $$\|A \star A - 1\|_{k, D} \geq \eps$$
for some $\eps \in [0,\frac{1}{2}]$ and spectrally-positive density $D$,
then
$$ \|A \star A\|_{k', D} \geq 1 + \frac{\eps}{2} $$
for some $k' \leq O(k/\eps)$.
\end{lem}
\begin{proof}
Let $k$ be an integer. We note that
$$ D \cdot (A \star A - 1)^{k} =  D \cdot ( (A - 1) \star (A - 1) )^k $$
is a spectrally-positive function on $G$. In particular,
$$ \E_{x \sim D} ((A \star A)(x) - 1)^k = \E_{x \in G} D(x) ((A \star A)(x) - 1)^k \geq 0.
$$
So, we may apply \cref{odd-moments} to the random variable $X := (A \star A)(x)$ where $x \sim D$. \qedhere
\end{proof}
\section{Finding a 3-progression in $A \subseteq [N]$ -- an overview} \label{integers-overview}

We describe how to modify our approach above to address the 3-progression problem in the setting $A \subseteq [N] \subseteq \Z$.
We recall our result regarding this problem.
\begin{thm}[\cref{many-3-progs}, restated]
Suppose $A \subseteq [N]$ has density $\mu \geq 2^{-d}$. Then the number of triples $(x,y,z) \in A^3$ with $x + y = 2z$ is at least
$$ 2^{-O(d^{12})} \cdot N^2 .$$
\end{thm}
We review the well-known approximate equivalence of the 3-progression problem in the setting $A \subseteq [N]$ with the same problem in the setting $A \subseteq \Z_N$. One may embed $A \subseteq [N]$ naturally inside $\Z_{N'}$, say for some $N' \approx 3N$, so that (i) the density of $A$ in its container decreases only slightly and (ii) we do not obtain any new solutions to
$$ x + y = 2z \textnormal{ mod } N'$$
which were not present already. The natural reduction in the other direction is even easier: in this direction, we suffer no loss in density. 

So, it suffices to consider the problem of finding many 3-progressions in the setting $A \subseteq \Z_N$, $|A| = 2^{-d} N$. Ultimately we will not take quite this approach, but we consider this setting first as it is more directly comparable with the setting $A \subseteq \F_q^n$. Let us discuss the compact summary of the approach for $\F_q^n$ given at the end of \cref{review}. All of our steps are, in fact, more generally applicable to any finite abelian group $ G $, up until the last point, where we use Sanders' invariance lemma to find an increment onto a large subspace. This is still not a problem: applying a more general form of Sanders' lemma can readily prove a statement of the following sort. Suppose $A \subseteq \Z_N$ has noticeably fewer than the expected number of 3-progressions. In that case, we can obtain a density increment onto a choice of either (i) some (translate of a) large Bohr set $B$ of rank at most $O(d^8)$, or (ii) some large generalized arithmetic progression $P$ of rank $O(d^8)$. The problem comes when we try to iterate this. Unlike the situation before, the container-sets $B$ and $P$ are not structurally isomorphic again to some finite abelian group, so our techniques no longer apply.\footnote{We could still apply them in a literal sense -- everything is still a subset of some finite group $G = \Z_N$ -- but quantitatively our techniques become trivial because we are constantly comparing to the uniform density function $1$ on $G$.}
This is not to say that such an approach (i.e., \ iteratively seeking density increments onto large approximate subgroups) does not work, only that we cannot analyze it; indeed, it seems likely that it should work.
To be concrete, we ask the following technical question.
\begin{question}
Suppose $A \subseteq B$, where $B \subseteq \Z_N$ is a Bohr set of rank $r$, and we have local density
$$ \mu = \frac{|A|}{|B|} \geq 2^{-d}. $$
Can it be shown that we must have either (i) at least
$$ 2^{-\textnormal{poly}(r)} \cdot \mu \cdot |A|^2 $$
solutions to $x + y = 2z$ with $x,y,z \in A$, or else (ii)
a density increment
$$ \frac{|A \cap B'|}{|B'|} \geq (1 + \Omega(1)) \cdot \mu $$
of $A$ onto some affine Bohr set $B'$ with
\begin{itemize}
\item rank $r' \leq r + \textnormal{poly}(d)$ and
\item size $|B'| \geq 2^{-\textnormal{poly}(r,d)} |B|$ ?
\end{itemize}
\end{question}
Answering this question seems maybe not too far out of reach, but in this work, we were unsuccessful in developing sufficiently strong local variants of our techniques that could address it. 
A positive answer to this question would surely give the ``right" version of our proof.\footnote{
Moreover, an approach that successfully answers this question would likely also produce a pleasant analog to our structural result (\cref{robust-sunflower}) in the finite field setting. It is not immediately clear what such an analog should look like, specifically.
}
Here instead, we will apply some tricks and cut some corners.

We proceed to discuss the overall structure of the arguments in the earlier works of Roth and Szemer{\'e}di and Heath-Brown for addressing the 3-progression problem in the integers. More specifically, we follow the interpretation of these works given in \cite{gowers01} and \cite{green-note}.
In what follows, we refer to the arguments of both works together simply as the ``early approach" (as opposed to the ``modern approach" developed by Bourgain and refined by Sanders, which emphasizes the use of Bohr sets \cite{bourgain99, bourgain08, sanders12certain}).

In the early approach, one considers an ad hoc passage back and forth from the setting $A \subseteq \Z$ into various cyclic groups to facilitate the density increment argument. 
We outline the main points.
\begin{itemize}
\item We begin with a set $A \subseteq [N]$ of size $|A| = \mu N$ which we are dissatisfied with: it has much fewer than $\mu |A|^2$ solutions to $x+y = 2z$ with $x,y,z \in A$.
\item We consider a sort of ``temporary" embedding of $A$ into some cyclic group $\Z_{N'}$ where we can do Fourier analysis. This embedding must be notably more efficient than the simple one described above. Ultimately, we will obtain only a small density increment $(1+\eps)$ at every iteration, so we cannot afford to lose a factor $3$ repeatedly. Additionally, we (roughly) need the embedding to be such that $A$ still has substantially fewer solutions to $x + y = 2z \textnormal{ mod } N'$ than expected, which is in tension with the previous constraint.
\item Given such an embedding, we argue via Fourier analysis that we may obtain a density increment onto some large arithmetic progression $P \subseteq \Z_{N'}$.
\item We must then argue that $P$ can be pulled back to some ``genuine" progression $P' \subseteq [N] \subseteq \Z$, so that we obtain a density increment of $A$ onto $P'$. This limits the kind of progressions we can allow in the previous step. Finally, we observe that the progression $P'$ is structurally isomorphic to the interval $[|P'|]$, which allows us to iterate the argument.
\end{itemize}

Let us offer a comment which we find clarifying.
There is an alternative interpretation of this approach where the actual algorithm is quite simple and natural. The step where we briefly consider some ``virtual" embedding into a cyclic group can be relegated to the analysis. 
Let us elaborate. For a set $A \subseteq [N]$ of size $|A| = \mu N$, say that a subset $A' \subseteq A$ is satisfactory if there are roughly at least $\mu |A'|^2$ solutions to $x+y=2z$ with $x,y,z \in A'$. We can make a conceptual distinction between the procedure used to find such a subset and the argument proving that the procedure is successful. In the context of the above, we note that if we are dissatisfied with our current subset, the only recourse which is ultimately available to us is to pass to some restriction $A \cap P$ onto a large progression $P \subseteq \Z$ where we obtain a density increment.

Thus, one can give a quite clear description of the algorithm implicit in the early approach, which need not mention cyclic groups: While there is a density increment onto some large progression $P \subseteq \Z$ available, take it. Throughout, hold on to the current container-set $P$, starting with $P = [N]$, and measure density relative to it. Once there are no more increments available: conclude -- we have found our candidate subset $A' \subset A$. From here, it can be considered a separate matter to argue that $A'$ is satisfactory: if it is not, we argue (e.g., \ by considering a ``virtual" embedding into some cyclic group, if we like) that we could obtain one further increment -- a contradiction.

A natural extension of this simple algorithm is to broaden the class of allowed container-sets $P$ to include low-rank generalized progressions; this is quite analogous to the allowance of subspaces of codimension larger than one in the setting of $\F_q^n$. This comes morally very close to the algorithm we will use, and we will also analyze it in much the same way as above: by making some ad hoc reductions to the setting of finite groups where our techniques apply.
We give a more detailed overview of our analysis in \cref{integers-analysis-overview}.
For now, we focus on fully specifying our algorithm.

We will need to consider an algorithm stronger than the one described above to make up for a specific weakness of our analysis. Namely, this weakness is related to the last step in the early approach described above, where we would like to pull back a progression in a cyclic group to a ``genuine" progression in the integers. This step is already costly for arithmetic progressions; in the early approach, one obtains a progression $P \subseteq [N]$ no larger than $\sqrt{N}$. Indeed, the situation is only worse for generalized progressions, and it will be imperative that we somehow avoid such a substantial size loss. 

We will be naturally led to consider the (only mildly) more general problem of establishing the existence of many 3-progressions in a set $A \subseteq P \subseteq \Z^r$, where $P$ is some set of the form $P = \prod_{i=1}^r [N_i] \subseteq \Z^r$, and again we have density 
$$ \mu = \frac{|A|}{|P|} \geq 2^{-d}, $$
and furthermore $r$ is assumed to be no larger than $d^{O(1)}$. 
We call the set $A$ and its container-set $P$ together a ``configuration". 
We consider some more liberal implementations of the density increment framework which would still suffice for lower bounding the number of 3-progressions in $A$.

In the broadest sense, whenever we are dissatisfied with $A \subseteq P$, we would like to establish the existence of some related configuration $A' \subseteq P'$, possibly lying in a slightly larger-dimensional space $\Z^{r'}$, where
\begin{itemize}
\item $A'$ has no more solutions to $x + y = 2z$ than does $A$,
\item the container set $P'$ is still highly structured -- in our setting: of the form $\prod_{i=1}^{r'}[N_i']$,
\item the density of $A'$ in its container has increased -- in our setting by some constant factor $(1 + \Omega(1))$, and
\item The size of $A'$ has not decreased too substantially. 
\end{itemize}

To this end, let us say that $A \subseteq P$ is ``spread" roughly when it has no possible ``density increments" satisfying the criteria above (we will settle on a precise formulation shortly). Then, our procedure for ``locating" a satisfactory configuration $A' \subseteq P'$ is as follows: while there are any good density increments available, take one of them. Otherwise, our resulting configuration is spread, and we are left to argue that this forces it to be satisfactory.

\begin{definition}[Good increments and spread configurations] \label{spreadness}
Fix some constants $c,K \in \N$.
Suppose that $A \subseteq P \subseteq \Z^r$, where
$P$ is of the form $[N_1] \times [N_2] \times \cdots \times [N_r]$, and $A$ has density
$$ \frac{|A|}{|P|} \geq 2^{-d} $$
in its container, for some $d \geq 1$. 

Let $A'$ be a subset of $A$, and let $\phi$ be a labelling of elements $a \in A'$ by points $p \in P' = \prod_{i=1}^{r'} [N_i'] \subseteq \Z^{r'}$; that is, $\phi$ is some injection $\phi : A' \rightarrow P'$. 

For $\eps > 0$, we say that $(A', \phi)$ is a $(1+\eps)$-good increment (or just a ``good increment", suppressing the dependence on $\eps$ in addition to the dependence on $c$ and $K$) if the following conditions are satisfied.

\begin{enumerate}
\item The labelling $\phi : A' \rightarrow \Z^{r'}$ is a Freiman homomorphism of order $2$.
\item We have the density increment 
$$ \frac{|A'|}{|P'|} \geq \left(1 + \eps \right) \frac{|A|}{|P|} .$$
\item We have bounded dimension growth 
$$ r' \leq r + K d^c. $$
\item We have bounded size loss 
$$\lg |A'| \geq \lg |A| - K d^c - K r^c. $$
\end{enumerate}

If the configuration $A \subseteq P$ has no $(1+\eps)$-good increments, we say that $A$ is $(1+\eps)$-spread (relative to $P$).
\end{definition}

We see that this quite liberal type of density increment (that is, one requiring merely a $2$-homomorphism) is indeed still useful for investigation of 3-progressions: if $x,y,z \in A'$ satisfy
$$ \phi(x) + \phi(y) = 2\cdot \phi(z) = \phi(z) + \phi(z) \textnormal{ (in } \Z^{r'}\textnormal{)},$$ then 
$$ x + y = z + z = 2 \cdot z \textnormal{ (in } \Z^r\textnormal{)}.$$ 
Thus, the number of 3-progressions does not increase when we pass from $A'$ to $\phi(A')$.  

\begin{prop}[Passing to a spread configuration] \label{pass-to-spread}
Fix a choice of constants $(c,K,\eps)$ quantifying spreadness. 

Suppose $A \subseteq [N]$ has density at least $2^{-d}$. 
Then there is a subset $A' \subseteq A$ and a Freiman $2$-homomorphism 
$$ \phi : A' \rightarrow P' = [N_1] \times [N_2] \times  \cdots \times [N_r] \subseteq \Z^r$$
such that
\begin{itemize}
    \item $\phi(A')$ is spread relative to its container $P'$,
    \item $r \leq O(d^{c+1})$, and
    \item $\lg |A'| \geq \lg N - O\left(d^{c^2 + c + 1}\right)$.
\end{itemize}

\end{prop}
\begin{proof}
Let $A_0 := A$ and $P_0 := [N]$. We consider the ``greedy algorithm" which at each step passes from the current configuration to some good increment, if one exists. In this way we produce a sequence of
\begin{itemize}
    \item configurations $A_i \subseteq P_i \subseteq \Z^{r_i}$,
    \item subsets $A_{i-1}' \subseteq A_{i-1}$, and
    \item bijections $\phi_{i} : A_{i-1}' \rightarrow A_i$
\end{itemize}
such that in addition, each of the maps $\phi_{i}$ is a $2$-homomorphism. 
The density of the $i$-th configuration is at least
$$ \frac{|A_i|}{|P_i|} \geq \left(1 + \eps\right)^{i} 2^{-d} \geq 2^{\eps i  - d},
$$ 
and so this algorithm must terminate within $n \leq d/\eps$ iterations. We are left with a configuration $A_n \subseteq P_n$ which is guaranteed to be spread. 
Taking the composition $\phi := \phi_{n} \circ \phi_{n-1} \circ \cdots \circ \phi_{1}$ gives a map which is a bijection between $A_n$ and its preimage $\phi^{-1}(A_n) \subseteq A$, and it is also a $2$-homomorphism from the preimage $\phi^{-1}(A_n)$ to $\Z^{r_n}$. 

Since the density never decreases, we can bound $r_n$ simply by
\begin{equation*}
    r_n \leq 1 + n \cdot K (d^{c}+1) \leq O(d^{c+1}). 
\end{equation*}
Since $r_i \leq r_n$ for all $i$, we can lower-bound the size of $A_n$ simply by
\begin{equation*}
    \lg |A_n| \geq |A| - n \cdot O(d^c + r^c) \geq \lg |A| - O(d^{c^2 + c + 1}) = \lg N - O(d^{c^2+c+1}). \qedhere
\end{equation*}
\end{proof}

With the density-increment framework in place, the task of proving \cref{3-progs} is reduced to the task of proving the following.

\begin{lem}[Spread configurations have many 3-progressions] \label{progressions-in-spread-sets}
For some fixed choice of constants $(c,K,\eps)$ quantifying spreadness, the following holds.

Suppose that $A \subseteq [N_1] \times [N_2] \times \cdots \times [N_r] \subseteq \Z^r$ has density $\mu$ relative to its container, and that $A$ is spread relative to its container.
Then there are at least
$$ r^{-O(r)} \cdot \mu \cdot |A|^2 $$
solutions to $x + y = 2z$ with $(x,y,z) \in A^3$. 
\end{lem}

We in fact prove the following more specific formulation.

\begin{lem} \label{progressions-in-spread-sets-specific}
Suppose $A \subseteq [N_1] \times [N_2] \times \cdots \times [N_r] \subseteq \Z^r$ has density $\mu \geq 2^{-d}$. 

Either the number of triples $(x,y,z) \in A^3$ with $x+y=2z$ is at least
$$ r^{-O(r)} \cdot \mu \cdot |A|^2, $$
or there is a $(1+2^{-10})$-good increment $(A',\phi)$ mapping into $\Z^{r'}$, specifically with
\begin{itemize}
    \item $r' \leq r + O(d^8)$ and
    \item $|A'| \geq r^{-O(r)} \cdot 2^{-O(d^2 r)} \cdot 2^{-O(d^{10})} \cdot |A|$.
\end{itemize}
\end{lem}

Given these specific parameters, we obtain \cref{many-3-progs} by using essentially the density-increment framework described above but with some optimizations made to certain details. 

\begin{proof}[Proof of \cref{many-3-progs}]
We begin with a set $A \subseteq [N]$ with density at least $2^{-d}$. We (repeatedly) take any available $(1+2^{-10})$-increment specified by a subset which is smaller than our current set only by a factor $r^{-O(r)} \cdot 2^{-O(d^{10})}$ and a 2-homomorphism $\phi$ into $\Z^{r'}$ for some dimension $r'$ exceeding the current dimension by only $O(d^8)$. After some number (say $n$) of such increments, we arrive at some configuration $A_n \subseteq P_n \subseteq \Z^{r_n}$ for which no further increment is possible, and indeed, we must have $n \leq O(d)$. Thus, also we have
\begin{itemize}
    \item $r_n \leq O(d^9)$ and, noting the asymptotic bound $r_n \lg(r_n) \leq O(d^{10}) $,
    \item $\lg |A_n| \geq \lg |A| - n \cdot O(r_n \lg  r_n) - n \cdot O(d^2 r_n) - n \cdot O(d^{10}) \geq \lg N - O(d^{12})$. 
\end{itemize}
We apply \cref{progressions-in-spread-sets-specific} to the final configuration $A_n \subseteq P_n$. Since we are in the ``spread" case, we must find at least
$$ r_n^{-O(r_n)} \cdot 2^{-d} \cdot |A_n|^2 \geq 2^{-O(d^{12})} \cdot N^2 
$$
3-progressions. \qedhere
\end{proof}

\subsection{Proof overview for \cref{progressions-in-spread-sets} and \cref{progressions-in-spread-sets-specific}} \label{integers-analysis-overview}

For this overview, let us focus on the one-dimensional case $A \subseteq [N]$, which is sufficient already to illustrate many of the critical points. Our proof of \cref{progressions-in-spread-sets} is by contradiction. We assume that the set $A$ has few 3-progressions (i.e., \ much fewer than the ``expected" number, roughly $|A|^3/N$), and we then show that $A$ cannot be spread by exhibiting a density increment satisfying the four criteria in \cref{spreadness}: We obtain a new configuration via a (i) Freiman 2-homomorphism with (ii) increased density, (iii) bounded dimension growth, and (iv) bounded size loss. 

Our strategy is to make ad hoc reductions to the setting where $ A $ is instead a subset of a finite group $ G $, where our techniques apply. For example, it follows readily from our prior arguments that if $A \subseteq \Z_N$ has density $\
\mu \geq 2^{-d}$, and the number of 3-progressions in $A$ deviates substantially from the expected number, then we get a density increment onto some generalized arithmetic progression $P \subseteq \Z_N$.

\begin{lem}[Structure vs.\ Pseudorandomness in $\Z_N$ -- special case of \cref{svr-local}] \label{svr}
Consider $A \subseteq \Z_N$ of size $|A| \geq 2^{-d} N$.

Suppose that 
$$\|A * A - 1\|_{k} \geq \Omega(1).$$
Then there exists a (proper) generalized arithmetic progression $P \subseteq G$, with
\begin{itemize}
    \item rank at most $r \leq O(k^4 d^4) $ and
    \item size $|P| \geq 2^{-O(k^5 d^5)} N$
\end{itemize}
such that
$$ \ip{P}{A} \geq 1 + \Omega(1). $$
\end{lem}

That is, we get a set $P \subseteq \Z_N$ with density-increment
$$ \frac{|A \cap P|}{|P|} \geq (1 + \Omega(1)) \frac{|A|}{N} $$
of the form
$$ P = \Big\{ a + \sum_{i=1}^r c_i \cdot x_i \; : \; x_i \in [N_i]\Big\} $$
with size
$|P| = N_1 N_2 \cdots N_r \geq 2^{-\textnormal{poly}(d,k)} N$.

So, to describe our tentative plan in detail: we plan to embed $A \subseteq [N]$ into $\Z_N$ via the obvious embedding $\phi : x \mapsto x \mod N$, and then to find a density increment onto some subset $\phi(A) \cap P$. We let $A' = \phi^{-1}( \phi(A) \cap P) = A \cap \phi^{-1}(P) \subseteq [N]$ be the preimage.
We then compose with the simple Freiman homomorphism $\phi'$ which takes points $p \in P$ to their ``label" $(x_1, x_2, \ldots, x_r) \in [N_1] \times [N_2] \times \cdots \times [N_r] \subseteq \Z^r$ (see \cref{progression_label}). Then, we hope that the pair $(A', \phi' \circ \phi)$ gives us our desired density increment.

There are two distinct issues to address. Firstly, we need to ensure that upon embedding into $\Z_N$, the density $\phi(A) * \phi(A)$ does indeed deviate substantially from $1$ at sufficiently many points that we can apply \cref{svr}. Secondly, we must ensure that the embedding $\phi : A' \rightarrow \Z_{N}$ is a $2$-homomorphism. 

Although both issues must be addressed, the first should be considered less serious, as explained next. By our assumption that $A$ has few 3-progressions, we have that the number of representations $R_A(2z)$ of $2z$ is much smaller than $|A|^2/N$ for most points $z \in A$. For simplicity, let us assume the $R_A(2z)$ is very small for \textit{all} points $z \in A$ (and, indeed, it is not hard to reduce to this case).
For $x \in [N]$, we have 
$$( \phi(A) * \phi(A))(\phi(x)) = \frac{N}{|A|^2} R_{\phi(A)}(\phi(x)) = \frac{N}{|A|^2}R_{A}(x) + \frac{N}{|A|^2}R_A(x + N).$$
We need to avoid the ``unlikely" case that identifying each $x$ with $x+N$ results in close approximation $R_{\phi(A)}(\phi(x)) \approx \frac{|A|^2}{N}$ for all but a very tiny number of points $x$.
One has considerable flexibility in applying various ad hoc tricks to avoid this case. For example, one can choose to instead embed into $\Z_{N'}$ for any choice of $N' \in [N, (1 + \delta) N]$ for some small parameter $\delta$ -- for the sake of discussion, say $\delta \leq 1/100$. This gives up a small amount of density in the short term, but it is acceptable to do so to satisfy the hypothesis of \cref{svr} since, in the end, we can still obtain an overall increase in density. Another trick one can consider is first passing to some restriction $A \cap I$ for any reasonably large interval $I \subseteq [N]$ with density, say $|A \cap I|/|I| \geq (1 - 1/100) \mu$. 

We briefly sketch some details of the specific ad hoc reduction used here. 
We assume, at only a negligible-factor loss in the density, that $A$ is in fact entirely contained within the interval $U := [2 \delta N, N] \subseteq [N]$. Additionally we assume we are in the ``nice" case where a reasonably large fraction of $A$ lies in the interval $M = (\frac{N}{2}, \frac{N}{2} + \delta N)$: say $|A \cap M| \geq \Omega(\delta) \cdot |A|$. If we are not in the nice case, we note that by using ideas in \cref{invar} we can pass to a restriction of $A$ to some fairly large interval such that the restriction becomes ``nice" (relative to that interval).\footnote{
In the one-dimensional case, it would follow already from the fact that $A$ is spread relative to $U$ that $A$ must be ``nice". However, the connection between density upper-bounds on structured sets to density lower-bounds on structured sets degrades when we pass to the general case $A \subseteq [N_1] \times [N_2] \times \cdots \times [N_r]$. In contrast, the idea of using translation-invariance remains quantitatively efficient.
}

The point of these two intervals (the ``upper-portion" $U$ and the ``middle slice" $M$ of $[N]$) is that they are designed specifically so that if $x,y \in U$ and $z \in M$ then
$$ x+y = 2z \mod N $$
only if
$$ x+y = 2z.$$
In particular, we obtain $R_{\phi(A)}(\phi(2z)) = R_{A}(2z)$ for all $z \in M$, so the convolution $\phi(A) * \phi(A)$ is indeed much smaller than $1$ on all the points $\phi(2z) \in \Z_N$ with $z \in A \cap M$. Since there are at least $\Omega(2^{-d} N)$ such points, we can conclude that, say,
$$ \|\phi(A) * \phi(A) - 1\|_{k} \geq \frac{1}{2} $$ 
for some $k \leq O(d)$, as desired.
Looking ahead, we note that in the general case of $A \subseteq [N_1] \times [N_2] \times \cdots \times [N_r]$, we intend to essentially apply this same trick independently in each of the $r$ coordinates, but with parameter $\delta \approx 1/100r$. 

Now we discuss the second key issue: can we ensure that the embedding $\phi : A' \rightarrow \Z_N$ is in fact a $2$-homomorphism? A notable feature of our approach is that we do not try to ensure that the embedding $\phi : A \rightarrow \Z_N$ is a $2$-homomorphism with respect to our original set $A$ -- this seems difficult to accomplish without conceding an unacceptable amount of density. Instead, we intend to exploit the fact that we only care about the behavior of $\phi$ on its restriction to $A' \subseteq A.$ 
Our starting point is the following formalization of a commonly-used trick (see \cref{interval_embed}), used to embed problems in $\Z$ into $\Z_N$. Suppose $I \subseteq \Z$ is an interval of size $|I| \leq N/t$. Then the natural embedding $\phi : I \rightarrow \Z_N$ is a $t$-homomorphism. 

Usually, this trick is applied before the embedding. Typically, one picks a specific interval $I$, considers the restriction $A \cap I$, and embeds this restriction into $\Z_N$. In contrast, we would like to delay making a specific choice of interval for as long as possible. We proceed as follows. Consider the natural embedding $\phi : [N] \rightarrow \Z_N$. We embed $A$ into $\Z_N$ and invoke \cref{svr} to find a density-increment $\mu' = |\phi(A) \cap P|/|P| \geq (1 + \Omega(1)) \mu$, where $P$ is a generalized progression of rank at most $\textnormal{poly}(d)$. Now consider $A' = \phi^{-1}(\phi(A) \cap P)  = A \cap \phi^{-1}(P)$, and also consider some partition $[N] = I_1 \cup I_2 \cup I_3$ of $[N]$ into three intervals each of size roughly $N/3$. It must be the case that one of the densities $\mu_i' := |A' \cap I_i|/|I_i|$ is at least as large as $\mu'$; let us start to modify the basic plan laid out above and set $A'' := A' \cap I_i$. 

We now have that the embedding $\phi : A'' \rightarrow \Z_N$ is a $2$-homomorphism, as desired, which nearly completes the proof. As stated, a small issue with this plan is that now the container of $\phi(A'')$, $\phi(I_i) \cap P$, is no longer necessarily a generalized progression. Certainly, this container-set, which is the intersection of a progression of rank $1$ and a progression of rank $r$, still has a large amount of additive structure. One way to continue here is to partition $P \cap \phi(I_i)$ into a small number of (still reasonably low-rank) generalized progressions and further restrict onto one of them -- this would suffice to complete the proof (for the one-dimensional case $A \subseteq [N]$). This gives a good idea of how we plan to address the second key issue. In the actual proof, we proceed somewhat differently: 
we mix the idea described here involving the intervals $I_1, I_2, I_3$ into the \textit{proof} of \cref{svr}, which works with a low-rank Bohr set $B$ as an intermediate step before arriving at a low-rank progression $P$ inside some translate of $B$. Compared with generalized progressions, Bohr sets are nicer because the intersection of two Bohr sets is again a Bohr set, which explains how we avoid the small issue encountered above. We note that for the general case of $A \subseteq [N_1] \times [N_2] \times \cdots \times [N_r]$, our corresponding generalization of the trick here involving the intervals $I_1, I_2, I_3$ is developed and formalized in \cref{freiman-hom}. See in particular \cref{t-safe}, which defines the notion of a ``safe" set, and \cref{safe_bohr}, which provides a reasonably large, safe Bohr set for $\Z_{N_1} \times \Z_{N_2} \times \cdots \times \Z_{N_{r}}$.


\section{Preliminaries for 3-progressions in the integers} \label{prelims-integers}

\subsection{Generalized arithmetic progressions}

\begin{definition}[Generalized progression] \label{generalized-progression}
In an abelian group $G$, a generalized arithmetic progression (or just a ``progression") is a set of the form
$$ P = \Big\{ a + \sum_{i=1}^r c_i \cdot x_i \; : \; x_i \in [N_i]\Big\} $$
for some elements $a, c_1, \cdots c_r \in G$. We say that the number $r$ is the ``rank" of $G$.
\end{definition}

In the case that every point $p = a + \sum_{i=1}^r c_i \cdot  x_i \in P$ is represented only once in this form (i.e.\ the case when $|P| = N_1 N_2 \cdots N_r$), we say that $P$ is a ``proper" progression. Since we will be interested only in proper progressions in this work, we often omit the qualifier ``proper".

\subsection{Bohr sets}

\begin{definition}[Bohr set] \label{bohr-set}
Let $G$ be a finite abelian group with character group $\hat{G}$.
A Bohr set of rank $1$ in $G$ is a set of the form
$$ \set{x \in G}{|\gamma(x) - 1| \leq \rho} $$
for some character $\gamma \in \hat{G}$ and some $\rho \geq 0,$ which we call the radius.
A Bohr set of rank $r$ is a set in $G$ describable as the intersection of at most $r$ such sets.

Given a set $\Gamma \subseteq G$, we use the notation 
$$ \textnormal{Bohr}(\Gamma, \rho) $$
to denote the rank $|\Gamma|$ Bohr set
$$ \set{x \in G}{|\gamma(x) - 1| \leq \rho \textnormal{ for all } \gamma \in \Gamma}.$$
\end{definition}

The following is an example of a simple connection between progressions and Bohr sets which we'll make use of at a few points.

\begin{example}[A centered interval is a Bohr set] \label{bohr_interval}
Consider the rank-$1$ Bohr set in $\Z_N$ corresponding to the character $e_1 = (x \mapsto e^{2 \pi i x/N})$:
$$ \textnormal{Bohr}(\{ e_1\}, \rho) = \set{x \in \Z_N}{|e^{2 \pi i \frac{x}{N}} -  1| \leq \rho}. $$
For any given $\rho \in [0,1]$, this set is simply some interval $[-m,m] \subseteq \Z_N$ with 
$$    \frac{\rho}{2 \pi} \leq \frac{m}{N} \leq \frac{\rho}{4}.  $$
\end{example}
\subsection{Properties of Bohr sets}

Given a Bohr set $B = \Bohr(\Gamma, \rho)$, we denote the \textit{dilation} of $B$ by $\delta$ by
$$ B_{\delta} := \Bohr(\Gamma, \delta \rho) .$$
We note the straightforward sumset inclusion
$$ B + B_{\delta} \subseteq B_{1 + \delta} .$$

A convenient fact about Bohr sets and their dilations is that we can easily give some good approximate bounds on their size.

\begin{prop}[Bohr set size estimates {\cite[Section 4.4]{tao-vu}}] \label{bohr-set-size}
Suppose $B = \Bohr(\Gamma, \rho) \subseteq G$ is a Bohr set of rank $|\Gamma| = r$ and with $\rho \in [0,2]$. 
We have the size estimates
\begin{itemize}
    \item $|B| \geq \left(\frac{\rho}{2 \pi} \right)^r |G|$,
    \item $|B_2| \leq 6^r |B|$, and
    \item for any $\delta \in [0,1]$, $|B_{\delta}| \geq \left(\frac{\delta}{2} \right)^r |B| .$
\end{itemize}
\end{prop}

The ``doubling" estimate 
$$ |B + B| \leq |B_{2}| \leq 6^r |B| $$
shows that $B$ is in some sense an ``approximate subgroup" -- it is (quantitatively) nearly closed under addition, assuming that one considers the factor $6^r$ to be small. In settings where this factor cannot be considered small, it can be useful to consider instead a slightly different quantification of approximate closure under addition, which  motivates the following definition. 

\begin{definition}[Regular Bohr set] \label{regular-bohr-set}
A Bohr set $B$ of rank $r$ is regular if, for all $\delta \in [0,\frac{1}{12 r}]$,
$$ \frac{|B_{1 + \delta}|}{|B|} \leq 1 + 12 r \delta$$ \
and
$$ \frac{|B_{1 - \delta}|}{|B|} \geq 1 - 12 r \delta.$$
\end{definition}

The point here then is that we have the ``doubling" estimate
$$ |B + B_{\delta}| \leq |B_{1 + \delta}| \leq  2 |B| $$
for $\delta \leq 1/12r$. Compared to the bound above, we have removed the exponential dependence on $r$ in the doubling constant at the cost of an exponential-in-$r$ factor loss in the size of one of the summands, which is more acceptable in certain contexts.

Fortunately, regular Bohr sets are easy to obtain:

\begin{prop}[Regularizing a Bohr set {\cite[Section 4.4]{tao-vu}}] \label{regularize}
Given a Bohr set $B$, there is some $\delta \in [\frac{1}{2}, 1]$ so that $B_\delta$ is regular. 
\end{prop}

Ultimately, our interest in Bohr sets in cyclic groups is due to the fact that large Bohr sets are guaranteed to contain a generalized progression which is still fairly large.

\begin{prop}[Large progression in a Bohr set {\cite[Proposition 4.23]{tao-vu}}] \label{progression_in_bohr_set}

Let $G$ be a cyclic group of size $N$, and let
Let $B = \textnormal{Bohr}(\Gamma, \rho) \subseteq G$ be a Bohr set with 
\begin{itemize}
    \item rank $|\Gamma| \leq r$ and
    \item radius $\rho \in [0,1]$.
\end{itemize}
Then $B$ contains a (proper) progression $P$ of rank $r$ and size
$$ |P| \geq  \left( \frac{\rho}{2 \pi r} \right)^r N .
$$
\end{prop}


\subsection{Translation invariance for approximate subgroups} \label{invar}

\begin{prop}[Smoothing an approximate subgroup]\label{smoothing}
Let $A, B$ be finite subsets of an abelian group $G$. Here we would also like to allow for infinite groups, so in the present context we switch to the counting measure: we define the distribution functions
\begin{itemize}
    \item $\pi_{A}(x) := \frac{\1_A(x)}{|A|},$
    \item $\pi_B(x) := \frac{\1_B(x)}{|B|}$.
\end{itemize}
We define the convolution of distributions $\pi_A * \pi_B$ according to the counting measure,
$$ (\pi_A * \pi_B)(x) := \sum_{y \in G} \pi_{B}(y) \pi_{A}(x - y),$$
so that the convolution of two distributions is again a distribution (i.e.\ $\sum_x (\pi_A * \pi_B)(x) = 1$). 

Suppose $S$ is a finite set in $G$ which contains the difference-set $A - B$.
Then, for all $x \in A$, we have the identity
$$ (\pi_S * \pi_B)(x) = \frac{1}{|S|} .$$
More generally, if $\nu$ is any distribution supported on $B$, then still for all $x \in A$ we have
$$ (\pi_S * \nu)(x) = \frac{1}{|S|}. $$
\end{prop}
\begin{proof}
We count the number of pairs $(s,b) \in S \times B$ for which $s + b = x$.
\begin{align*}
    \sum_{s,b} \ind{s + b = x} &= \sum_{s \in S} \ind{s \in x - B} = |S \cap (x - B)| = |x - B| = |B|,
\end{align*}
since $x - B \subseteq A - B \subseteq S$ for any $x \in A$. 
So we do in fact have
$$ (\pi_{S} * \pi_{B})(x) = \frac{|B|}{|S \times B|} = \frac{1}{|S|} $$
whenever $x \in A$. 

Now we consider the case of general $\nu$ supported on $B$. 
We note that $\nu$ may be expressed as a convex-combination of flat distributions supported on $B$ -- that is, we may write
$$ \nu = \E_{B'} [ \pi_{B'} ] $$
with respect to some probability distribution over subsets $B ' \subseteq B$. Then the claim follows from our previous argument since $S$ still contains $A - B'$ for every such $B'$: 
$$  (\pi_S * \nu)(x) = \E_{B'}[ (\pi_S * \pi_{B'})(x) ] = \frac{1}{|S|} $$
for $x \in A$. \qedhere
\end{proof}
\begin{cor}[Strong one-sided approximation]\label{one-sided}
Let $A,B,S \subseteq G$, where $S \supseteq A - B$. 
Suppose that $|S| \leq (1 + \delta)|A|$ for some $\delta \geq 0$.
Then for any nonnegative function $f : G \rightarrow \R_{\geq 0}$ and any distribution $\nu$ supported on $B$, 
$$ \ip{\pi_S * \nu}{f} \geq (1 + \delta)^{-1} \ip{\pi_A}{f} .$$
\end{cor}
\begin{proof}
For every point $x \in A$, we compare $\pi_A(x) = \frac{1}{|A|}$ with $(\pi_{S} * \nu)(x) = \frac{1}{|S|}$. 
\end{proof}
\begin{cor}[Smoothing a regular Bohr set] \label{smoothing-bohr}
Suppose $B$ is a regular Bohr set of rank $r$ in a finite abelian group $G$.
Then, for any nonnegative function $f$ on $G$,
$$ \ip{B_{1 + \delta} * B_{\delta}}{f} \geq (1 - 12 \delta r) \ip{B}{f}. $$
\end{cor}

\begin{proof}
We note that $-B_{\delta} = B_{\delta}$ and $B_{1+\delta} \supseteq B + B_{\delta}$. 
If $\delta \geq 1/12r$ then the claim is trivial. 
Otherwise, we have $|B_{1+\delta}| \leq (1 + 12 \delta r)|B|$, and so $|B|/|B_{1+\delta}|\geq (1+12 \delta r )^{-1} \geq 1 - 12 \delta r$.
\end{proof}

\subsection{Freiman homomorphisms}  \label{freiman-hom}

\begin{definition}[Freiman Homomorphism] \label{freiman}
Suppose we have a set $A \subseteq G$ where $G$ is an abelian group, and $G'$ is another abelian group. A map
$$ \phi : A \rightarrow G'$$
is said to be a Freiman homomorphism of order $t$ if, 
for any $x_1, x_2, \ldots, x_t \in A$ and $y_1, y_2,  \ldots, y_t \in A$, 
$$ \phi(x_1) + \cdots +  \phi(x_t) = \phi(y_1) + \cdots + \phi(y_t)
$$
implies
$$ x_1 + \cdots + x_t = y_1 + \cdots + y_t.
$$
In particular, a linear map $\phi : G \rightarrow G'$ is a $t$-homomorphism on $A$ if and only if $\phi$ is injective on the sumset $tA$. 
\end{definition}

While not immediate, it is easy to check (by making a translation) that a $t$-homomorphism $\phi$ is also a $t'$-homomorphism for $t' < t$; in particular any Freiman homomorphism $\phi$ must at least be an injection. We also point out the trivial property that if $\phi : A \rightarrow G'$ is a  $t$-homomorphism, and $A' \subseteq A$, then the restriction of $\phi$ to $A'$ is also a $t$-homomorphism. 

We use the following simple Freiman homomorphism often.

\begin{example}[Labelling a proper progression] \label{progression_label}

Let
$$ P = \Big\{ a + \sum_{i=1}^r c_i \cdot x_i \; : \; x_i \in [N_i]\Big\} \subseteq G $$
be a (proper) progression, and define the map
$$ \phi : P \rightarrow [N_1] \times [N_2] \times \cdots \times [N_r] \subseteq \Z^r$$
by
$$ \phi(p) = (x_1, x_2, \ldots, x_r). $$
This map is a Freiman homomorphism of all orders.
\end{example}
\begin{proof}
The only potentially tricky point is that $\phi$ is actually well-defined, which is true only because $P$ is proper. Besides this, it is clear that equality of sums of vectors $\sum_{j=1}^t x^j = \sum_{j=1}^t y^j$ within $\Z^r$ implies the equality
$$  a + \sum_{i=1}^r c_i \cdot (x_i^1 + x_i^2 + \cdots + x_i^t)  = 
 a + \sum_{i=1}^r c_i \cdot (y_i^1 + y_i^2 + \cdots + y_i^t) 
$$
in $G$.
\end{proof}

The following is a formalization of a standard trick used (e.g.\ in \cite{gowers01}) to reduce questions regarding $t$-progressions in the integers to some corresponding questions regarding $t$-progressions in $\Z_N$. 
The trick itself is so simple that it is often presented without any corresponding formalization. However, in preparation for some more complicated extensions which can no longer be reasonably handled ``by inspection", we work out the details here with some care.

\begin{example}[Embedding an interval] \label{interval_embed}
Define the natural map
$$ \phi : [N] \rightarrow \Z_N $$
by
$$ \phi(x) = x \mod N .$$
Let $I = [a,a+m] \subseteq [N]$ be some interval with
$$ m < \frac{N}{t}.$$
Then the map $\phi$ is a Freiman $t$-homomorphism when restricted to $I$.
\end{example}
\begin{proof}
The sumset $t I \subseteq \Z$ is contained in $I' = [ta, ta + tm]$. Let $x,y \in I'$. If
$x \equiv y \mod N$, this means that $N$ divides the distance $|x - y|$. However, this distance is at most
$$ (ta + tm) - ta < N.$$
So in fact we must have $|x - y| = 0$. 
\end{proof}

It will turn out to be convenient for us if we could find a more ``intrinsic" formulation of this trick, in the sense that we would like to find a formulation which refers to a nice set $S \subseteq \Z_N$ rather than a nice set $I \subseteq [N]$. 
For instance, we might wish to say that $\phi$ is such that if $S$ is any interval in $\Z_N$ with length at most $N/t$, then the map $\phi$ is a $t$-homomorphism on the pullback $\phi^{-1}(S) \subseteq [N]$. This statement is in fact false (already for $t \geq 2$), which is witnessed by the example $S = \{-1, 0 ,1 \} \subseteq \Z_N$: under our map, this set has preimage
$$ \phi^{-1}(S) = \{1\} \cup \{N-1,N\}, $$
and now the sumset $\phi^{-1}(S) + \phi^{-1}(S)$ contains two multiples of $N$: both $N$ and $2N$. 
This demonstrates alarmingly the importance of certain ``implementation details" of linear maps between groups in the present context which can often be safely ignored in other contexts. 

We can however make the following simple observation that the behavior of the family of sets in $[N]$ obtainable from preimages of intervals in $\Z_N$ is not unboundedly bad; in the terminology of \cite{gowers01, green-note}, an interval in $\Z_N$ corresponds to a union of at most two ``genuine" intervals:

\begin{observation} \label{union_of_intervals}
Fix the map $\phi : [N] \rightarrow \Z_N$ defined by $\phi(x) = x \mod N$. 
Define the family of ``intervals of size $m$" in $\Z_N$ by
$$\mathcal{F}_m = \set{a + \phi([m])}{a \in \Z_N}.\footnote{
Note that we recover the same family if we instead consider sets $a + \phi(I) \subseteq \Z_N$ where $I = [a,b] \subseteq [N]$ is any other interval of size $|I| = m$. 
}$$
For any $I \in \mathcal{F}_m$, the pullback $I' = \phi^{-1}(I)$ is either 
\begin{itemize} 
    \item an interval $I' = [a,b] \subseteq [N]$ of size $|I'| = m$, or
    \item the union of two intervals:  specifically $I' = [1,b] \cup [a,N]$, with $b < m$ and $a > N - m$. 
\end{itemize}
\end{observation}

We record the following immediate consequence of \cref{interval_embed} and \cref{union_of_intervals} -- this statement shows that it is possible to obtain the ``intrinsic" formulation of \cref{interval_embed} we were looking for, so long as we are willing to concede a small fraction of points in $[N]$.

\begin{prop}[A safe set in $\Z_N$] \label{safe_interval}
Fix the map $\phi : [N] \rightarrow \Z_N$ defined by $\phi(x) = x \mod N$. 
Fix a parameter $\delta \in (0,\tfrac{1}{2}]$, and define the ``upper portion" of $[N]$ by
$$U = [\delta N, N] \subseteq [N].$$
Suppose $I \subseteq \Z_N$ is some interval\footnote{
See \cref{union_of_intervals} for what is precisely meant by this.
}
of size $|I| \leq \delta N$,
and that $t \leq 1/\delta$.

Then, for any translation by some $a \in \Z_N$, $\phi$ is a $t$-homomorphism on the set
$$ \phi^{-1}(I + a) \cap U. $$
\end{prop}
\begin{proof}
Using the characterization in \cref{union_of_intervals}, we see that the set $\phi^{-1}(I + a) \cap U$ is in fact equal to some interval\footnote{It may be equal to the ``trivial" interval $\emptyset$ -- this case is also fine.}
in $[N]$ of size at most $\delta N$. This is the case considered in \cref{interval_embed}, so the claim follows. 
\end{proof}

We introduce the following definition to capture this phenomenon in general.

\begin{definition}[$t$-safe set] \label{t-safe}
Consider two abelian groups $G,G'$ and a subset $A \subseteq G$. 
Fix an injection $\phi : A \rightarrow G$. 
We say that a set $B \subseteq G'$ is ``$t$-safe" with respect to $\phi$ if, for every translation $\theta \in G'$, the restriction
$$ \phi : A \cap \phi^{-1}(B+\theta) \rightarrow G' $$
is a Freiman homomorphism of order $t$. 
\end{definition}

By a simple generalization of \cref{safe_interval}, we can describe a highly structured, reasonably large set which is safe with respect to the natural embedding of 
$ [\delta N_1, N_1] \times [\delta N_2, N_2] \times \cdots \times [\delta N_r, N_r]$ into  $\Z_{N_1} \times \Z_{N_2} \times \cdots \times \Z_{N_r}.
$

\begin{prop}[A safe set in $\Z_{N_1} \times \Z_{N_2} \times \cdots \times \Z_{N_r}$] \label{safe_bohr}
Fix some natural numbers $N_1, N_2, \ldots N_r$ and  a parameter $\delta \in (0,\tfrac{1}{2}]$.
Consider the set $ [N_1] \times [N_2] \times \cdots \times [N_r] \subseteq \Z^r,$
and define the ``upper portion" of this set:
$$ U := \prod_{i=1}^r [\delta N_i, N_i].$$
Fix the map $\phi : [N_1] \times [N_2] \times \cdots \times [N_r] \rightarrow G = \Z_{N_1} \times \Z_{N_2} \times \cdots \times \Z_{N_r}$ defined by
$$\phi(x) = x \mod (N_1, N_2, \ldots, N_r).$$
For $i=1,2,\ldots,r$, let $m_i \leq \delta N_i$, $I_i := [m_i],$ $B_i := I_i \textnormal{ mod } N_i \subseteq \Z_{N_i}$, and let
$$ B := B_1 \times B_2 \times \cdots \times B_r.
$$
If $t \leq 1/\delta$, then $B$ (and hence also any translate of $B$) is $t$-safe with respect to the restriction $\phi : U \rightarrow G.$
That is, the natural embedding $\phi(x) = x \mod (N_1, N_2 \ldots, N_r)$ is a $t$-homomorphism on the set
$$ \phi^{-1}(B + \theta) \cap U$$
for any translation $a \in G$. 
\end{prop}
\begin{proof}

We note that $\phi^{-1}(B + \theta) \cap U = \prod_{i=1}^r \phi_i^{-1}(B_i + \theta_i) \cap [\delta N_i, N_i]$, and we apply \cref{safe_interval} on each coordinate. To finish, we use the easily verifiable fact that if $\phi_1 : A_1 \rightarrow G_1$ and $\phi_2 : A_2 \rightarrow G_2$ are both Freiman $t$-homomorphisms, then the map from $A_1 \times A_2$ to $G_1 \times G_2$ given by
$$ (x_1,x_2) \mapsto (\phi_1(x_1), \phi_2(x_2))
$$
is also a Freiman $t$-homomorphism.
\end{proof}

\section{Proof of \cref{progressions-in-spread-sets} and \cref{progressions-in-spread-sets-specific}} \label{proof-of-3-progs-in-the-integers}

\cref{progressions-in-spread-sets} follows straightforwardly from the following three claims, which are proved in this section. Here, for easier reading we present a slightly informal formulation of each one, suppressing some minor details. The more specific \cref{progressions-in-spread-sets-specific} follows from the details found in the corresponding complete formulations.

\begin{prop}[Informal version of \cref{obtaining-nice-config}: passing to a ``nice" configuration]
Suppose $A \subseteq [N_1] \times [N_2] \times \cdots \times [N_r]$ is a spread configuration with density $\mu$. Then either 
$A$ has at least 
$$ r^{-O(r)} \cdot \frac{|A|^3}{|G|} $$
3-progressions, or
there is a large subset of $A$ which is Freiman $2$-isomorphic to a ``nice" configuration $A' \subseteq P' = [N_1'] \times [N_2'] \times \cdots \times [N_r']$ with the assurance that
\begin{itemize}
    \item $A' \subseteq P'$ is only slightly less spread, and
    \item $|A'| \geq r^{-O(r)} |A|$. 
\end{itemize}

\end{prop}

\begin{prop}[Informal version of \cref{embedding}: embedding a ``nice" configuration]
Suppose $A \subseteq [N_1] \times [N_2] \times \cdots \times [N_r]$ is a ``nice" configuration with density $\mu \geq 2^{-d}$. Then (firstly) the group $G = \Z_{N_1} \times \Z_{N_2} \times \cdots \times \Z_{N_r}$ must be cyclic.
Let $\phi$ be the natural embedding $\phi : A \rightarrow G$. 
Then we have
$$ \|\phi(A) \star \phi(A)\|_{k,B \star B} \geq 1 + \Omega(1) $$
for some $k \leq O(d)$ and some regular Bohr set $B$ which
\begin{itemize}
    \item has rank $r$ and radius $\rho \approx 1/r$, and
    \item is $2$-safe with respect to the embedding $\phi : A \rightarrow G$.
\end{itemize}

\end{prop}

\begin{lem}[Informal version of \cref{svr-local}: obtaining a density increment]
Suppose $A$ is a subset of a cyclic group $G$ with density $\mu \geq 2^{-d}$,
and let $B$ be a regular Bohr set with rank $r$ and radius $\rho \approx 1/r$. 
If
$$ \|A \star A\|_{k, B \star B} \geq 1 + \Omega(1), $$
then 
$$ \frac{|A \cap P|}{|P|} \geq \left(1 + \Omega(1) \right) \mu
$$
for some generalized progression $P$ which
\begin{itemize}
    \item has rank $r' \leq r + \textnormal{poly}(k,d)$,
    \item has size $|P| \geq 2^{-\textnormal{poly}(k,d,r)}  |G|$, and
    \item is contained in some translate $B + \theta$.
\end{itemize}
\end{lem}

\begin{proof}[Proof of \cref{progressions-in-spread-sets-specific}.]
Let $A_0 = A$ and $r_0 = r$. 
In the case that $A$ has at least
$$ r^{-O(r)} \frac{|A|^3}{|P|} $$
3-progressions, we are done. 
Otherwise, we apply \cref{obtaining-nice-config} with $\eps = 2^{-9}$.
Either this immediately supplies us with the desired density increment, or else we obtain a subset $A_0' \subseteq A_0$ and a $2$-homomorphic bijection 
$ \phi_1 : A_0' \rightarrow A_1 \subseteq P_1$ where 
\begin{itemize}
    \item $P_1 = [p_1] \times [p_2] \times \cdots \times [p_{r_1}] \subseteq \Z^{r_1},$ 
    \item $|A_1| \geq r^{-O(r)} |A_0|$ and $r_1 \leq r_0$, 
    \item $A_1 \subseteq P_1$ is $\delta$-nice with $\tfrac{1}{10} \geq \delta \geq \Omega(1/r)$, and 
    \item $\mu_1 := \frac{|A_1|}{|P_1|} \geq (1 - 5 \eps) \mu .$
\end{itemize}
We apply \cref{embedding}. 
Letting  $\phi_2 : A_1 \rightarrow A_2 \subseteq G$ be the natural embedding into the group $G = \Z_{p_1} \times \Z_{p_2} \times \cdots \times \Z_{p_{r_1}}$,
we have
$$ \| A_2 \star A_2 \|_{k, B \star B} \geq 1 + \tfrac{1}{4} $$
for some $k \leq O(d)$ and some large regular Bohr set $B$ which is $2$-safe with respect to $\phi_2$. Applying \cref{svr-local}, we obtain a subset
$$ A_2' = A_2 \cap P_3 $$
where $P_3$ is some proper progression which
\begin{itemize}
    \item has rank $r_3 \leq r + O(d^8)$,
    \item has size $|P_3| \geq r^{-O(r)} \cdot 2^{-O(d^2 r)} \cdot 2^{-O(d^{10})} \cdot |G|$,
    \item is contained in some translate $B+\theta$, and
    \item provides the density increment
    $$ \mu_3 = \frac{|A_2 \cap P_3|}{|P_3|} \geq \left(1 + \tfrac{1}{32}\right) \mu_1 \geq  \left(1 + \tfrac{1}{32}\right) \left(1 - \tfrac{5}{2^9}\right)\mu \geq \left(1 + \tfrac{1}{50}\right) \mu. $$
\end{itemize}
As the density does not decrease, we note also that
$$ |A_2'| \geq r^{-O(r)} \cdot 2^{-O(d^2 r)} \cdot 2^{-O(d^{10})} |A_2|. 
$$
Since $A_2' \subseteq P_3 \subseteq B + \theta$, the map $\phi_2$ is a $2$-homomorphism when restricted to $\phi_2^{-1}(A_2')$. 

Finally, we let $\phi_3 : P_3 \rightarrow \Z^{r_3}$ be the simple Freiman homomorphism described in \cref{progression_label} and we arrive at our final configuration
$$ A_3 = \phi_3 (A_2') \subseteq \phi_3 (P_3) \subseteq \Z^{r_3}. $$
Letting $\phi := \phi_3 \circ \phi_2 \circ \phi_1$ and $A' := \phi^{-1}(A_3)$,
we obtain the desired increment $(A',\phi)$. \qedhere
\end{proof}

\subsection{Obtaining a density increment}

\begin{lem}[Structure vs.\ Pseudorandomness in $\Z_N$ -- local variant] \label{svr-local}
Consider
\begin{itemize}
    \item a cyclic group $G$.
    \item a subset $A \subseteq G$ of size $|A| \geq 2^{-d} |G|$,
    \item a regular Bohr set $B = \Bohr(\Gamma, \rho)$ of rank $|\Gamma| = r$, 
    \item $k \in \N$, and
    \item a constant\footnote{We allow the implied constants in the $O$-notation here to depend on $\eps_0$.} $\eps_0 \in [0,1]$.
\end{itemize}
If 
$$ \| A \star A \|_{k, B \star B} \geq 1 + \eps_0, $$
then there is a (proper) generalized progression $P$, contained in some translate $B + \theta$, with
\begin{itemize}
    \item rank $r' \leq r + O(k^4 d^4)$ and
    \item size $|P| \geq   \rho^{r'} \cdot  r^{-O(r)} \cdot 2^{-O(dkr)} \cdot 2^{-O(d^5 k^5)} \cdot  N$,
\end{itemize}
and such that
$$ \frac{|A \cap P|}{|P|} \geq \left( 1 + \frac{\eps_0}{8} \right) \frac{|A|}{|G|}. $$
\end{lem}

We will need the following translation-invariance lemma due to Schoen and Sisask.

\begin{lem}[Special case of 
{\cite[Theorem 5.4]{schoen-sisask}}] \label{schoen-sisask}
Consider 
\begin{itemize}
    \item a finite abelian group $G$,
    \item subsets $X,Y \subseteq G$,
    \item a regular Bohr set $B = \Bohr(\Gamma, \rho) \subseteq G$ with rank $|\Gamma| = r$,
    \item an indicator function $f : G \rightarrow \{0,1\}$, and 
    \item $\eps \in [0,1]$. 
\end{itemize}
If
$$ |Y + B| \leq 2^{d} |Y|$$
and
$$ |X + Y + B| \leq 2^{s} |X|,$$
then there is a Bohr set $B' = \Bohr(\Gamma', \rho') \subseteq B$ such that
\begin{itemize}
    \item $ \left| \big((X * Y) \star f  \big)(b) - \big((X * Y) \star f  \big)(0)  \right| \leq \eps $ for all $b \in B'$,
    \item $B'$ has rank $r' \leq r + O \left( d s^3 / \eps^2  +  d s \log(1/\eps)^2 / \eps^2 \right)$, and
    \item $B'$ has radius $\rho' \geq   \rho  \cdot \eps \cdot 2^{-s/2} / (r^2  r') $.
\end{itemize}
In particular,
$$ \left|\ip{D * X * Y}{f} - \ip{X * Y}{f} \right| \leq \eps $$
for any density function $D$ supported on $B'$.
\end{lem}

It is maybe not immediately clear that this is indeed a special case of Lemma 5.4 of \cite{schoen-sisask}, so we elaborate.
We assume familiarity with the statement as it appears in \cite{schoen-sisask}. 
Let $L \subseteq G$ denote the support of our indicator function $f = \1_L$.
The key point is that for our formulation here, one can observe that the only values of $f$ which can possibly play any role are those values $f(z)$ for which $z$ is ``near" zero -- more accurately: $z \in X + Y + B$. This assertion makes use of the fact that we indeed have $0 \in B$, so that $X+Y \subseteq X + Y + B$. 
Thus, it makes no difference for us to insist that $L \subseteq X + Y + B$ -- that is, to replace  $f$ by $f \cdot \1_{X+Y+B}$. 
Given this, the remaining translation between the statements is straightforward: we apply their Lemma 5.4 with $S = B$. Their sets $A,M,L$ correspond to our sets $Y,X,-L$. Their parameters $\eta, K$ correspond to our parameters $2^{-s},2^d$. Their lemma provides a bound on 
$$  \left| \E_{\substack{x \in X \\ y \in Y}} f(x + y + z + b ) - 
\E_{\substack{x \in X \\ y \in Y}} f(x + y + z ) \right| $$
for all $b \in B'$ and all $z \in G$, whereas we seek to control this difference only at the single point $z = 0$. 

\begin{proof}[Proof of \cref{svr-local}]
By assumption we have
$$ \ip{B \star B}{|A \star A|^{k}} = \ip{B}{B * (A \star A)^{k}} \geq \left(1 + \eps_0 \right)^k .$$
We set $\delta_0 = \eps_0/100r$, and pick $\delta \in [\delta_0/2,\delta_0]$ so that $B_{\delta}$ is regular. We apply \cref{smoothing-bohr} to get
$$ \ip{B_{1 + \delta} * B_{\delta} }{B * (A \star A)^{k}} \gtrapprox \left(1 + \eps_0 \right)^k,$$
yielding
$$ \ip{B_{1 + \delta} * B_{\delta} }{B * (A \star A)^{k}}^{1/k} \geq 1 + \frac{\eps_0}{2} .$$
It follows that for some fixed $\theta \in B_{1 + \delta}$, the translate $C := -B + \theta = B + \theta$ satisfies
$$ \ip{B_{\delta} * C}{(A \star A)^{k}}^{1/k} \geq 1 + \frac{\eps_0}{2}. $$
By (possibly) increasing $k$ only slightly we may assume that $2^{1 - \eps_0 k/4} \leq 2^{-5} \eps_0.$
Let $\eps = \eps_0/4$ and apply \cref{local-witness} to get dense subsets $X \subseteq C $ and $Y \subseteq B_{\delta}$, both with relative density at least $2^{-O(kd)}$, such that
$$ \ip{X * Y}{f} \leq \frac{\eps}{8}, $$
where
$$ f(x) := \ind{\big(A \star A\big)(x) \leq 1 + \eps} .$$

Now, we invoke \cref{schoen-sisask} with sets $X,Y$ and with Bohr set $B_{\eta}$, for some $\eta \in [\frac{\delta^2}{2}, \delta^2]$, to get a final Bohr set
$$ B' = \Bohr(\Gamma', \rho') \subseteq B_{\eta}$$
full of translations which leave $X * Y$ approximately fixed, in the limited sense that they have little effect on the value $\ip{X * Y}{f}$.
We then pass to a large progression $P \subseteq B'$ with
$$ \ip{P * X * Y}{f} \leq \frac{\eps}{4} .$$
Since
$$ |Y + B_{\eta}| \leq |B_{\delta} + B_{\eta}| \leq |B_{\delta + \delta^2}| \leq 2 |B_{\delta}| \leq 2^{O(dk)} |Y|, $$
and
$$ |X + Y + B_{\eta}| \leq |(B + \theta) + B_{\delta + \delta^2}| = 
|B + B_{\delta + \delta^2}| \leq |B_{1 + \delta + \delta^2}| \leq 2 |B| \leq 2^{O(dk)} |X|,
$$
we obtain a Bohr set $B'$ with parameters
\begin{itemize}
    \item rank $r'  \leq r + O(d^4 k^4)$ and 
    \item radius $\rho' \geq \rho \cdot r^{-O(1)} \cdot 2^{-O(dk)}$,
\end{itemize}
and a progression $P \subseteq B' \subseteq B$ with parameters
\begin{itemize}
    \item rank $r' = r + O(d^4 k^4)$ and 
     \item size $|P| \geq   \rho^{r'} \cdot  r^{-O(r)} \cdot 2^{-O(dkr)} \cdot 2^{-O(d^5 k^5)} \cdot  N$,
\end{itemize}
Since $\ip{P * X * Y}{f} \leq \eps /4$, it must be the case that 
$ \ip{P'}{f} \leq \eps / 4$
for some translate $P'$ of $P$.
Recalling the definition of $f$, we obtain the lower-bound
$$ \ip{P'}{A \star A} \geq \left(1 - \frac{\eps}{4} \right) (1 + \eps) \geq 1 + \frac{\eps}{2} .$$
Again, it follows that
$$ \ip{P''}{A} \geq 1 + \frac{\eps}{2} = 1 + \frac{\eps_0}{8} ,$$
for some further translate $P''$ of $P'$,
which concludes the proof since $P''$ is itself a progression with the same size and rank as $P$. \qedhere
\end{proof}

\subsection{Embedding a nice configuration}

\begin{definition}
Fix a parameter $\delta > 0$.
For an interval $I = [a,b] \subseteq \Z$, the ``upper portion" of $I$ is the subset
$$ U := [a + 2 \delta (b-a), b], $$
and the ``middle slice" of $I$ is the subset
$$ M := \left[a + \tfrac{1}{2}(b-a) + 1, a + (\tfrac{1}{2} + \delta)(b-a)\right] .$$
\end{definition}

The point of these ad hoc definitions is that if we consider an interval $I = [N] \subseteq \Z$, if $x,y \in U$ and $z \in M$, then
$$ x + y = 2z \mod N$$
only if 
$$ x + y = 2z. $$
This can be checked by inspection: since $x+y,2z \in [2N]$, the only way we can have $N$ dividing $|x+y - 2z|$ while  $|x+y - 2z| \neq 0$ is that $|x + y - 2z| = N$. The sets $U$ and $M$ are designed specifically so that (for any choice of $\delta \in [0,\frac{1}{2}]$) this is not possible. 

\begin{definition}
Fix a parameter $\delta > 0$.
For some intervals $I_i \subseteq \Z$,
and a cube $P = \prod_{i=1}^r I_i \subseteq \Z^r$, the ``upper portion" of $P$ is the subset
$$ U := \prod_{i=1}^r U_i$$
and the ``middle slice" of $P$ is the subset
$$ M := \prod_{i=1}^r M_i .$$
\end{definition}

We note the density lower-bounds
$$ \frac{|U|}{|P|} \geq (1-2 \delta)^r \geq 1 - 2 \delta r, $$
$$ \frac{|M|}{|P|}  \geq \Omega(\delta)^r .$$

\begin{definition}[Nice configuration] \label{nice-configuration}
Consider a configuration $A \subseteq [N_1] \times [N_2] \times \cdots \times [N_r] \subseteq \Z^r$ with density
$$ \mu = \frac{|A|}{N_1 N_2 \cdots N_r} .$$
Consider a parameter $\delta \leq 1/2r$, and 
let $v,m$ denote the vectors
\begin{itemize}
    \item $v = \left( \lfloor (\frac{1}{2} + \frac{\delta}{2}) N_1 \rfloor, \lfloor (\frac{1}{2} + \frac{\delta}{2})N_2 \rfloor, \ldots, \lfloor (\frac{1}{2} + \frac{\delta}{2}) N_r \rfloor \right)$,
    \item $m = \left( \lfloor \frac{\delta}{3} N_1 \rfloor, \lfloor \frac{\delta}{3} N_2 \rfloor, \ldots, \lfloor \frac{\delta}{3} N_r \rfloor \right)$.
\end{itemize}
We say the configuration is ($\delta$)-nice if 
\begin{enumerate}
    \item  $N_1, N_2, \ldots, N_r$ are some distinct prime numbers not including $2$,
    \item $A$ is in fact contained within the upper-portion $U = \prod_{i=1}^r [2 \delta N_i, N_i]$, 
    \item Given $b, b' \in \Z^r$ drawn uniformly at random from $[-m_1, m_1] \times [-m_2, m_2] \times \cdots \times [-m_r, m_r]$, we have the following ``weighted" density lower-bound in the middle-slice:
$$  \E_{b,b'} \left[ \1_A(v + b - b') \right] \geq \frac{\mu}{2}. $$
    \item We have a uniform bound on the number of representations of any point $2z$:
    $$ \max_{z \in A} R_A(2z) \leq   \frac{\mu}{4} \cdot |A|.
    $$

\end{enumerate}
\end{definition}

\begin{prop}[Embedding a nice configuration] \label{embedding}
Suppose $A \subseteq [N_1] \times [N_2] \times \cdots \times [N_r] \subseteq \Z^r$ is a $\delta$-nice configuration with density
$$ \frac{|A|}{N_1 N_2 \cdots N_r} \geq 2^{-d} .$$
Let $\phi$ be the natural map from $\Z^r$ to the group $G = \Z_{N_1} \times \Z_{N_2} \times \cdots \times \Z_{N_r}$ given by $\phi(x) = x \mod (N_1, N_2, \ldots, N_r)$. 
Then, there is an even integer $k \leq O(d)$ and a Bohr set $B = \Bohr(\Gamma, \rho) \subseteq G$ such that
$$\| \phi(A) \star \phi(A) \|_{k, B \star B} \geq 1 + \tfrac{1}{4} $$
where the Bohr set $B$
\begin{itemize}
    \item has rank $|\Gamma| = r$,
    \item has radius $\rho \geq \Omega(\delta)$
    \item is regular, and 
    \item is $2$-safe with respect to the restricted map $\phi : A \rightarrow G$.
\end{itemize}
\end{prop}

\begin{proof}[Proof of \cref{embedding}]

In light of \cref{local-decoupling} and \cref{local-upper-to-lower}, it suffices to establish the following: the existence of a $2$-safe regular Bohr set $B$ (with adequate rank and radius) such that, for some translate $D = T_{\theta} (B \star B)$,
$$ \|\phi(A) * \phi(A) - 1\|_{k, D} \geq \tfrac{1}{2} $$
for some even integer $k \leq O(d)$. 
For any $z \in A \cap M$, where 
$$ M := \prod_{i=1}^{r} \left[ \frac{N_i}{2},\frac{N_i}{2} + \delta N_i \right] \subseteq \Z^r, $$
we have by construction that
$$ R_{\phi(A)}(2 \phi(z)) = R_{A}(2z) \leq \frac{|A|^2}{4|G|}.$$
That is, we have not increased the number of representations of $2z$ as a sum of two elements in $A$ by allowing equality mod $(N_1,N_2, \ldots, N_r)$. 
After re-normalizing to a density function we obtain
$$ \left(\phi(A) * \phi(A) \right)(2 \phi(z)) = \frac{|G|}{|A|^2} R_{\phi(A)}(2 \phi(z)) \leq \frac{1}{4} $$
for such $z$. 
Let $v,m \in \Z^r$ be the vectors in \cref{nice-configuration}, and let
$b,b'$ be uniformly random points drawn from the cube
$$ C_m :=  [-m_1,m_1] \times [-m_2,m_2] \times \cdots \times [-m_r,m_r] \subseteq \Z^r. $$
We note that by construction we always have $v + b - b' \in M$.
It follows (from this and the third criterion for niceness in \cref{nice-configuration}) that $2\phi(v + b - b')$ has a non-negligible chance of being a point where $\phi(A) * \phi(A)$ differs substantially from uniform, so
$$ \E_{b,b' \in C_m} \left| \Big(\phi(A) * \phi(A) \Big) \left(2 \phi(v) + 2 \phi(b) - 2 \phi(b') \right) - 1 \right|^{k} \geq \tfrac{1}{2} \cdot 2^{-d} \cdot \left(\tfrac{3}{4} \right)^{k},
$$
still for any choice of $k$. Define the set
$$ B := \set{2 \phi(b)}{b \in C_m} \subseteq \phi(2 C_m) \subseteq G .$$
By \cref{safe_bohr}, this set is $2$-safe with respect to the restriction
$$ \phi : U \rightarrow G, $$
and since $A \subseteq U$ it is also $2$-safe with respect to the further restriction
$ \phi : A \rightarrow G. $

We apply essentially the observation in \cref{bohr_interval} to say that $B$ is in fact describable as a Bohr set of the group $G = \Z_{N_1} \times \Z_{N_2} \times \cdots \times \Z_{N_r}$. Indeed, for $i = 1,2,\ldots,r$ let $\alpha_i \in \Z_{N_i}$ be the multiplicative inverse of $2 \mod N_i.$ Now consider the set of characters $\Gamma$ consisting of
$$ \gamma_i : x \mapsto e^{2 \pi i \alpha_i x_i / N_i} $$
for $i = 1,2,\ldots,r$. As in \cref{bohr_interval}, we can say that
$$ B = \Bohr(\Gamma, \rho) $$
for some choice of $\rho \geq \Omega(\delta).$ Additionally, due to the simple structure of $B$ we can verify directly that it is regular:
$$ |B_{1+\eta}| \leq \left(1 + \tfrac{2}{\pi}\eta \right)^{r} |B| \leq (1 +  \eta r) |B|$$
and
$$ |B_{1-\eta}| \geq \left(1 - \tfrac{2}{\pi}\eta)^{r} |B| \geq (1 - \eta r\right) |B|$$
for $\eta \leq 1/r$. 

Thus, we have that for the density $D = T_{-2 \phi(v)} (B \star B)$,
$$\ip{D}{|\phi(A) * \phi(A) - 1|^{k}} \geq 2^{-(d+1)} \cdot \left(\tfrac{3}{4} \right)^{k}. $$
Now choose $k$ large enough that $2^{(d+1)/k} \leq \frac{3}{2}$, so that we get
\begin{equation*}
  \|\phi(A) * \phi(A) - 1 \|_{k, D} = \ip{D}{|\phi(A) * \phi(A) - 1|^{k}}^{1/k} \geq 2^{-(d+1)/k} \cdot \tfrac{3}{4} \geq \tfrac{1}{2}. \qedhere
\end{equation*}

\end{proof}

\subsection{Obtaining a nice configuration}

\begin{prop}  \label{obtaining-nice-config}
Fix a constant\footnote{We allow the implied constants in the $O$-notation here to depend on $\eps$.} $\eps \in (0,\frac{1}{10}]$ and let $\delta := \eps/2r$. 
Suppose $A \subseteq [N_1] \times [N_2] \times \cdots \times [N_r]$ is a configuration with density $\mu$. Then, either $A$ has at least
$$ r^{-O(r)} \cdot \mu \cdot |A|^2
$$
solutions to $x + y = 2z$ with $(x,y,z) \in A^3$, or
we can obtain
a subset $A' \subseteq A$ and a Freiman homomorphism (of all orders) $\phi : A' \rightarrow P' \subseteq \Z^{r'}$ with
\begin{itemize}
    \item size $|A'| \geq r^{-O(r)} |A|$,
    \item rank $r' \leq r$,
\end{itemize}
and such that we obtain either: (i) a $(\delta)$-nice configuration with density 
$$ \mu':= \frac{|A'|}{|P'|} \geq \left(1 - 5 \eps \right) \mu, $$
or (ii) a configuration which is not necessarily nice but has increased density:
$$ \mu' \geq \left(1 + \frac{\eps}{2}\right) \mu. $$
\end{prop}
\begin{proof}
Essentially the idea is to consider the restriction of $A$ to some subcube of the form $P' = \prod_{i=1}^r [a_i, a_i  + p_i - 1]$ for some distinct primes $p_i$ each roughly of size $\delta N_i$. From here it is easy enough to just translate the whole configuration $(A \cap P') \subseteq P'$ down to $\prod_{i=1}^r [p_i]$ -- such a translation is clearly a Freiman homomorphism of all orders. 

To ensure that we can make such a choice of primes we use the fact that the number of primes in any interval $[n/2, n]$ is at least $\Omega(n/\log n)$ -- this fact follows from the prime number theorem, but such a lower bound also follows from substantially more elementary arguments such as Erd\"os' proof of Bertrand's postulate. 

Without loss of generality suppose that $N_1 \geq N_2 \geq \cdots \geq N_r$. We wish to make a choice of distinct odd primes $p_i \in [\delta N_i /2, \delta N_i]$. Let $n_0$ be the smallest natural number such that, for any $n \geq n_0$, the number of primes $p \in [n/2,n]$ is at least $r + 1$. It is the case that $n_0$ is at most $O(\frac{r \log r}{\log \log r})$, although we note it would be fine for us even just to say that $n_0 \leq r^{O(1)}$. If $\delta N_r \geq n_0$ then clearly we can make such choice of primes $p_i$ -- otherwise, we will need to ``fix" $A$ in the coordinates $i$ corresponding to size-lengths $N_i$ which are too small. Specifically, we pass from $A$ to the subset 
$$ A' = \set{a \in A}{a_r = x}, $$
where we choose $x \in [N_r]$ so that the size of $A'$ is maximized. 
This results in no loss in density:
$$ \frac{|A'|}{N_1 N_2 \cdots N_{r-1}} \geq \frac{|A|}{N_1 N_2 \cdots N_r} $$
and only a small loss in size:
$$ |A'| \geq \left( \tfrac{\delta}{n_0} \right)  |A| \geq r^{-O(1)} |A|.
$$
With the last coordinate fixed, the projection map $\phi : \Z^r \rightarrow \Z^{r-1}$ is clearly a Freiman homomorphism of all orders on the set $A'$. Continuing in this way, we may assume (at the cost of only a factor $r^{-O(r)}$ loss in size, and no loss in density)
that we begin with a configuration $A \subseteq \prod_{i=1}^{r} [N_i]$ where $r$ has only decreased and that there are some distinct prime numbers $p_i \in [\delta N_i / 2, \delta N_i]$ for $i=1,2,\ldots,r$. 

At this point we pause to apply a pruning step to address niceness criterion number four. Let $A' \subseteq A$ be the subset of points $x \in A$ with
$$ R_A(2x) \leq \frac{4}{\eps} \cdot \E_{z \in A}R_A(2z).
$$
By a Markov inequality, $|A'| \geq (1 - \eps / 4) |A|$, 
and so $A'$ has density $\mu' \geq (1 - \eps/4) \mu$. 

We now consider the restriction of $A'$ to some subcube
$$ P_{a} := a + [p_1] \times [p_2] \times \cdots \times [p_{r}] \subseteq \Z^{r}.
$$ 
Let $U_a$ denote the corresponding upper-portion of this cube.
Our plan is to set 
$$A'' := A' \cap U_a$$
for some choice of $a \in \Z^{r}$, and our final configuration will be  (a translation of) $A'' \subseteq P_a$. 
Let 
\begin{itemize}
\item $v = \left( \lfloor (\frac{1}{2} + \frac{\delta}{2}) p_1 \rfloor, \lfloor (\frac{1}{2} + \frac{\delta}{2})p_2 \rfloor, \ldots, \lfloor (\frac{1}{2} + \frac{\delta}{2}) p_r \rfloor \right),$
\item $m = \left( \lfloor \frac{\delta}{3} p_1 \rfloor, \lfloor \frac{\delta}{3} p_2 \rfloor, \ldots, \lfloor \frac{\delta}{3} p_r \rfloor \right),
$
\end{itemize}
and define
$$ B = [-m_1,m_1] \times [-m_2,m_2] \times \cdots \times [-m_{r},m_{r}] \subseteq \Z^{r}.$$

We consider the quantities
$$ \mu_1''(a) := \frac{|A''|}{|U_a|} = \E_{u \in U_a} \1_{A'}(u) $$
and
$$ \mu_2''(a) := \E_{b,b' \in B} \1_{A'}(a + v + b - b'). $$
We note that we may assume that
$$ \mu_1''(a), \mu_2''(a) \leq (1 + \eps) \mu' $$
for all translations $a \in \Z^{r}$ -- otherwise we are done immediately as we would obtain a density increment either onto some $U_a$ or onto some translate of $B$:
$$ \mu'' \geq (1+\eps)(1-\eps/4) \mu \geq (1 + \eps/2) \mu. $$

We consider a uniformly-random translation $a$ drawn from the set 
$$ [-p_1, N_1] \times [-p_2, N_2] \times \cdots  \times [-p_r, N_r] \subseteq \Z^{r}. $$
Applying \cref{smoothing} (or more specifically, its \cref{one-sided}), it follows that both of the expectations $\E[\mu_1''(a)]$ and $\E[\mu_2''(a)]$ are at least 
$$(1 + \delta)^{-r} \mu' \geq (1 - \delta r) \mu' \geq (1 - \eps/2) \mu.$$

We can now apply a Markov inequality to the (nonnegative) random variables
$$ X_1(a) := (1 + \eps)- \frac{\mu_1''(a)}{\mu'} $$
and
$$ X_2(a) := (1 + \eps) - \frac{\mu_2''(a)}{\mu'} $$
to conclude that there is a nonzero probability that we choose a translation $a$ simultaneously satisfying
$$ X_1(a), X_2(a) \leq 3 \eps, $$
or rather
\begin{equation*}
 \mu_1''(a), \mu_2''(a) \geq (1 - 4 \eps) \mu'. 
\end{equation*}
Finally, we note that
\begin{equation*}
  \mu'' := \frac{|A''|}{|P_a|} = \frac{|U_a|}{|P_a|} \frac{|A''|}{|U_a|} \geq (1 - \eps/2) \mu_1''(a) \geq (1-\eps/2)(1-4\eps)(1-\eps/4) \mu \geq (1 - 5 \eps) \mu
\end{equation*}
and

\begin{equation*}
   R_{A''}(2x) \leq R_{A'}(2x) \leq \frac{4}{\eps} \frac{1}{|A|}\sum_{z \in A} R_{A}(2z) 
\end{equation*}
for all $x \in A''.$
To conclude, we note that the quantity on the right hand side is at most $\frac{\mu''}{4} |A''|$, 
except in the case that
\begin{equation*}\sum_{z \in A} R_{A}(2z) 
\geq  \frac{\eps}{16} \cdot  \mu''  \cdot |A| |A''| 
\geq r^{-O(r)} \cdot \mu \cdot |A|^2, 
\end{equation*}
i.e., the case that $A$ has many 3-progressions. \qedhere

\end{proof}

\bibliographystyle{alpha}
\bibliography{main}

\begin{thebibliography}{ALWZ20}

\bibitem[ALWZ20]{alwz20}
Ryan Alweiss, Shachar Lovett, Kewen Wu, and Jiapeng Zhang.
\newblock Improved bounds for the sunflower lemma.
\newblock In {\em Proceedings of the 52nd Annual ACM SIGACT Symposium on Theory of Computing}, pages 624--630, 2020.

\bibitem[Beh46]{behrend46}
Felix~A Behrend.
\newblock On sets of integers which contain no three terms in arithmetical progression.
\newblock {\em Proceedings of the National Academy of Sciences}, 32(12):331--332, 1946.

\bibitem[BK12]{bk12}
Michael Bateman and Nets Katz.
\newblock New bounds on cap sets.
\newblock {\em Journal of the American Mathematical Society}, 25(2):585--613, 2012.

\bibitem[BKM22]{GHZadd}
Mark Braverman, Subhash Khot, and Dor Minzer.
\newblock Parallel repetition for the {GHZ} game: Exponential decay.
\newblock {\em CoRR}, abs/2211.13741, 2022.

\bibitem[Blo16]{bloom16}
Thomas~F Bloom.
\newblock A quantitative improvement for {R}oth's theorem on arithmetic progressions.
\newblock {\em Journal of the London Mathematical Society}, 93(3):643--663, 2016.

\bibitem[Bou99]{bourgain99}
Jean Bourgain.
\newblock On triples in arithmetic progression.
\newblock {\em Geometric and Functional Analysis}, 9(5):968--984, 1999.

\bibitem[Bou08]{bourgain08}
Jean Bourgain.
\newblock {R}oth’s theorem on progressions revisited.
\newblock {\em Journal d'Analyse Math{\'e}matique}, 104(1):155--192, 2008.

\bibitem[BS19]{bs19}
Thomas~F Bloom and Olof Sisask.
\newblock Logarithmic bounds for {R}oth's theorem via almost-periodicity.
\newblock {\em DISCRETE ANALYSIS}, 2019.

\bibitem[BS20]{bs20}
Thomas~F Bloom and Olof Sisask.
\newblock Breaking the logarithmic barrier in {R}oth's theorem on arithmetic progressions.
\newblock {\em arXiv preprint arXiv:2007.03528}, 2020.

\bibitem[CFL83]{CFL}
Ashok~K. Chandra, Merrick~L. Furst, and Richard~J. Lipton.
\newblock Multi-party protocols.
\newblock In David~S. Johnson, Ronald Fagin, Michael~L. Fredman, David Harel, Richard~M. Karp, Nancy~A. Lynch, Christos~H. Papadimitriou, Ronald~L. Rivest, Walter~L. Ruzzo, and Joel~I. Seiferas, editors, {\em Proceedings of the 15th Annual {ACM} Symposium on Theory of Computing, 25-27 April, 1983, Boston, Massachusetts, {USA}}, pages 94--99. {ACM}, 1983.

\bibitem[Cha02]{chang02}
Mei-Chu Chang.
\newblock A polynomial bound in {F}reiman's theorem.
\newblock {\em Duke mathematical journal}, 113(3):399--419, 2002.

\bibitem[CLP17]{clp17}
Ernie Croot, Vsevolod~F Lev, and P{\'e}ter~P{\'a}l Pach.
\newblock Progression-free sets in are exponentially small.
\newblock {\em Annals of Mathematics}, pages 331--337, 2017.

\bibitem[CS10]{cs10}
Ernie Croot and Olof Sisask.
\newblock A probabilistic technique for finding almost-periods of convolutions.
\newblock {\em Geometric and functional analysis}, 20(6):1367--1396, 2010.

\bibitem[EG17]{eg17}
Jordan~S Ellenberg and Dion Gijswijt.
\newblock On large subsets of with no three-term arithmetic progression.
\newblock {\em Annals of Mathematics}, pages 339--343, 2017.

\bibitem[Elk10]{elkin10}
Michael Elkin.
\newblock An improved construction of progression-free sets.
\newblock In {\em Proceedings of the twenty-first annual ACM-SIAM symposium on Discrete Algorithms}, pages 886--905. SIAM, 2010.

\bibitem[FL17]{fl17}
Jacob Fox and L{\'a}szl{\'o}~Mikl{\'o}s Lov{\'a}sz.
\newblock A tight bound for {G}reen's arithmetic triangle removal lemma in vector spaces.
\newblock In {\em Proceedings of the Twenty-Eighth Annual ACM-SIAM Symposium on Discrete Algorithms}, pages 1612--1617. SIAM, 2017.

\bibitem[Gow01]{gowers01}
William~T Gowers.
\newblock A new proof of {S}zemer{\'e}di's theorem.
\newblock {\em Geometric \& Functional Analysis GAFA}, 11(3):465--588, 2001.

\bibitem[Gow11]{gowers-blog}
William~T Gowers.
\newblock What is difficult about the cap set problem?
\newblock \url{http://gowers.wordpress.com/2011/01/11/what-is-difficult-about-the-cap-set-problem/}, 2011.
\newblock Accessed: 1-5-2023.

\bibitem[Gre99]{green-note}
Ben Green.
\newblock Progressions of length 3 following {S}zemer{\'e}di.
\newblock \url{https://people.maths.ox.ac.uk/greenbj/papers/szemeredi-roth.pdf}, 1999.

\bibitem[GW10]{gw10}
Ben Green and Julia Wolf.
\newblock A note on {E}lkin’s improvement of {B}ehrend’s construction.
\newblock In {\em Additive number theory}, pages 141--144. Springer, 2010.

\bibitem[HB87]{heath87}
David~Rodney Heath-Brown.
\newblock Integer sets containing no arithmetic progressions.
\newblock {\em Journal of the London Mathematical Society}, 2(3):385--394, 1987.

\bibitem[HHL18]{hhl18}
Hamed Hatami, Kaave Hosseini, and Shachar Lovett.
\newblock Structure of protocols for xor functions.
\newblock {\em SIAM Journal on Computing}, 47(1):208--217, 2018.

\bibitem[KB80]{kaas-buhrman}
Rob Kaas and Jan~M Buhrman.
\newblock Mean, median and mode in binomial distributions.
\newblock {\em Statistica Neerlandica}, 34(1):13--18, 1980.

\bibitem[KSS18]{kss18}
Robert Kleinberg, Will Sawin, and David Speyer.
\newblock The growth rate of tri-colored sum-free sets.
\newblock {\em Discrete Analysis}, page 3734, 2018.

\bibitem[Lov15]{lovett15}
Shachar Lovett.
\newblock An exposition of {S}anders' quasi-polynomial {F}reiman-{R}uzsa theorem.
\newblock {\em Theory of Computing}, pages 1--14, 2015.

\bibitem[Mes95]{meshulam95}
Roy Meshulam.
\newblock On subsets of finite abelian groups with no 3-term arithmetic progressions.
\newblock {\em Journal of Combinatorial Theory, Series A}, 71(1):168--172, 1995.

\bibitem[O{'}B11]{obryant11}
Kevin O{'}Bryant.
\newblock Sets of integers that do not contain long arithmetic progressions.
\newblock {\em the electronic journal of combinatorics}, 18(P59):1, 2011.

\bibitem[Rot53]{roth53}
Klaus~F Roth.
\newblock On certain sets of integers.
\newblock {\em J. London Math. Soc}, 28(104-109):3, 1953.

\bibitem[San10]{sanders10}
Tom Sanders.
\newblock Popular difference sets.
\newblock {\em Online J. Anal. Comb.}, 5(5):4, 2010.

\bibitem[San11]{sanders11}
Tom Sanders.
\newblock On {R}oth's theorem on progressions.
\newblock {\em Annals of Mathematics}, pages 619--636, 2011.

\bibitem[San12a]{sanders12certain}
Tom Sanders.
\newblock On certain other sets of integers.
\newblock {\em Journal d'Analyse Math{\'e}matique}, 1(116):53--82, 2012.

\bibitem[San12b]{sanders12}
Tom Sanders.
\newblock On the {B}ogolyubov--{R}uzsa lemma.
\newblock {\em Analysis \& PDE}, 5(3):627--655, 2012.

\bibitem[Sch15]{schoen15}
Tomasz Schoen.
\newblock New bounds in {B}alog-{S}zemer{\'e}di-{G}owers theorem.
\newblock {\em Combinatorica}, 35(6):695--701, 2015.

\bibitem[Sch21]{schoen21}
Tomasz Schoen.
\newblock Improved bound in {R}oth's theorem on arithmetic progressions.
\newblock {\em Advances in Mathematics}, 386:107801, 2021.

\bibitem[Shk17]{shkredov17}
ID~Shkredov.
\newblock Some remarks on the {B}alog--{W}ooley decomposition theorem and quantities {D}+, {D}$\times$.
\newblock {\em Proceedings of the Steklov Institute of Mathematics}, 298:74--90, 2017.

\bibitem[SS13]{ss13}
Tomasz Schoen and Ilya~D. Shkredov.
\newblock Higher moments of convolutions.
\newblock {\em Journal of Number Theory}, 133(5):1693--1737, 2013.

\bibitem[SS16]{schoen-sisask}
Tomasz Schoen and Olof Sisask.
\newblock {R}oth’s theorem for four variables and additive structures in sums of sparse sets.
\newblock In {\em Forum of Mathematics, Sigma}, volume~4. Cambridge University Press, 2016.

\bibitem[Sze90]{szemeredi90}
Endre Szemer{\'e}di.
\newblock Integer sets containing no arithmetic progressions.
\newblock {\em Acta Mathematica Hungarica}, 56(1):155--158, 1990.

\bibitem[TV06]{tao-vu}
Terence Tao and Van~H Vu.
\newblock {\em Additive combinatorics}, volume 105.
\newblock Cambridge University Press, 2006.

\end{thebibliography}

\appendix

\section{Comparison with prior approaches and other measures of additive pseudorandomness} \label{comparison}

In this section, we compare our techniques for controlling the number of solutions to $a + b = c$ with some previous approaches, and we compare our conditions for quantifying ``additive pseudorandomness" with some analogous conditions considered by other works.

\textbf{Additive energy.}
We begin by discussing the notion of ``additive energy", which is a good central measure of additive structure to keep in mind as it can be easily compared with everything else we discuss. For a set $A \subset G$ the additive energy of $A$ is the quantity $E(A) = \sum_{x \in A} R_{A}^{-}(x)^2$, which is also the same as the number of solutions to $a_1 - a_2 = a_3 - a_4$ with $a_i \in A$. It is always between $|A|^2$ and $|A|^3$. It can be expressed in the density formulation as the quantity 
$$E(A) = \frac{|A|^4}{|G|} \|A \star A\|_{2}^2.$$
Thus, for $k = 2$ our ``self-regularity" condition corresponds simply to some bound on the additive energy.\footnote{Similarly, for $k \geq 2$ our self-regularity condition is simply a bound on the higher-order energy $\sum_{x} R_{A}^{-}(x)^k$ which is also a well-studied concept; see e.g.\ \cite{ss13}. 
}
For sufficiently dense sets, there is a different lower bound which improves on the trivial bound $E(A) \geq |A|^2$: it is always the case that $E(A)$ is at least $|A|^4/|G|$. The bounds 
$$|A|^4/|G| \leq E(A) \leq |A|^3$$
translate gracefully to the following equivalent bounds on $\|A \star A\|_2^2$. By Jensen's inequality and Young's inequality,
$$ 1 = \|A \star A\|_{1}^2 \leq \|A \star A\|_{2}^2 \leq \|A\|_2^2. $$
Note that set of size $|A| = 2^{-d}|G|$ corresponds to a density function with $\|A\|_{2}^2 = 2^{d}$. More generally, the bounds above apply to any density function $A(x)$. 

It is a typical sort of problem in additive combinatorics to prove something nontrivial about the ``additive structure" of $A$ when $E(A)$ is ``large". In some contexts this means that $E(A)$ only slightly less than maximal: for instance if $E(A) \geq \frac{1}{10}|A|^3$, or perhaps $E(A) \geq |A|^{3 - 1/10}$. In the present work, we are interested in what can be said about $A$ given that the additive energy is only slightly greater than minimal: for instance if $\|A \star A\|_{2}^2 \geq 1 + \frac{1}{10}$. Thus, the alternative normalization $\|A \star A\|_{2}^2$ of the additive energy is a natural choice for this setting.

\textbf{The Roth-Meshulam argument.} We review the classical Roth-Meshulam argument (\cite{roth53, meshulam95}) in the case of $\F_q^n$, and discuss its relation to the present work. For an additive character 
$$e_{\alpha} : \F_q^n \rightarrow \C,$$ we consider the corresponding Fourier coefficient\footnote{See \cref{prelims} for details regarding  definitions and normalization conventions related to the  Fourier expansion.}
$$ \hat{A}(\alpha) := \E_{x \in A} e_{\alpha}(x) = \ip{A}{e_{\alpha}}. $$
We note that, for any set $A$, $\hat{A}(0) =  1 $, and generally that $|\hat{A}(\alpha)| \leq 1$.

In the Roth-Meshulam argument, the key quantity for measuring the additive pseudorandomness of a set $A$ is the size of its largest nontrivial Fourier coefficient:
$$  \max_{\substack{\alpha \in \F_q^n \\ \alpha \neq 0}}  |\hat{A}(\alpha)|. $$ 
The overall argument can be summarized compactly as follows.
\begin{prop}
Let $A,B,C \subseteq \F_q^n$ be sets of size at least $2^{-d} |\F_q^n|$.
\begin{enumerate}
\item If $\|A\|_{\perp,1} \leq 1 + \eps$ then
$$ \max_{ \alpha \neq 0}   |\hat{A}(\alpha)| \leq 2 \eps.$$
\item $$|\ip{A}{B * C} - 1| \leq \max_{ \alpha \neq 0}  |\hat{A}(\alpha)| \cdot
 \|B\|_2 \|C\|_2 \leq \max_{ \alpha \neq 0}  |\hat{A}(\alpha)| \cdot  2^{d} . 
$$
\end{enumerate}
\end{prop}
\begin{proof}
For a linear subspace $W$, let $P_W$ denote the projection operator 
$$ P_W : f \mapsto W^{\perp} * f.$$
We note that the quantity $\|A\|_{\perp,r}$ can be equivalently characterized as
\begin{equation}
    \|A\|_{\perp,r} = \max_{\substack{W \subseteq \F_q^n \\ \textnormal{dim}(W) \leq r }} \| P_W A \|_{\infty}. 
\end{equation}
We also note that for a function
$$ A(x) = \sum_{\alpha \in \F_q^n} \hat{A}(\alpha) e_{\alpha}(-x) $$
we have 
$$ (P_W A)(x) = \sum_{\alpha \in W} \hat{A}(\alpha) e_{\alpha}(-x). $$
Fix some $\alpha \neq 0$, and let $W$ be some one-dimensional subspace containing it. 
We have
$$ |\hat{A}(\alpha)| = |\ip{P_W A - 1}{e_{\alpha}}| \leq \|P_W A - 1\|_1. $$
To finish, we use the basic fact that for any density function $F$, 
$$ \|F - 1\|_{1} = 2 \| (F - 1)_{+} \|_{1}  \leq 2 \|(F - 1)_{+} \|_{\infty} = 2( \|F\|_{\infty} - 1). $$
This proves the first claim. For the second claim we make the following calculation.
\begin{align*}
|\ip{A}{B * C} - 1| &\leq \sum_{\alpha \neq 0} |\hat{A}(\alpha)| |\hat{B}(\alpha)| |\hat{C}(\alpha)| \\
&\leq \max_{\alpha \neq 0} |\hat{A}(\alpha)| \cdot \sum_{\beta} |\hat{B}(\beta)| |\hat{C}(\beta) | \\
&\leq \max_{\alpha \neq 0}|\hat{A}(\alpha)| \cdot  \sqrt{\sum_{\beta} |\hat{B}(\beta)|^2} \sqrt{\sum_{\beta} |\hat{C}(\beta) |^2 } \\
&= \max_{\alpha \neq 0} |\hat{A}(\alpha)| \cdot \|B\|_2 \|C\|_2 . \qedhere
\end{align*}
\end{proof}
Thus, the overall structure of the argument very much resembles our own: to begin, there is a ``primitive" pseudorandom condition (spreadness with respect to subspaces of codimension one), which can be ensured by a density increment argument. Then there is an intermediate pseudorandom condition (i.e.\ a bound on $ \max_{\alpha \neq 0} |\hat{A}(\alpha)|$), which firstly can be derived from spreadness, and secondly is sufficient to directly control the quantity $\ip{A}{B * C}$ when $B$ and $C$ are large.

We note another similarity to the present work. Consider the following Fourier-analytic interpretation of the additive energy of $A$:
$$ \|A \star A\|_{2}^2 = 1 +  \sum_{\alpha \neq 0} |\hat{A}(\alpha)|^4 = 1 + \|A \star A - 1\|_{2}^2. $$
It is well known that (for very large sets $A$) a bound on $\max_{\alpha \neq 0}  |\hat{A}(\alpha)|$ is roughly equivalent to a bound on $\|A \star A - 1\|_{2}^2$:
$$ \max_{\alpha \neq 0}  |\hat{A}(\alpha)|^4 \leq \sum_{\alpha \neq 0} |\hat{A}(\alpha)|^4
\leq \max_{\alpha \neq 0} |\hat{A}(\alpha)|^2 \cdot \sum_{\beta} |\hat{A}(\beta)|^2
= \max_{\alpha \neq 0} |\hat{A}(\alpha)|^2 \cdot \|A\|_2^2. 
$$
Thus, another reasonable summary of the Roth-Meshulam argument is that for very large sets $A,B,C$,
\begin{itemize}
\item $\|A \star A\|_{2}$ can be controlled by $\|A\|_{\perp,1}$, and
\item $\ip{A}{B * C}$ can be controlled by $\|A \star A\|_{2}$.
\end{itemize}

\textbf{The Bloom-Sisask physical space argument.}
In \cite{bs19}, Bloom and Sisask show (in particular) how to obtain parameters for the cap-set problem very near to those given by the Roth-Meshulam approach, but with arguments working almost exclusively in physical space (rather than Fourier space). 
This is a notable similarity to the present work, where we make very little use of any ``quantitative" Fourier analysis.\footnote{
We implicitly rely on a small amount of quantitative Fourier analysis through our use of \cref{sanders-invar-appendix}, whose proof relies on Chang's inequality.
Besides this, our remaining Fourier-analytic arguments (i.e., regarding the decoupling inequality and spectral positivity) are ``qualitative". 
}
In particular, the Bloom-Sisask physical space argument more closely resembles our approach than any other prior work.
We offer the following interpretation of their argument (as it applies to the problem of controlling the quantity $\ip{A * B}{C}$, for sets $A,B,C \subseteq \F_q^n$ of density $2^{-d}$, by a spreadness assumption).

There are three steps. Firstly: the plan is to control $\ip{A * B}{C}$ by obtaining a bound on the key quantity $\|A * B - 1\|_{k}$, for some $k \approx d$. This is precisely what is done also in our work, and already it distinguishes the approach from many others, which consider instead some key quantity describable in Fourier space. Secondly, Bloom and Sisask argue (by a combination of the Croot-Sisask lemma\footnote{
The version of the Croot-Sisask lemma used is notably more efficient in its dependence on the error parameter than, e.g.,\ \cref{croot-sisask}.
The parameters involved in this alternative version are more subtle and it can be convenient to handle the cases, say, $\|A * B\|_{k} \leq 10$ and $\|A * B\|_{k} \geq 10$ separately.
} and Chang's inequality) that $\|A * B - 1\|_{k}$ is similar to $\|V * A * B - 1\|_{k}$, for some large subspace $V$.
Thus, if $$ \|A * B - 1\|_{k} \geq \Omega(1) $$
then
$$ \|V * A * B - 1\|_{k} \geq \Omega(1) $$
for some $V$ of codimension
$r \leq 2^{d + O(\log d + \log k)}.$
The final point is that such deviation from $1$ is impossible if $A$ and $B$ are both spread. Specifically, we apply the claim given below with $D := V * A$ and $D' := V * B$ (noting that $V * V = V$). Overall, the Bloom-Sisask physical space argument shows that if $|\ip{A}{B * C} - 1| \geq \Omega(1)$, then either $\|A\|_{\perp, r} \geq 1 + \Omega(1)$ or $\|B\|_{\perp,r} \geq 1 + \Omega(1)$, for some
$$r \leq 2^{d + O(\log d)}.$$

\begin{prop}
Let $D$ and $D'$ be density functions on a finite abelian group $G$. For $\eps \geq 0$,
if
$$ \|D\|_{\infty} \leq 1 + \eps \textnormal{  and  } \|D'\|_{\infty} \leq 1 + \eps $$
then
$$ \|D * D' - 1\|_{\infty} \leq \eps $$
\end{prop}
\begin{proof}
    The pointwise upper bound $D * D' - 1 \leq \eps$ follows easily by averaging: $\|D * D'\|_{\infty} \leq \|D\|_{\infty} \leq 1 + \eps$.
    Now consider the pointwise lower bound. If $\eps > 1$, such a bound is trivial, so suppose $\eps \leq 1$. We have
    $$ 0 \leq (1 + \eps - D) * (1 + \eps - D') = (1+\eps)^2 - 2(1 + \eps) + D * D' = -1 + \eps^2 + D * D' ,$$
    and so
    \begin{equation*}
        D * D \geq 1 - \eps^2 \geq 1 - \eps. \qedhere
    \end{equation*}
\end{proof}

\textbf{Understanding the power of density increments.}
Let us say, only in the context of the present section, that a set $A \subseteq \F_q^n$ of size $|A| = 2^{-d}|\F_q^n|$ is simply ``spread" if
$$ \|A\|_{\perp, r} \leq 1 + O\left( \frac{r (d+1)}{n \log n} \right) $$
for all $r \leq n/2.$ One can check that this corresponds to the condition that $A$ has ``no strong increments" considered by Bateman and Katz \cite{bk12} within their regime of interest: namely $d = \Theta(\log n)$. We are interested in this specific choice of parameters essentially because it represents the limit of what can be reasonably obtained by a generic density increment argument.\footnote{We do not intend this as any kind of formal claim.} 
Given this definition, we consider the following substantially more concrete variant of \cref{central-question}.

\begin{question} \label{central-question-spread}
Suppose $A,B,C \subseteq \F_q^{n}$ are sets of size at least $2^{-d} |\F_q^n|$.
For what values of $d$ can we deduce that the number of solutions to $a + b + c = 0$
is within a factor $2$ of $|A||B||C|/|\F_q^n|$, 
from an assumption that
\begin{enumerate}
\item[(i)] one,
\item[(ii)] two, or
\item[(iii)] three
\end{enumerate}
of the sets are ``spread" (in the specific sense described above)? 
\end{question}

We note that Gowers asks essentially part (i) of this question in a blog post \cite{gowers-blog}. More accurately, he asks only about when we can deduce the existence of at least one solution.
The Roth-Meshulam argument described above is relevant here: it requires only a spreadness assumption on the single set $A$. Given that $A$ is spread, it tells us that
$$ |\ip{A}{B * C} - 1| \leq O\left(\frac{(d+1) 2^{d}}{n \log n} \right), $$
and so it provides the following positive answer regarding (i): to ensure that 
$$\ip{A}{B * C} \in \left[\tfrac{1}{2}, 2\right],$$
a spreadness assumption on just $A$ is sufficient, for values of $d$ up to
$$d = \lg n - c$$
where $c$ is some constant. 
We consider the following example which shows that for this question, the answer obtained from the Roth-Meshulam argument is actually best-possible, despite the fact that it makes use only of a bound on $\|A\|_{\perp,1}$.
Let $d = \lg n$, and consider\footnote{The same example can be constructed over other small fields by considering $\F_q^{t} \times \F_q^{n-t}$ with $q^t \approx n$.} 
$$A := \left( \F_2^d \setminus \{0\} \right) \times \F_2^{n-d} \subseteq \F_2^n$$
and
$$ B := \{0\} \times \F_2^{n-d} \subseteq \F_2^n. $$
We have $B + B = B$ and $\ip{A}{B * B} = 0$. 
One can check that $A$ is indeed spread simply because, for any $r$,
$$ \|A\|_{\perp,r} \leq \|A\|_{\infty} = \frac{1}{1 - 2^{-d}} \approx 1 + \frac{1}{n}.$$
That is, $A$ has no strong increment onto a large affine subspace in particular because it has \textit{no strong increment onto a set of any kind}.
Thus, for
$$ d \geq \lg n, $$
it is not possible to control $\ip{A}{B * C}$ by assuming spreadness for just $A$; this fully answers the question of Gowers. However, we point out that the situation changes dramatically if we assume spreadness for just $A$ but ask only for \textit{upper bounds}. Then the Roth-Meshulam argument is no longer optimal, and indeed, in \cref{II} we show that $k$-regularity follows from $r$-spreadness for $r \approx k^7 d$. In particular we have that $\ip{A}{B * C} \leq 2$ whenever $A$ is ``spread" for $d$ up to roughly
$$ d \approx n^{1/9}. $$
\textbf{Summary for \cref{central-question-spread}.} We summarize what we understand regarding \cref{central-question-spread}.
\begin{itemize}
\item For $d \leq \lg n - O(1)$, the Roth-Meshulam argument shows that a spreadness assumption on just $A$ is sufficient to ensure $\ip{A}{B * C} \in [\frac{1}{2}, 2]$. Moreover, only a bound on $\|A\|_{\perp,1}$ is needed.
\item For $d \geq \lg n$, a spreadness assumption on just $A$ does not ensure $\ip{A}{B * C} > 0$.
\item For $d \lessapprox n^{1/9}$, a spreadness assumption on just $A$ ensures $\ip{A}{B * C} \leq 2$.
\item For $d \lessapprox n^{1/9}$, a spreadness assumption on $A$ and $B$ ensures $\ip{A * B}{C} \in [\frac{1}{2}, 2]$.
\end{itemize}

\textbf{The Fourier sum of cubes measure.} 
Recall the estimate
$$ |\ip{A}{B * C} - 1| \leq \sum_{\alpha \neq 0} |\hat{A}(\alpha)| |\hat{B}(\alpha)| |\hat{C}(\alpha)| $$
which was used in the Roth-Meshulam argument. A sensible approach for part (iii) of \cref{central-question-spread}, and one which is very natural from a Fourier-analytic perspective, is to depart from the earlier argument and treat this quantity more symmetrically by considering 
$$ \sum_{\alpha \neq 0} |\hat{A}(\alpha)| |\hat{B}(\alpha)| |\hat{C}(\alpha)| \leq
\left(\sum_{\alpha \neq 0} |\hat{A}(\alpha)|^3 \right)^{1/3} \left(\sum_{\alpha \neq 0} |\hat{B}(\alpha)| ^3\right)^{1/3} \left(\sum_{\alpha \neq 0} |\hat{C}(\alpha)|^3 \right)^{1/3} .
$$
This shows that a bound of, say,
$$ \sum_{\alpha \neq 0} |\hat{A}(\alpha)|^3 \leq \tfrac{1}{2} $$
on the sum of cubes of nontrivial Fourier coefficients is sufficient to control $\ip{A * B}{C}$, if we assume such a bound it for all three sets. 
Indeed, this was the approach taken in the work of Bateman and Katz on the cap-set problem (\cite{bk12}), who show that
$$ \sum_{\alpha \neq  0} |\hat{A}(\alpha)|^3 \geq \Omega(1) $$
is sufficient to infer that $A$ has density increment onto a large affine subspace, for values of $d$ up to $d = (1 + c) \lg n$ for some small constant $c > 0$, notably breaking the technical barrier described above. 

It has on occasion been taken as the obvious starting point for any sort of analytic approach to the 3-progression problem to begin with the assumption $\sum_{\alpha \neq 0} |\hat{A}(\alpha)|^3 \geq \Omega(1)$ and then to see what we can conclude about $A$. In light of this, it is interesting to note that even in view of our work it seems still unclear whether that approach is viable for obtaining strong bounds.
Specifically, we ask the following technical question -- we don't know of any specific application related to its resolution, but it seems interesting from a technical perspective for trying to better understand the relation between the Fourier-analytic approach with various ``physical space" techniques such as the Croot-Sisask lemma and sifting. 
\begin{question}
Let $A \subseteq G$ be a set of size $|A| \geq 2^{-d} |G|$.
Suppose that
$$ \sum_{\alpha \neq 0} |\hat{A}(\alpha)|^3 \geq \Omega(1).
$$
Does it follow that
$$ \|A \star A - 1\|_{k} \geq \Omega(1) $$
for some $k \leq \textnormal{poly}(d)$?
\end{question}
We make two remarks regarding this. Firstly, the answer is affirmative if we instead start with, say, the stronger assumption
$$ \left| \sum_{\alpha \neq 0 } \hat{A}(\alpha)^3 \right| \geq \Omega(1).$$
In this case we can write $\sum_{\alpha \neq 0 } \hat{A}(\alpha)^3 = \ip{A * A - 1}{B}$, where $B(x) = A(-x)$, and bound this approximately by $\|A \star A - 1\|_{d}$.
Secondly, we note a connection in the opposite direction. For even $k \in \N$ we have the Fourier-analytic interpretation
$$ \|A \star A - 1\|_{k}^k = \sum_{\substack{\alpha_1 + \alpha_2 + \cdots + \alpha_k = 0 \\ \alpha_i \neq 0}} |\hat{A}(\alpha_1)|^2 |\hat{A}(\alpha_2)|^2  \cdots |\hat{A}(\alpha_k)|^2 . $$
Applying Young's inequality to this gives
$$ \|A \star A - 1\|_{k} \leq \left(\sum_{\alpha \neq 0} |\hat{A}(\alpha)|^{2 + \frac{2}{k-1}} \right)^{1 - \frac{1}{k}}. $$

Finally, we can remark (at least superficially) on a connection between our self-regularity condition and another key quantity considered both in the work of Bateman and Katz as well as the work of Bloom and Sisask (\cite{bk12},\cite{bs20}). For $k \in \N$ we have the alternative Fourier-analytic interpretation
$$\|A \star A\|_{2k}^{2k} = \ip{(A \star A)^{k}}{(A \star A)^{k}} 
= \sum_{\beta} \left(\sum_{\alpha_1 + \alpha_2 + \cdots + \alpha_{k} = \beta} |\hat{A}(\alpha_1)|^2 |\hat{A}(\alpha_2)|^2  \cdots |\hat{A}(\alpha_{k})|^2 \right)^2 .
$$
By comparison, the authors of the aforementioned works consider various sets of the form
$$ \Delta = \set{\alpha \in G}{\eta_1 \leq |\hat{A}(\alpha)| \leq \eta_2} $$
as well as the resulting quantity
$$ \sum_{\beta} \left(\sum_{\alpha_1 + \alpha_2 + \cdots + \alpha_{k} = \beta} 
 \1_{\Delta}(\alpha_1) \cdot  \1_{\Delta}(\alpha_2)  \cdots  \1_{\Delta}(\alpha_{k})
\right)^2 ,$$
for various choices of $k$.

\section{Detailed account of the proof of \cref{II}} \label{extended-proof-overview}

We proceed to explain our approach for proving \cref{II}.
Let us focus only on the second claim, which we recall states essentially the following. Suppose that a large set $A \subseteq \F_q^n$, $|A| \geq 2^{-d} |\F_q^n|$, fails to be self-regular:
$$ \|A \star A\|_{k} \geq 1 + \eps.$$
Then we can find an affine subspace $V \subseteq \F_q^n$ of codimension at most $\textnormal{poly}(d,k,1/\eps)$ giving the density increment
$$ \ip{V}{A} \geq 1 + \Omega(\eps).$$
For simplicity we also focus only on the case $k = 2$ and $\eps = 1$; 
this case is already nontrivial and sufficient to illustrate the main obstacles and how they will be overcome. 

Our starting point is to instead look for something stronger: a ``density increment" 
$$ \ip{V}{A \star A} \geq 1 + \Omega(1).$$
This is a key ``leap of faith" in the proof, in the sense that we cannot offer a reason to expect (a priori) that this should be possible, even assuming that our original task is possible. However, we note that, firstly, it is certainly sufficient by a simple averaging argument. Secondly, if this approach is indeed workable it would be quite nice from a mechanical perspective, as we start with an assumption on the density function $A \star A$ (i.e., \ that $\|A \star A\|_{2} \geq 2$), and we now seek a conclusion about the \textit{same object}: $\ip{V}{A \star A} \geq 1 + \Omega(1)$. This is in contrast to the typical structure of an \textit{inverse problem} -- considered generally to be fairly difficult -- where we make an assumption on $A \star A$ and seek a conclusion regarding $A$; by taking this approach, we can suppress the fact that we are secretly working on an inverse problem at the outset.
To elaborate on the averaging argument, we have
$$ \ip{V}{A \star A} = \ip{A * V}{A} = \E_{\substack{x \in V\\ a \in A}} A(x + a) = \E_{a \in A} \left[\E_{x \in V + a} A(x) \right], $$
so if $\ip{V}{A \star A} \geq 1 + \Omega(1)$ then we must have a density increment $\ip{V'}{A} \geq 1 + \Omega(1)$ onto some fixed affine subspace $V' = V + a$. 

We pause to consider whether our new goal is at least analytically plausible. By this, we mean that we ask if we can at least obtain 
$$ \ip{F}{A \star A} \geq 1 + \Omega(1)$$
for \textit{some} high-entropy density function $F$ -- one with $\|F\|_{\infty} \leq 2^{\textnormal{poly}(d)}$ -- without asking that $F$ have any particular additive structure. 
Ideally, we would also not need to appeal to the additive structure of our density $D = A \star A$, and instead, we see what can be said just from the basic analytic facts $D \geq 0$, $\|D\|_1 = 1$, $\|D\|_2 \geq 2$, and $\|D\|_{\infty} \leq 2^{d}$. 
One can check that the best choice of $F$ satisfying such an entropy constraint on $\|F\|_{\infty}$ will be (essentially) the uniform density over some super-level set of $D$. However, there is also a simple choice which is already quite satisfactory, $F = D$:
$$ \ip{D}{D} = \|D\|_{2}^2 \geq 4.$$
So, what we are asking for is at least analytically plausible -- there is a high entropy density function $F$ witnessing $D \gg 1$ -- and our task is to argue that we may take $F$ to have a high degree of additive structure.

\subsection{Sanders' invariance lemma}

To motivate our next step, we consider the following lemma due to Sanders regarding the translation-invariance of the convolution of two large sets $A * B$. Sanders obtains this lemma by combining the Croot-Sisask lemma \cite{cs10},  Chang's inequality \cite{chang02}, and some additional Fourier-analytic arguments. The specific formulation given below does not appear explicitly in his work, 
so we include a proof in the appendix.\footnote{See however \cite[Appendix: Proof of Theorem 11.1]{sanders12} or \cite{lovett15} for some similar statements.}

\begin{lem}[Sanders' invariance lemma \cite{sanders12}] \label{sanders-invariance}
Suppose $A,B \subseteq \F_q^n$ are sets of sizes $|A| \geq 2^{-d} |\F_q^n|$ and $|B| \geq 2^{-k} |\F_q^n|$.
Fix a bounded function $f : \F_q^n \rightarrow [-1,+1]$.
Then, for any $\eps \geq 2^{-k}$, there exists a linear subspace $V$ (possibly depending on $f$) with codimension at most $O(d k^3 / \eps^2)$ satisfying
$$ \left| \ip{V * A * B}{f} - \ip{A * B}{f} \right| \leq \eps. $$
\end{lem}

We suggest the following interpretation of this result. 
It is ``almost" saying that $A * B$ is close in statistical distance to $V * A * B$ for some large subspace $V$, in light of the dual characterization of the $1$-norm distance:
$$\|D - D'\|_{1} = \max_{f: \F_q^n \rightarrow [-1,+1]} \ip{D - D'}{f}.$$
However, it is not quite this strong since the subspace $V$ is allowed to depend on the dual witness $f$.
Under this interpretation, the lemma can be productively compared with Chang's inequality via the Fourier-analytic identity
$$ (V * A * B)(x) = \sum_{\alpha \in V^{\perp}} \widehat{A}(\alpha) \widehat{B}(\alpha) e_{\alpha}(-x).
$$
The importance of Sanders' lemma for us can be captured succinctly by the following immediate consequence. We state it in terms of the notation defined by \cref{perp-norm} and \cref{star-norm} in \cref{technical-introduction}, which we repeat here for the reader's convenience.
\begin{equation*} 
\|f\|_{\perp, r} := \max_{\substack{\textnormal{affine subspace }V \subseteq \F_q^n \\ \textnormal{Codim}(V) \leq r}} \ip{V}{f},
\end{equation*}
\begin{equation*}
\|f\|_{*,k} := \max_{\substack{B,C \subseteq G \\ \|B\|_{\infty}, \|C\|_{\infty} \leq 2^{k}}} \ip{B * C}{f}.
\end{equation*}

\begin{cor} \label{bootstrap}
Let $d \geq 1$ and $\eps \in [2^{-d},1]$. 
For any bounded function $f : \F_q^n \rightarrow [0,1]$, we have
$$ \|f\|_{\perp,d^4/\eps^2} \geq \|f\|_{*,d} - O(\eps). $$
\end{cor}
That is, Sanders shows us how to bootstrap a high-entropy density function $F$ with mild additive structure witnessing $\ip{F}{f} \geq \gamma$ into a reasonably high-entropy density function $F'$ with strict additive structure witnessing $\ip{F'}{f} \geq \gamma - O(\eps)$. 
Sanders applies this fact to the difference-set indicator function $f := \1_{A-A}$ to obtain his solution to the Approximate Bogolyubov Problem (\cref{sanders-approx-bogolyubov}): indeed, we clearly have $\|\1_{A-A}\|_{\star,d} = 1$, as witnessed by $A \star A$ itself.

Inspired by this, we might hope (possibly naively) that we could apply such an argument to the function $f = A \star A$, as we have noted already that
$$ \ip{A \star A}{A \star A} \geq 4,$$
and so $\|A \star A \|_{*,d} \geq 4.$
Certainly this is not permitted by Sanders' lemma as stated, since $A \star A$ is not a bounded function, but we can for example try to see what can be learned by applying Sanders' lemma to various approximate level-set indicator functions
$$ f(x) = \ind{\eta_1 \leq (A \star A)(x) \leq \eta_2}. $$
For the sake of discussion, let us consider the following (purposefully slightly vague) notion.
Say that ``$F$ robustly witnesses $D \gg 1 + \eta$" if (firstly)
$$ \ip{F}{D} \geq 1 + \eta,$$
and additionally that $F$ witnesses a deviation of such magnitude typically, rather than just on average.
Certainly we would include the following exemplary case: if
$$\ip{F}{\1_S} \geq 1 - \eta/2,$$
where 
$$S = \set{x}{D(x) \geq 1 + 2\eta}$$ 
and $\eta \leq 1/2$; in this case we indeed have
$$ \ip{F}{D} \geq (1 + 2\eta)\ip{F}{\1_S} \geq 1 + \eta.$$

In light of Sanders' invariance lemma, we see that to reach our goal it suffices to find some large sets $B,C \subseteq |\F_{q}^n|$ of sizes $|B|,|C| \geq 2^{-\textnormal{poly}(d)}|\F_q^n|$ such that the convolution $B * C$ robustly witnesses $A \star A \gg 1 + \Omega(1)$. Indeed, if we can show that
$$ \|f\|_{*, k}  \geq 1 - \eta/2$$
where
$$f(x) = \ind{(A \star A)(x) \geq 1 + 2 \eta},$$
for some $k \leq \textnormal{poly}(d)$ and some $\eta \geq \Omega(1)$, then we can apply \cref{bootstrap} with $\eps = \Theta(\eta)$ to find a large affine subspace with
$$ \ip{V}{A \star A} \geq (1 + 2\eta) \ip{V}{f} \geq 1 + \Omega(1).$$

\subsection{Finding a robust witness}

We proceed to consider the problem of finding a robust witness to $A \star A \gg 1 + \Omega(1)$ of the form $B * C$. 
It is tempting to check if simply $F = A \star A$ is already a good robust witness for itself. For general density functions $D$ obeying only $\|D\|_2 \geq 2$ and $\|D\|_{\infty} \leq 2^d$ this is certainly not the case: for example it is not hard to arrange that for uniformly random $x$, $D(x)$ corresponds to some random variable $X \in [0,2^{d}]$ with $\E[X] = 1$ and $\E[X^2] = 4$ and yet 
$$ \E \left[ X \cdot \ind{X \geq 1} \right] \leq  \frac{O(1)}{2^{2d}}.$$
We can consider whether the additive structure of $D = A \star A$ rescues us from this example. 
Upon checking some extreme examples, such as when $A$ a subspace of density $2^{-d}$, or $A$ is a random set of density $2^{-d}$, this looks initially promising; in both cases the function $F = A \star A$ is a reasonably robust witness of $A \star A \gg \|A \star A\|_{2}$. However, the following third example, which we call the ``planted subspace" example, shows this is not the case in general.

\begin{example}[Planted subspace]
Consider a set $A \subseteq \F_2^n$ of the form
$$ A =  (W \cup C) \times \F_{2}^{n-d},$$
where $C \subseteq \F_2^d$ is a random set of size $|C| \leq |\F_2^d|^{1/10}$ and
$W \subseteq \F_2^d$ is a subspace of size $|C|^{3/4}$. \linebreak

It is likely that $|C + C| \approx |C|^2$, $|W + C| \approx |C|^{7/4}$,
and (very roughly speaking) that
$$ R_{W \cup C} \approx R_{W} + R_{C} + R_{W,C} \approx |W| \cdot \1_{W} + \1_{C + C} + \1_{W + C}.$$
Suppose this is the case, and 
consider the random variable 
$$Z := R_{W \cup C}(y),$$ where $y \sim R_{W \cup C}$.\footnote{That is, $y \in \F_2^d$ is a random variable with pdf proportional to $R_{W \cup C}$.} 
We have
$$ \E[Z] = \frac{\sum_{x} R_{W \cup C}(x)^2}{\sum_{x} R_{W \cup C}(x)} \approx \frac{|C|^{9/4} + |C|^2 + |C|^{7/4}}{|C|^{6/4} + |C|^2 + |C|^{7/4}} \approx \frac{|C|^{9/4}}{|C|^2} = |C|^{1/4}, $$ 
and yet 
$$ \prob{Z \leq O(1)} \approx 1 - \frac{|C|^{6/4}}{|C|^2} = 1 - |C|^{-1/2}.$$
\end{example}

This example shows that it is possible that the majority of the contribution to $\|A \star A\|_{2}$ is attributable to some additive structure within $A \star A$ which is strong but rare, and we will need to work somewhat harder in order to detect it.

Let us consider once again whether what we ask for is at least plausible analytically: that we can find a high-entropy density function robustly witnessing $D \gg 1 + \Omega(1)$ whenever $\|D\|_2 \geq 2$ and $\|D\|_{\infty} \leq 2^d$. Again, the best choice of witness $F$ given a constraint on $\|F\|_{\infty}$ would be some super-level set of $D$. However, we consider a simple choice which is already quite satisfactory. We denote by $D^{\wedge k}$ the density function which is proportional to $D^k$; namely
\begin{equation}
D^{\wedge k}(x) := \left( \frac{D(x)}{\|D\|_k} \right)^k.
\end{equation}
In what follows we refer to this density function as the ``degree-$k$ compression" of $D$. 
Since $\|D\|_k \geq \|D\|_1 = 1$, it follows that compressions satisfy the entropy-deficit bound
$$ \|D^{\wedge k}\|_{\infty} \leq \|D\|_{\infty}^k \leq 2^{dk}.$$
Compressions $D^{\wedge k}$ also robustly witness $D \gg (1-\eps)\|D\|_k$.
\begin{prop}[Compressions escape sub-level sets] \label{compressions-escape}
Let $D : \Omega \rightarrow \R_{\geq 0}$ be a density function on some arbitrary finite set $\Omega$, and let $k \geq 1$.
Consider the sub-level set
$$ S := \set{x \in \Omega}{D(x) \leq (1-\eps) \|D\|_k}.$$
We have
$$ \ip{D^{\wedge k}}{\1_{S}} \leq (1-\eps)^{k} \E[\1_S] \leq e^{-\eps k}.
$$
Alternatively, consider 
$$ S' := \set{x \in \Omega}{D(x) \leq c \cdot \|D\|_k^{1 + \frac{1}{k-1}}}.$$
We have
$$ \ip{D^{\wedge k}}{\1_{S'}} \leq c^{k-1}.$$
\end{prop}
\begin{proof}
Bound $D^{k} \cdot \1_{S} \leq (1 - \eps)^k \|D\|_{k}^k \cdot \1_S$ pointwise and then take the expectation.
Alternatively, we may bound $D^{k} \cdot \1_{S'} \leq  c^{k-1} \cdot \|D\|_{k}^k \cdot D$ pointwise and use $\E[D] = 1$.
\end{proof}

We apply this to our current situation for some choice of $k \geq 2$, noting that $\|D\|_{k} \geq \|D\|_{2} \geq 2.$ We find that it suffices to choose some $k \leq O(1)$ to get a (high-entropy) robust witness to $D \gg 1 + \Omega(1)$, as desired.

Importantly for us, for integral values of $k$, degree-$k$ compressions of a convolution $D = A \star A$ also retain a certain amount of additive structure. More specifically, we will soon see that the fact that
$$ \ip{(A \star A)^{\wedge k}}{\1_S} \leq \eta, $$
entails the following consequence: that
there is some subset $A' \subseteq A$, of size $|A'| \geq 2^{-O(dk)} |A|$,
such that 
$$ \ip{A' \star A'}{\1_S} \leq O(\eta).$$
More specifically, $A'$ is of the form
$$ A' = A \cap (A + s_1) \cap (A + s_2) \cap \cdots \cap (A + s_{k-1})$$
for some choice of additive shifts $s_i$. 
Thus we can obtain a robust witness to $A \star A \gg 1 + \Omega(1)$, which has the form $A' \star A'$, as desired. 

We offer the following interpretation of this. Taking multiple additive perturbations of the set $A$ and considering their common intersection is sufficient to ``uncover" or ``reveal" the additive structure hidden within $A$. This technique has been used in some form by various works in the existing literature; we call it ``sifting", and in this work we develop some refinements to it.

\subsection{The Pre-BSG Lemma and the sifting lemma} \label{the-pre-BSG-lemma-and-sifting}

Let us consider the following lemma due to Schoen \cite{schoen15}, which we state in the density formulation.

\begin{lem}[Schoen's Pre-BSG Lemma]
Let $A$ be subset of a finite abelian group $G$ of size $|A| \geq \delta |G|$. 
Consider the sub-level set 
$$ S := \set{x \in G}{(A \star A)(x) \leq c \cdot \|A \star A\|_{2}^2}. $$
There is a subset $A' \subseteq A$ with
$$\ip{A' \star A'}{\1_S} \leq 16 c$$
and
$$ \frac{|A'|}{|G|} \geq \tfrac{1}{3} \delta^2 \cdot \|A \star A\|_2^2 \geq \tfrac{1}{3} \delta^2.$$
Specifically, $A'$ is of the form
$$ A' = A \cap (A + s)$$
for some $s \in G$.
\end{lem}

We call this the ``Pre-BSG Lemma" as it can be considered as a sort of soft, analytic precursor to the Balog-Szemer{\'e}di-Gowers lemma -- indeed, Schoen shows that it readily implies the following.

\begin{lem}[Balog-Szemer{\'e}di-Gowers \cite{schoen15}]
Suppose that $A \subseteq G$ has 
$$E(A) = \sum_{x \in G} R_{A}(x)^2 = \kappa |A|^3.$$
Then there is a subset $B \subseteq A$ with
$$ |B - B| \leq O(\kappa^{-4} |B|) $$
and
$$ |B| \geq \Omega(\kappa |A|). $$
\end{lem}

The name ``Pre-BSG Lemma" is certainly ahistorical, however in a sense it is fairly appropriate in spirit: the technique used to prove it (sifting) is also the technique Gowers uses to prove his own combinatorial precursor-lemma (\cite[Lemma 7.4]{gowers01}) from which he deduces (a form of) the BSG Lemma.

We remark on a peculiar feature of the Pre-BSG Lemma. 
It is a tool that is useful for studying both sparse and dense subsets\footnote{Certainly it seems to have seen more applications to the study of sparse sets (via the BSG Lemma). However, see e.g.\ Sanders' work \cite{sanders10} where he applies a (chronologically earlier) variant to the study of dense sets.
} of a finite abelian group $G$, which we roughly delineate as the regimes $|A| \leq |G|^{1/2}$ and $|A| \geq |G|^{1/2}$.
This is in contrast to e.g.\ other tools such as Chang's inequality, which says nothing interesting regarding sparse sets, or \textit{the BSG Lemma itself} which says nothing interesting regarding dense sets except in some very extreme cases.\footnote{
Indeed, for a set of size $|A| = \delta |G|$, consider the ratio $\kappa = E(A)/|A|^3$.
We have $\kappa = \delta  \|A \star A\|_{2}^{2} \in [\delta,1].$
In our present context, we consider $A$ to have noticeable additive structure already when $\|A \star A\|_{2} \geq 1 + \Omega(1)$; in contrast the BSG Lemma above is not better than the trivial bound $|A - A| \leq |G| \leq \delta^{-1} |A|$ unless $\|A \star A\|_2 \geq \delta^{-3/8}$.
}
While the density-formulation of the Pre-BSG Lemma stated above is technically equivalent\footnote{That is, equivalent in the case of finite groups $G$.} to the
counting-measure formulation stated in \cite{schoen15}, translating between the two settings is a nontrivial exercise in bookkeeping and for studying sparse sets $A \subseteq G$ one would much prefer to work from Schoen's original formulation. It is possibly better for practical purposes even to consider the tool as ``conceptually different" as it applies to the two settings.

One can check that Schoen's Pre-BSG Lemma, together with the plan laid out above, is sufficient to obtain a density increment
$$\ip{V}{A} \geq 1 + \Omega(1)$$
from an assumption such as $$\|A \star A\|_{2}^2 \geq 2^{7}.$$
To obtain a density increment from milder assumptions of the form $\|A \star A\|_{k} \geq 1 + \eps$, we turn to the following.

\begin{lem}[Sifting lemma -- simplified version]
Suppose $A \subseteq G$ has size $|A| = \delta |G|$.
Fix a nonnegative function $f : G \rightarrow \R_{\geq 0}$ and an integer $k \geq 2$.
Suppose that
$$ \ip{(A \star A)^{\wedge k}}{f} = \eta.$$
Then there is a subset $A' \subseteq A$ with
$$ \ip{A' \star A'}{f} \leq 2 \eta $$ 
and
$$ \frac{|A'|}{|G|} \geq \tfrac{1}{2} \delta^{k}. $$
Specifically, $A'$ is of the form
$$ A' = A \cap (A + s_1) \cap (A + s_2) \cap \cdots \cap (A + s_{k-1}) $$
for some shifts $s_i \in G$.
\end{lem}

Combining this with the plan laid out above suffices to prove \cref{II}; it remains only to prove the sifting lemma and to work out the quantitative details. We note that as stated this claim is incomparable with Schoen's Pre-BSG Lemma -- indeed, it says nothing interesting regarding sparse sets $A \subseteq G$ because of the size guarantee on $A'$ -- but it does suffice for the applications considered in this work. With some extra effort one can derive a version of the sifting lemma which leads to the following strict improvement which may be useful elsewhere; we state it below in the counting-measure formulation.

\begin{lem}[Extended Pre-BSG Lemma] \label{extended-pre-bsg}
Suppose $A$ is a finite subset of an abelian group $G$ with
$$E_{k}(A) = \sum_{x \in G} R_{A}^{-}(x)^k = \kappa |A|^{k+1}$$
for some integer $k \geq 2$. Consider the sub-level set
$$S := \set{x \in G}{R_{A}^-(x) \leq c \cdot \kappa^{\frac{1}{k-1}} \cdot |A|}. $$
There is a subset $A' \subseteq A$ with
$$ \frac{1}{|A'|^2} \sum_{a,b \in A'} \1_{S} (a - b) \leq 2 c^{k-1} $$
and
$$ |A'| \geq  \tfrac{1}{2} \kappa |A|. $$
\end{lem}

\section{Sanders' invariance lemma}

\subsection{The Croot-Sisask lemma}

For a point $p \in G$, define the shift operator $T_p$ by 
$$ (T_p f)(x) := f(x - p) .$$
This operation can be equivalently described as convolution of $f$ with the point-mass density at $p$ (and hence it commutes with other convolutions). For a density $B$, we define the linear operator $T_B$ analogously: $(T_B f)(x) := (B * f)(x)$. 

\begin{lem}[Croot-Sisask lemma] \label{croot-sisask}
Let $G$ be a finite abelian group, and
suppose $A,A' \subseteq G$ are such that
$$ |A| \geq 2^{-d} |A + A'|, $$
which is in particular the case if simply $A' = G$ and  $|A| \geq 2^{-d} |G|$. 
Fix an even integer $k \geq 2$. and
a function $f : G \rightarrow \R_{\geq 0}$ such that $\|f\|_k \leq 1$ (which is in particular the case if $f$ maps into $[0,1]$).
For any $\eps \in [0,1]$, there exists a set $S \subseteq A'$ of size at least
$$ |S| \geq 2^{-O(k d/ \eps^2)} |A'| $$
such that for any points $p,p' \in S$,
$$ \norm{T_{p - p'} A * f - A * f}_k \leq \eps. $$
Consequently, for any integer $t$ and any point $p \in tS - tS$,
$$ \norm{T_{p} A * f - A * f}_k \leq t \eps. $$
\end{lem}

\begin{proof}
Fix an integer $\ell \in \N$ which is somewhat larger than $k$.
We consider the idea of approximating $A * f$ by a random ``sketch"
$$ s(a_1, a_2, \ldots, a_\ell) :=  \frac{1}{\ell}\sum_{i=1}^\ell T_{a_i} f$$
If each $a_i$ is drawn independently from $A$, then $\E[ T_{a_i} f] = A * f$, and the vector-valued Khintchine inequality (proved below for even $k$) gives
$$ \E_{a} \| s(a) - A * f \|_k^k = \E_{a} \norm{ \frac{1}{\ell} \sum_{i=1}^\ell (T_{a_i} - T_A) f }_k^k \leq \left(\frac{k}{\ell}\right)^{k/2} \cdot \|f\|_k^k,
$$
So by a Markov inequality, the chance that $\| s(a) - A * f \|_k \geq 2  \sqrt{\frac{k}{\ell}}$ is at most $2^{-k}$. 

Now, we seek a simultaneous approximation of $T_p A * f = T_{A + p} f$ for all points $p \in A'$ by some sketches 
$$ s(y) = s(y_1, y_2 ,\ldots y_\ell) := \frac{1}{\ell}\sum_{i=1}^\ell T_{y_i} f,$$
this time allowing $y \in (A+A')^{\ell}$ instead of just $a \in A^\ell$. 
We say that $y$ is a \textit{plausible} sketch for $T_{A + p} f$ if $y_i \in (A + p)$ for all $i$. 
We say that $y$ is a \textit{good} sketch for $T_{A + p} f$ if $$\| s(y) - T_{A + p}  f \|_k \leq 2  \sqrt{\frac{k}{\ell}}.$$
By the same argument as above, for every $p$, at least a fraction $(1-2^{-k})$ of the sketches which are plausible for $T_{A + p} f$ are in fact good for $T_{A + p} f$. 

Now we form a bipartite graph on the vertex-set $A' \times (A \times A')^\ell$, including the edge $(p,y)$ whenever $s(y)$ is a good sketch for $T_{A + p} f$. We wish to find a sketch $s(y)$ on the right side which is simultaneously good for as many points $p$ as possible, and we will do so by a pigeonhole argument:
\begin{itemize}
    \item Vertices on the left side have degree at least $\tfrac{1}{2} |A|^\ell$.
    \item We have the following relationship between the average degrees of both sides ($D_L, D_R$), the sizes of the vertex-sets of both sides ($N_L, N_R$), and the total number of edges $E$:
    $$ D_L \cdot N_L = E = D_R \cdot N_R .$$
    \item Thus, $D_R \geq \frac{D_L}{N_R} \cdot N_L \geq \frac{1}{2} \left( \frac{|A|}{|A + A'|} \right)^\ell |A'|$.
\end{itemize}
    So, we can find a fixed sketch $s(y)$ on the right adjacent to a large set of points $S \subseteq A'$ on the left, meaning that $s(y)$ is a good sketch for every $p \in S$. Since $T_{A + p} f$ are all close to $s(y)$, they are all close to each other:
    $$ \norm{T_p A * f - T_{p'} A *f}_k =  \norm{T_{A + p} f - T_{A + p' f}}_k \leq \norm{T_{A + p} f - s(y)}_k + \norm{s(y) - T_{A + p'}f}_k   \leq 4 \sqrt{\frac{k}{\ell}}
    $$
    for all $p,p' \in S$. To conclude, note that since $\| T_x g \|_k = \| g \|_k$ for any function $g$ and any point $x$, we have
    $$ \norm{T_p A * f - T_{p'} A *f}_k = \norm{T_{-p'} (T_p A * f - T_{p'} A *f)}_k 
    = \norm{T_{p-p'} A * f -  A * f}_k .
    $$
    Setting $\ell = 16 \cdot k / \eps^2$ gives the desired parameters. \qedhere
 
\end{proof}

\begin{prop}[Vector-valued Khintchine inequality]
    Let $k \geq 2$ be an even integer.
    Let $v_1, v_2, \ldots, v_\ell \in \R_{\geq 0}^m$ be some independent vector-valued random variables with means $\mu_i := \E[v_i]$.
    Suppose that 
    $$\E \|v_i\|_k^k :=  \E \left[ \frac{1}{m}\sum_{j \in [m]} v_i(j)^k \right] \leq M$$
    for all $i$.
    Consider the average
    $$v := \frac{1}{\ell} \sum_{i=1}^{\ell} \left(v_i - \mu_i \right) $$
    of the (centered versions of) the $v_i$'s. 
    We have the following bound on the average value of $\|v\|_k^k$:
    $$ \E\left[ \|v\|_k^k\right] \leq \left( \frac{k}{\ell}\right)^{k/2}  M.
    $$
\end{prop}
\begin{proof}
By re-scaling we may assume $M = 1$. We may also assume $\mu_i = 0$ for all $i$, since $ \E \|v_i - \mu_i\|_k^k \leq \E \|v_i\|_k^k$ for $k \geq 2$. 
Define the $k$-wise dot product of some $k$ vectors $u_i \in \R^m$ by
$$ u_1 \cdot u_2 \cdots u_k := \E_{j \in [m]} u_1(j) u_2(j) \cdots u_k(j) .$$
Now, we compute
$$ \| v \|_k^k = \E_{(i_1,i_2,\ldots,i_k) \in [\ell]^k} v_{i_1} \cdot v_{i_2} \cdots v_{i_k} $$
Upon taking the expectation of this quantity, any term which contains a certain random variable $v_i$ only once in the dot product becomes zero. For the remaining terms, we can bound
$$ \E[ v_{i_1} \cdot v_{i_2} \cdots v_{i_k} ]  \leq \prod_{j=1}^k \left( \E \|v_{i_j}\|_k^k \right)^{1/k} \leq 1.
$$
So we reduce to the combinatorial problem of counting the number of tuples ${\bf i} \in [\ell]^k$ which have no unique entries. Let $T(k,\ell)$ denote this quantity. We can bound this number recursively as follows: Consider the first entry $i_1$ of such a tuple. There must be some $i_j$ with $i_1 = i_j$; we consider all the possibilities for this, then remove the first and $j$-th entry from the tuple, and then count the number of tuples in $[\ell]^{k-2}$ which have no unique entries. This argument gives the bound
$$ T(k,\ell) \leq k \cdot \ell \cdot T(k-2, \ell) \leq \cdots \leq  \left(k \ell\right)^{k/2} .$$
So overall we have
\begin{equation*}
    \E \left[ \| v \|_k^k \right] \leq \frac{\left(k \ell\right)^{k/2} }{\ell^k} = \left( \frac{k}{\ell}\right)^{k/2}. \qedhere
\end{equation*} 
\end{proof}

\subsection{Sanders' invariance lemma}

\begin{lem}[Sanders' invariance lemma, restated] \label{sanders-invar-appendix}
Suppose $A,B \subseteq \F_q^n$ are sets of sizes $|A| \geq 2^{-d} |\F_q^n|$ and $|B| \geq 2^{-k} |\F_q^n|$.
Fix a bounded function $f : \F_q^n \rightarrow [0,1]$.
Then, for any $\eps \geq 2^{-d}$, there exists a linear subspace $V$ (possibly depending on $f$) with codimension at most $O(k d^3 / \eps^2)$ satisfying
$$ \left| \ip{V * A * B}{f} - \ip{A * B}{f} \right| \leq \eps. $$
More specifically, we have the pointwise bound
$$ \left| \E_{\substack{a \in A \\ b \in B}} f(x + v + a + b) - \E_{\substack{a \in A \\ b \in B}} f(x + a + b) \right|\leq \eps$$
for all $v \in V$ and $x \in \F_q^n$.
\end{lem}

\begin{proof}
We seek a large linear subspace $V$ such that 
$$ \| T_v A \star f - A \star f\|_k \leq \frac{\eps}{2} $$
for all $v \in V$.
 We note that this suffices, since
$$  \left| \ip{T_x B}{T_v A \star f - A \star f} \right| \leq \norm{T_x B}_{1}^{1-1/k} \norm{T_x B}_{\infty}^{1/k}  \| T_v A \star f - A \star f\|_k \leq 2 \| T_v A \star f - A \star f\|_k 
$$
for any $x \in \F_q^n$. 
First we apply the Croot-Sisask lemma to obtain a large set $S$ of shifts $T_p$ which leave $A \star f$  approximately unchanged: $\| T_{p-p'} A \star f - A \star f\|_k \leq \eta$ for all $p,p' \in S$ where $\eta$ is some parameter we will pick later. We can obtain
$$ |S| \geq 2^{-O(kd/\eta^2)} |\F_q^n| .$$

Continuing on, let's write $g := A \star f$ for brevity. 
We now describe how to pick the linear subspace $V$. 
First, consider the large spectrum of $S$:
$$ \text{Spec}_{1/2}(S) := \set{\alpha \in \F_q^n}{|\widehat{S}(\alpha)| \geq \tfrac{1}{2}}. $$
We let $W := \text{Span}(\text{Spec}_{1/2}(S))$. 
By Chang's inequality for vector spaces over finite fields (proved below), the dimension of $W$ is at most $O(kd/ \eta^2)$. 
Finally, we let $V := W^\perp$. 

Consider the density function $D = S \star S$, and
let $X$ denote the $t$-iterated convolution $D * D * \cdots * D$. We have
$$ X(x) = \sum_{\alpha} |\widehat{S}(\alpha)|^{2t} e_\alpha(-x) ,$$
and in particular for any $\alpha \not\in W$, we have  $|\widehat{X}(\alpha)| \leq 2^{-2t}$.
We note that
$$ (V * X * g - X * g)(x) =  - \sum_{\alpha \not\in W} \widehat{X}(\alpha) \widehat{A}(\alpha)  \widehat{f}(\alpha)  e_{\alpha}(-x).$$
Thus 
$$ \| V * X * g - X * g \|_{\infty} \leq 2^{-2t} \sqrt{\sum_{\alpha} |\widehat{A}(\alpha)|^2} \sqrt{ \sum_{\alpha} |\widehat{f}(\alpha)|^2} \leq 2^{-2t} \cdot 2^{d} .
$$
Finally, we estimate 
\begin{align*}
 \norm{g - V * g}_k &\leq \norm{g - X * g}_k + \norm{X * g - V * X * g}_k + \norm{V * X * g - V * g}_k \\
 &= \norm{g - X * g}_k + \norm{X * g - V * X * g}_k + \norm{V * (X * g - g)}_k \\
 &\leq 2 \norm{g - X * g}_k + \norm{X * g - V * X * g}_\infty \\
 &\leq 2t\eta + 2^{d - 2t},
\end{align*}
and
$$ \norm{T_v g - V * g}_k = \norm{g - T_{-v} V * g}_k = \norm{g - V * g}_k \leq 2t\eta +   2^{d - 2t},$$
which together give a bound on $\norm{T_v g - g}_k$.
Choosing $t := O(d)$ and $1/\eta := O(t/\eps)$, we obtain the desired error bound.
The resulting bound on the codimension of $V$ is $O(kdt^2/\eps^2) = O(k d^3 /\eps^2)$. \qedhere
\end{proof}

\begin{lem}[Chang's inequality for vector spaces over $\F_q$]
Suppose $A \subseteq \F_q^n$ has size $|A| \geq 2^{-d} |\F_q^n|$, where $d \geq 1$.
Then any subset of linearly independent vectors $L \subseteq \text{Spec}_{\eps}(S)$ has size at most
$$|L| \leq  4 d / \eps^2 .$$
\end{lem}
\begin{proof}
Suppose $\alpha_1 ,\ldots, \alpha_{\ell} \subseteq \text{Spec}_{\eps}(S)$ are linearly independent. 
Let $c_i := \widehat{A}(\alpha_i)$, and consider the  auxiliary function
$$ f(x) :=  \sum_{i=1}^{\ell} c_i \cdot e_{\alpha_i}(-x). $$
On one hand we have
$$ \ip{A}{f} = \sum_{i=1}^{\ell} c_i  \E_{x \in A} e_{\alpha_i}(-x) =   \sum_{i=1}^{\ell} |\widehat{A}(\alpha_i)|^2.
$$
On the other hand (choosing $k = d$) we can bound
$$\ip{A}{f} \leq \|A\|_{1}^{1-1/k} \|A\|_{\infty}^{1/k} \|f\|_{k} = 2^{d/k} \cdot \|f\|_{k} \leq 2 \cdot \|f\|_{k}.
$$
Now, since the vectors $\alpha_i \in \F_q^n$ are linearly independent, the complex random variables defined by $X_i := e_{\alpha_i}(x)$ (where $x$ is uniformly random in $\F_q^n$) are in fact mutually independent. 
So by a Khintchine inequality, $\|f\|_k \leq \sqrt{k} \cdot \|f\|_2  = \sqrt{k} \cdot \sqrt{ \sum_{i} |\widehat{A}(\alpha_i)|^2}.$ 
We conclude that
$$ \eps^2 \cdot \ell \leq \sum_{i=1}^{\ell} |\widehat{A}(\alpha_i)|^2 \leq 4 k = 4 d,
$$
and so $\ell \leq 4d/\eps^2$. \qedhere
\end{proof}

\section{Odd moments}\label{sec:odd-moments}

\begin{prop}[\cref{odd-moments}, restated] \label{append-odd-moments}
For a real-valued random variable $X$ and $k \geq 1$, we use the notation
$$ \|X\|_{k} := \left( \E |X|^k \right)^{1/k}. $$
Suppose $X$ is such that 
\begin{itemize}
    \item $\E (X - 1)^{k} \geq 0$ for all odd $k \in \N$, and
    \item $\|X - 1\|_{k_0} \geq \eps$ for some even $k_0 \geq 2$ and some $\eps \in [0,1]$.
\end{itemize}
Then, for any integer $k' \geq 2 k_0/\eps$,
$$ \|  X  \|_{k'} \geq 4^{-1/k'} (1 + \eps).  $$
Furthermore, if $\eps \leq 1/2$, then
$$ \|  X  \|_{k'} \geq 4^{-1/k'} (1 + \eps) \geq 1 + \frac{\eps}{2}.  $$

\end{prop}

\begin{proof}[Proof of \cref{append-odd-moments}]
We wish to show that $\E X^{k'} \geq \frac{1}{4} (1+\eps)^{k'}.$
We express
\begin{align*}
    \E X^{k'} &= \E (1 + (X-1))^{k'} \\
    &= \sum_{k=0}^{k'} \binom{k'}{k} \cdot \E (X - 1)^{k} \\
    &\geq \sum_{\substack{k \textnormal{ even}\\ k_0 \leq k \leq k'}} \binom{k'}{k} \cdot \eps^{k}.
 \end{align*}
Towards understanding this quantity we apply the following trick. We introduce a uniform random variable $\theta \in \{+1,-1\}$ and consider the value $\E [(1 + \theta \eps)^{k'}]$.
On one hand this value is of course
$$ \E_\theta  (1 + \theta \eps)^{k'} = \tfrac{1}{2}(1 + \eps)^{k'} + \tfrac{1}{2} (1 - \eps)^{k'}. $$
It can be expressed also as
\begin{align*}
    \E_\theta  (1 + \theta \eps)^{k'} &= \sum_{k=0}^{k'} \binom{k'}{k} \E [\theta^{k}] \eps^{k} = \sum_{\substack{k \textnormal{ even}\\ 0 \leq k \leq k'}}  \binom{k'}{k}  \eps^{k} ;
\end{align*}
we conclude that
$$ \sum_{\substack{k \textnormal{ even}\\ 0 \leq k \leq k'}}  \binom{k'}{k}  \eps^{k} \geq \tfrac{1}{2}(1 + \eps)^{k'}.$$

At this point, let us normalize by $(1 + \eps)^{k'}$ so that we can proceed via a probabilistic interpretation: we are interested in the probability of a certain event concerning a random variable $k \in \N$ drawn according to a binomial distribution with $k'$ trials and success-probability
$p = \frac{\eps}{1+\eps}$: 
$$ k \sim \textnormal{Bin}\left(k', \frac{\eps}{1+\eps} \right). $$
In particular, we are interested in the chance that both
\begin{enumerate}
    \item $k$ is even, and
    \item $k \geq k_0$. 
\end{enumerate}
So far, we have calculated that the chance that $k$ is even is at least a half. 
In addition, it is known that a median of a binomial distribution $\textnormal{Bin}(k',p)$ is at least $\lfloor p k'\rfloor$ \cite{kaas-buhrman}.
Since $k_0 \leq \eps k' / 2 \leq \eps k' / (1 + \eps)$, we also have that the chance that $k \geq k_0$ is at least a half.
One might then suspect that the probability that both events occur simultaneously should be roughly at least $1/4$; the following calculation confirms this. Suppose $t \leq \eps k'/2 $ is an odd natural number. We have
\begin{align*}
    \prob{k = t} &= \frac{k'-(t-1)}{t} \cdot \eps \cdot \prob{k = t - 1} \\
    &\geq \frac{k'-t}{t} \cdot \eps \cdot \prob{k = t - 1} \\
    &\geq \left( \frac{2}{\eps} - 1 \right) \cdot \eps \cdot \prob{k = t - 1} \\
    &\geq \prob{k = t - 1},
\end{align*}
and so
$$ \prob{k \textnormal{ is even and } k < k_0} \leq \tfrac{1}{2} \cdot \prob{k < k_0 } \leq \tfrac{1}{4}.
$$
We infer that
$$ \prob{k \textnormal{ is even and } k \geq k_0} \geq \tfrac{1}{2} - \tfrac{1}{4} = \tfrac{1}{4} $$
and so
$$ \E X^{k'} \geq \tfrac{1}{4} \cdot \left(1 + \eps\right)^{k'}. $$
Equivalently,
$$ \left( \E X^{k'} \right)^{1/k'} \geq 4^{-1/k'} \cdot (1 + \eps) .
$$
To conclude, one can check that
$$ 4^{-\eps / 4} \cdot (1 + \eps) \geq 1 + \frac{\eps}{2} $$
for $\eps \in [0,\frac{1}{2}]$. \qedhere
\end{proof}

\end{document}